\documentclass[a4paper,english]{amsart}
\usepackage[latin9]{inputenc}
\usepackage{color,hyperref}
\usepackage{amstext}
\usepackage{amsthm}
\usepackage{amssymb}
\PassOptionsToPackage{normalem}{ulem}
\usepackage{ulem}

\makeatletter

\numberwithin{equation}{section}
\numberwithin{figure}{section}
\theoremstyle{plain}
\newtheorem{thm}{\protect\theoremname}
  \theoremstyle{definition}
  \newtheorem{defn}[thm]{\protect\definitionname}
  \theoremstyle{remark}
  \newtheorem{rem}[thm]{\protect\remarkname}
  \theoremstyle{plain}
  \newtheorem{lem}[thm]{\protect\lemmaname}
  \theoremstyle{plain}
  \newtheorem{cor}[thm]{\protect\corollaryname}
  \theoremstyle{plain}
  \newtheorem{prop}[thm]{\protect\propositionname}

\makeatother

\usepackage{babel}
  \providecommand{\corollaryname}{Corollary}
  \providecommand{\definitionname}{Definition}
  \providecommand{\lemmaname}{Lemma}
  \providecommand{\propositionname}{Proposition}
  \providecommand{\remarkname}{Remark}
\providecommand{\theoremname}{Theorem}

\begin{document}

\title[Spectral Invariance on Conic Manifolds with Boundary]{Spectral Invariance of Pseudodifferential Boundary Value Problems on Manifolds
with Conical Singularities}

\author{Pedro T. P. Lopes}
\address{Instituto de Matem\'atica e Estat\'istica, Universidade de S\~ao Paulo,  Rua do Mat\~ao 1010, 05508-090, S\~ao Paulo, SP, Brazil}
\email{pplopes@ime.usp.br}
\thanks{Pedro T. P. Lopes was partially supported by FAPESP (Processo n\'umero 2016/07016-8)}

\author{Elmar Schrohe}
\address{Institut für Analysis, Leibniz Universität Hannover, Welfengarten 1,
30167 Hannover, Germany.}
\email{schrohe@math.uni-hannover.de}


\date{\today}

\begin{abstract}
We prove the spectral invariance of the algebra of classical pseudodifferential
boundary value problems on manifolds with conical singularities in the $L_{p}$-setting.
As a consequence we also obtain the spectral invariance of the classical
Boutet de Monvel algebra of zero order operators with parameters.
In order to establish these results, we show the equivalence of Fredholm
property and ellipticity for both cases. 
\end{abstract}

\maketitle

\section{Introduction}

Elliptic boundary value problems on manifolds with conical singularities
have been studied since the 60's, where the work of V. A. Kondratiev
\cite{Kondratiev} stands out, see also Kozlov, Maz'ya and Rossmann
\cite{KozlovMazyaRossmonn} for a detailed presentation. The
pseudodifferential analysis started with  the work
of R. Melrose and G. Mendoza \cite{MelroseMendoza,Melrose}, B. Plamenevsky \cite{Plamenevskij}, and B. -W.
Schulze \cite{Schulze}. Algebras of pseudodifferential boundary
value problems for conical singularities were constructed in the 90's
by A. O. Derviz \cite{Derviz} and E. Schrohe and B.-W. Schulze \cite{SchroheSchulzeConicalBoundaryI,SchroheSchulzeConicalBoundaryII}.
The latter approach combines elements of the Boutet de Monvel calculus \cite{boutetMonvel}
with the pseudodifferential analysis developed
by B. -W. Schulze \cite{SchulzeNorthHolland,Schulze}.
While initially only $L_{2}$-based Sobolev spaces were used, 
S. Coriasco, E. Schrohe and J. Seiler established the continuity
also on Bessel potential  and Besov spaces \cite{CoriascoSchroheSeilerRealizations},
see also \cite{CoriascoSchroheSeilerBoundedHinfty}, relying on work of 
G. Grubb and N. Kokholm \cite{Grubblp,Grubblpparameters}.

Our main result is  the spectral invariance
of the algebra developed  
in \cite{SchroheSchulzeConicalBoundaryII} in the $L_{p}$-setting,
see Theorem \ref{thm:Teorema Principal}. 
This algebra contains, after the composition with order
reducing operators, the classical differential boundary value problems
studied by V. A. Kondratiev \cite{Kondratiev},
hence also their inverses, whenever these exist.
As a by-product we
obtain the spectral invariance of the algebra of zero order classical Boutet
de Monvel operators with parameters in the $L_{p}$-setting, see  Theorem \ref{thm:Teoremabmcomparametros}.
This algebra includes, after composition with order reducing operators,
the differential boundary value problems studied by M. S. Agranovich
and M. I. Vishik in \cite{AgranovichVishik}, which were an important
ingredient for the work of Kondratiev.

It is an immediate consequence of Theorem \ref{thm:Teorema Principal} that the invertibility 
of a conically degenerate boundary value problem is to a large extent independent 
of the space it is considered  on: 
It depends neither on the Sobolev regularity parameter $s$ nor on $1<p<\infty$. 
This is of great practical importance as it allows to check invertibility in the most convenient setting. A similar result holds for the Fredholm property, as we show in 
Corollary \ref{equivalence}. 

This article extends the results of \cite{SchroheSeilerSpectConical} 
to conical manifolds with boundary.  
The need to work with Besov spaces led to interesting new features. In Theorem  \ref{thm:Teoremabmcomparametros}, for example, we consider a zero order parameter-dependent operator $A=\{A(\lambda); \lambda\in \Lambda\}$ in Boutet de Monvel's calculus. 
We show that the invertibility of $A(\lambda)$ for each $\lambda$ together with a 
norm estimate $\|A(\lambda)^{-1}\|\le c\langle \lambda\rangle^r$ 
for a constant  $c\ge 0$ and sufficiently small $r>0$ implies that the inverse also is 
parameter-dependent of order zero. 
In particular, the operator norm will then be uniformly bounded. 
Similar effects can be observed when showing the equivalence of parameter-ellipticity and the Fredholm property with parameters. 

This paper is a step toward the analysis of nonlinear partial differential equations on 
manifolds with boundary and conical singularities, see e.g. 
\cite{RoidosSchrohePorous} by Roidos and Schrohe,  \cite{ShaoSimonett2014} by Shao and Simonett or \cite{Vertman2016} by Vertman for the case without boundary. 
A next step concerns the analysis of resolvents of closed extensions in the spirit 
of Gil, Krainer and Mendoza \cite{GilKrainerMendoza} or \cite{SchroheSeilerResolventConeDifferential} in the case without boundary and
Krainer \cite{Krainer} for conic manifolds with boundary.

\section{Parameter-dependent Boutet de Monvel algebra\label{sec:Explicando-Parameter-dependent-BdM}}

To make this article readable for non-experts, we briefly describe the parameter-dependent Boutet
de Monvel algebra with classical symbols on compact manifolds with
boundary in the $L_{p}$-setting. We first define several operator  classes on 
the half-space $\mathbb{R}_{+}^{n}=\left\{ x\in\mathbb{R}^{n};\,x_{n}>0\right\} $.

The set of parameters of the operators and symbols will always be a conical open set $\Lambda\subset\mathbb{R}^{l}$, that is, $p\in\Lambda$ implies that $tp\in\Lambda$
for  $t>0$. It can be the empty set, in which case we recover
the usual symbols and operators. We write $\mathbb{N}_{0}:=\left\{ 0,1,2,...\right\}$
and $\mathbb{R}_{++}=\mathbb{R}_{+}\times\mathbb{R}_{+}$.
For a Fréchet space $W$, the Schwartz space $\mathcal{S}\left(\mathbb{R}^{n},W\right)$
consists of all $u\in C^{\infty}\left(\mathbb{R}^{n},W\right)$
such that $\sup_{x\in\mathbb{R}^{n}}p\left(x^{\alpha}\partial_{x}^{\beta}u\left(x\right)\right)<\infty$
for every continuous seminorm $p$ of $W$. We simply write 
$\mathcal{S}\left(\mathbb{R}^{n}\right)$, if $W=\mathbb{C}$.
If $\Omega\subset\mathbb{R}^{n}$
is an open set, $\mathcal{S}(\Omega)$ denotes the restrictions
to $\Omega$ of functions in $\mathcal{S}\left(\mathbb{R}^{n}\right)$, and 
$C_{c}^{\infty}(\Omega)$ the space of smooth functions with compact support in 
$\Omega$.
The operator of restriction
of distributions defined in $\mathbb{R}^{n}$ to $\Omega$ is denoted
by $r_{\Omega}:\mathcal{D}'\left(\mathbb{R}^{n}\right)\to\mathcal{D}'\left(\Omega\right)$.
The extension by zero of a function $u$ defined in $\Omega$
to $\mathbb{R}^{n}$ will be denoted by $e_{\Omega}$:
\[
e_{\Omega}\left(u\right)\left(x\right)=\left\{ \begin{array}{c}
u\left(x\right),\,x\in\Omega\\
0,\,x\notin\Omega
\end{array}\right..
\]
If $\Omega=\mathbb{R}_{+}^{n}$, we denote $r_{\mathbb R^n_+}$ also by $r^+$ 
and $e_{\mathbb{R}_{+}^{n}}$ by $e^{+}$. 
The open
ball in $\mathbb{R}^{n}$ with the Euclidean norm whose center is
$x$ and radius is $r>0$ will be denoted by $B_{r}\left(x\right)$. Our convention for the Fourier transform
is $\mathcal{F}u(\xi)=\hat{u}\left(\xi\right)=\int e^{-ix\xi}u(x)dx$.
We shall often use the function $\left\langle .\right\rangle :\mathbb{R}^{n}\to\mathbb{R}$
defined by
\[
\left\langle \xi\right\rangle :=\sqrt{1+\left|\xi\right|^{2}}
\]
and sometimes we use $\left\langle \xi,\lambda\right\rangle :=\sqrt{1+\left|\left(\xi,\lambda\right)\right|^{2}}$
and similar expressions, as well.

Finally, given two Banach spaces $E$ an $F$, we denote by $\mathcal{B}\left(E,F\right)$
the bounded operators from $E$ to $F$ and use the notation
$\mathcal{B}\left(E\right):=\mathcal{B}\left(E,E\right)$.
\begin{defn}
\label{def:(Pseudodifferential-Symbols)}The space $S^{m}\left(\mathbb{R}^{n}\times\mathbb{R}^{n},\Lambda\right)$
of parameter-dependent symbols of order $m\in\mathbb{R}$ consists
of all functions $p\in C^{\infty}\left(\mathbb{R}^{n}\times\mathbb{R}^{n}\times\Lambda\right)$
that satisfy 
\[
\left|\partial_{x}^{\beta}\partial_{\xi}^{\alpha}\partial_{\lambda}^{\gamma}p\left(x,\xi,\lambda\right)\right|\le C_{\alpha\beta\gamma}\left\langle \xi,\lambda\right\rangle ^{m-\left|\alpha\right|-\left|\gamma\right|},\,\,\,\left(x,\xi,\lambda\right)\in\mathbb{R}^{n}\times\mathbb{R}^{n}\times\Lambda.
\]
A symbol $p$ defines a parameter-dependent pseudodifferential operator
$op\left(p\right)\left(\lambda\right):\mathcal{S}\left(\mathbb{R}^{n}\right)\to\mathcal{S}\left(\mathbb{R}^{n}\right)$
by the formula:
\[
op\left(p\right)\left(\lambda\right)u\left(x\right)=(2\pi)^{-n}\int e^{ix\xi}p\left(x,\xi,\lambda\right)\hat{u}\left(\xi\right)d\xi.
\]
We say that $p$ is classical,  if there are symbols $p_{\left(m-j\right)}\in S^{m-j}\left(\mathbb{R}^{n}\times\mathbb{R}^{n},\Lambda\right)$,
$j\in\mathbb{N}_{0}$, such that
\begin{enumerate}
\item For all $t\ge1$ and $\left|\left(\xi,\lambda\right)\right|\ge1$,
we have $p_{\left(m-j\right)}\left(x,t\xi,t\lambda\right)=t^{m-j}p_{\left(m-j\right)}\left(x,\xi,\lambda\right)$.
\item We have the asymptotic expansion  $p\sim\sum_{j=0}^{\infty}p_{\left(m-j\right)}$,
i.e., $p-\sum_{j=0}^{N-1}p_{\left(m-j\right)}\in S^{m-N}\left(\mathbb{R}^{n}\times\mathbb{R}^{n},\Lambda\right)$,
for all $N\in\mathbb{N}_{0}$.
\end{enumerate}
This subset is denoted by $S_{cl}^{m}\left(\mathbb{R}^{n}\times\mathbb{R}^{n},\Lambda\right)$.
It is a Fréchet space with the natural seminorms.
\end{defn}
\begin{defn}
Let $p\in S_{cl}^{m}\left(\mathbb{R}^{n}\times\mathbb{R}^{n},\Lambda\right)$,
$m\in\mathbb{Z}$, be written as a function of $\left(x',x_{n},\xi',\xi_{n},\lambda\right)\in\mathbb{R}^{n-1}\times\mathbb{R}\times\mathbb{R}^{n-1}\times\mathbb{R}\times\Lambda$.
We say that it satisfies the transmission condition, if $p\sim\sum_{j=0}^{\infty}p_{\left(m-j\right)}$
and if, for all $k\in\mathbb{N}_{0}$ and for all $\alpha\in\mathbb{N}_{0}^{n+l}$,
we have
\[
D_{x_{n}}^{k}D_{\left(\xi,\lambda\right)}^{\alpha}p_{\left(m-j\right)}\left(x',0,0,1,0\right)=\left(-1\right)^{m-j-\left|\alpha\right|}D_{x_{n}}^{k}D_{\left(\xi,\lambda\right)}^{\alpha}p_{\left(m-j\right)}\left(x',0,0,-1,0\right).
\]
In this case, the operator $P\left(\lambda\right)_{+}:=r^{+}op\left(p\right)\left(\lambda\right)e^{+}\colon\mathcal{S}(\mathbb{R}_{+}^{n})\to\mathcal{S}(\mathbb{R}_{+}^{n})$
is well defined.
\end{defn}
Two more classes of functions are required. Our notation here follows
G. Grubb \cite{Grubbamarelo}.
\begin{defn}
We denote by $S^{m}(\mathbb{R}^{n-1},\mathcal{S}_{+},\Lambda)$, $m\in\mathbb{R}$,
the space of all functions $\tilde{f}\in C^{\infty}\left(\mathbb{R}^{n-1}\times\mathbb{R}_{+}\times\mathbb{R}^{n-1}\times\Lambda\right)$
that satisfy:
\[
\big\Vert x_{n}^{k}D_{x_{n}}^{k'}D_{x'}^{\beta'}D_{\xi'}^{\alpha'}D_{\lambda}^{\gamma}\tilde{f}\left(x',x_{n},\xi',\lambda\right)\big\Vert _{L^{\infty}\left(\mathbb{R}_{+x_{n}}\right)}\le C_{k,k',\alpha',\beta',\gamma}\left\langle \xi',\lambda\right\rangle ^{m+1-k+k'-\left|\alpha'\right|-\left|\gamma\right|}.
\]
The subset $S_{cl}^{m}(\mathbb{R}^{n-1},\mathcal{S}_{+},\Lambda)$
consists of all $\tilde{f}$ with an asymptotic expansion $\tilde{f}\sim\sum_{j=0}^{\infty}\tilde{f}_{\left(m-j\right)}$, i.e.
there are functions $\tilde{f}_{\left(m-j\right)}\in S^{m-j}(\mathbb{R}^{n-1},\mathcal{S}_{+},\Lambda)$,
$j\in\mathbb{N}_{0}$, such that $\tilde{f}-\sum_{j=0}^{N-1}\tilde{f}_{\left(m-j\right)}\in
S^{m-N}(\mathbb{R}^{n-1},\mathcal{S}_{+},\Lambda)$ for
all $N\in\mathbb{N}_{0}$, and
\[
\tilde{f}_{\left(m-j\right)}\Big(x',\frac{1}{t}x_{n},t\xi',t\lambda\Big)
=t^{m+1-j}\tilde{f}_{\left(m-j\right)}(x',x_{n},\xi',\lambda),\, t\ge1,\,\left|(\xi',\lambda)\right|\ge1.
\]
Similarly, $S^{m}(\mathbb{R}^{n-1},\mathcal{S}_{++},\Lambda)$ denotes
all $\tilde{g}\in C^{\infty}\left(\mathbb{R}^{n-1}\times\mathbb{R}_{++}\times\mathbb{R}^{n-1}\times\Lambda\right)$
with
\begin{eqnarray*}
\lefteqn{\big\Vert y_{n}^{l}x_{n}^{k}D_{y_{n}}^{l'}D_{x_{n}}^{k'}D_{x'}^{\beta'}D_{\xi'}^{\alpha'}D_{\lambda}^{\gamma}\tilde{g}\left(x',x_{n},y_{n},\xi',\lambda\right)\big\Vert _{L^{\infty}\left(\mathbb{R}_{++\left(x_{n},y_{n}\right)}\right)}}\\
&\le& C_{k,k',l,l',\alpha',\beta',\gamma}\left\langle \xi',\lambda\right\rangle ^{m+2-k+k'-l+l'-\left|\alpha'\right|-\left|\gamma\right|}.
\end{eqnarray*}
Write $\tilde{g}\in S_{cl}^{m}(\mathbb{R}^{n-1},\mathcal{S}_{+},\Lambda),$
if $\tilde{g}\sim\sum_{j=0}^{\infty}\tilde{g}_{\left(m-j\right)}$
with $\tilde{g}_{\left(m-j\right)}\in S^{m-j}(\mathbb{R}^{n-1},\mathcal{S}_{++},\Lambda)$
such that $\tilde{g}-\sum_{j=0}^{N-1}\tilde{g}_{\left(m-j\right)}$
belongs to $S^{m-N}(\mathbb{R}^{n-1},\mathcal{S}_{++},\Lambda)$,
for all $N\in\mathbb{N}_{0}$, and
\[
\tilde{g}_{\left(m-j\right)}\Big(x',\frac{1}{t}x_{n},\frac{1}{t}y_{n},t\xi',t\lambda\Big)=t^{m+2-j}\tilde{g}_{\left(m-j\right)}\left(x',x_{n},y_{n},\xi',\lambda\right),\, t\ge1,\,\left|\left(\xi',\lambda\right)\right|\ge1.
\]
\end{defn}
We may now define the operators that, together
with the pseudodifferential ones, appear in the Boutet de Monvel calculus:
the Poisson, trace and singular Green operators. We will always restrict ourselves to the classical elements. The notation
$\gamma_{j}:\mathcal{S}(\mathbb{R}_{+}^{n})\to\mathcal{S}(\mathbb{R}^{n-1})$,
$j\in\mathbb{N}_{0}$, indicates the operator $\gamma_{j}u\left(x'\right)=\lim_{x_{n}\to0}D_{x_{n}}^{j}u\left(x',x_{n}\right)$
as well as its extension to Sobolev, Bessel and Besov spaces.

\begin{defn}
\label{def:Poisson,TraceandGreen}Let $\lambda\in\Lambda$, 
$m\in \mathbb R$ and $d\in \mathbb N_0$.\\
1) A classical parameter-dependent Poisson operator of order $m$
is an operator family $K\left(\lambda\right):\mathcal{S}(\mathbb{R}^{n-1})\to\mathcal{S}(\mathbb{R}_{+}^{n})$ associated with $\tilde{k}\in S_{cl}^{m-1}(\mathbb{R}^{n-1},\mathcal{S}_{+},\Lambda)$
of the form
\begin{equation}
K\left(\lambda\right)u\left(x',x_{n}\right)=\left(2\pi\right)^{1-n}\int_{\mathbb{R}^{n-1}}e^{ix'\xi'}\tilde{k}(x',x_{n},\xi',\lambda)\hat{u}\left(\xi'\right)d\xi',\label{eq:DefPoisson}
\end{equation}
For $\tilde{k}\sim\sum_{j=0}^{\infty}\tilde{k}_{\left(m-1-j\right)}$,
we define $\tilde{k}_{\left(m-1\right)}\left(x',\xi',D_{n},\lambda\right):\mathbb{C}\to\mathcal{S}\left(\mathbb{R}_{+}\right)$
by
\[
\tilde{k}_{\left(m-1\right)}\left(x',\xi',D_{n},\lambda\right)\left(v\right)=v\tilde{k}_{\left(m-1\right)}\left(x',x_{n},\xi',\lambda\right).
\]
2) A classical parameter-dependent trace operator of order $m$
and class $d$ is an operator family $T\left(\lambda\right):\mathcal{S}(\mathbb{R}_{+}^{n})\to\mathcal{S}(\mathbb{R}^{n-1})$
of the form
\[
T\left(\lambda\right)=\sum_{j=0}^{d-1}S_{j}\left(\lambda\right)\gamma_{j}+T'\left(\lambda\right),
\]
where $S_{j}\left(\lambda\right)$ is a parameter-dependent pseudodifferential
operator of order $m-j$ on $\mathbb{R}^{n-1}$ and $T'\left(\lambda\right):\mathcal{S}(\mathbb{R}_{+}^{n})\to\mathcal{S}(\mathbb{R}^{n-1})$
is of the form
\begin{equation}
T'(\lambda)u\left(x'\right)=\left(2\pi\right)^{1-n}\int_{\mathbb{R}^{n-1}}e^{ix'\xi'}\int_{\mathbb{R}_{+}}\tilde{t}\left(x',x_{n},\xi',\lambda\right)\left(\mathcal{F}_{x'\to\xi'}u\right)\left(\xi',x_{n}\right)dx_{n}d\xi'\label{eq:DefTrace}
\end{equation}
with $\tilde{t}\in$$S_{cl}^{m}(\mathbb{R}^{n-1},\mathcal{S}_{+},\Lambda)$.
For $\tilde{t}\sim\sum_{j=0}^{\infty}\tilde{t}_{\left(m-j\right)}$
we define $\tilde{t}_{\left(m\right)}\left(x',\xi',D_{n},\lambda\right):\mathcal{S}\left(\mathbb{R}_{+}\right)\to\mathbb{C}$
by
\[
\tilde{t}_{\left(m\right)}\left(x',\xi',D_{n},\lambda\right)u=\int_{\mathbb{R}_{+}}\tilde{t}_{\left(m\right)}\left(x',x_{n},\xi',\lambda\right)u\left(x_{n}\right)dx_{n}.
\]
3) A classical parameter-dependent singular Green operator of order $m$
and type $d$ is an operator family $G\left(\lambda\right):\mathcal{S}(\mathbb{R}_{+}^{n})\to\mathcal{S}(\mathbb{R}_{+}^{n})$
of the form
\[
G\left(\lambda\right)=\sum_{j=0}^{d-1}K_{j}'\left(\lambda\right)\gamma_{j}+G'\left(\lambda\right),
\]
where $K_{j}'$ are Poisson operators of order $m-j$ and $G'\left(\lambda\right):\mathcal{S}\left(\mathbb{R}_{+}^{n}\right)\to\mathcal{S}\left(\mathbb{R}_{+}^{n}\right)$
is an operator of the form
\begin{equation}
G'\left(\lambda\right)u(x)=(2\pi)^{1-n}\int_{\mathbb{R}^{n-1}}e^{ix'\xi'}\int_{\mathbb{R}_{+}}\tilde{g}(x',x_{n},y_{n},\xi',\lambda)\left(\mathcal{F}_{x'\to\xi'}u\right)\left(\xi',y_{n}\right)dy_{n}d\xi',\label{eq:DefGreen}
\end{equation}
where $\tilde{g}\in S_{cl}^{m-1}(\mathbb{R}^{n-1},\mathcal{S}_{++},\Lambda)$.
We define the operator $g_{\left(m-1\right)}\left(x',\xi',D_{n},\lambda\right):\mathcal{S}\left(\mathbb{R}_{+}\right)\to\mathcal{S}\left(\mathbb{R}_{+}\right)$
by
\begin{eqnarray*}\lefteqn{
g_{\left(m-1\right)}\left(x',\xi',D_{n},\lambda\right)u\left(x_{n}\right)}\\
&=&\sum_{l=0}^{d-1}\tilde{k}'_{l\left(m-l-1\right)}\left(x',x_{n},\xi',\lambda\right)D_{x_{n}}^{l}u\left(0\right)
+\int_{\mathbb{R}_{+}}\tilde{g}_{\left(m-1\right)}\left(x',x_{n},y_{n},\xi',\lambda\right)u\left(y_{n}\right)dy_{n}.
\end{eqnarray*}
\end{defn}
\begin{rem}
\label{rem:Poisson,TraceandGreencont} With a symbol $p\in S_{cl}^{m}\left(\mathbb{R}^{n}\times\mathbb{\mathbb{R}}^{n},\Lambda\right)$
that satisfies the transmission condition, we associate the operator
$p_{\left(m\right)+}\left(x',0,\xi',D_{n},\lambda\right):\mathcal{S}\left(\mathbb{R}_{+}\right)\to\mathcal{S}\left(\mathbb{R}_{+}\right)$
defined by: 
\[
p_{\left(m\right)+}\left(x',0,\xi',D_{n},\lambda\right)u\left(x_{n}\right)=\frac{1}{2\pi}\int_{\mathbb{R}}e^{ix_{n}\xi_{n}}p_{\left(m\right)}\left(x',0,\xi',\xi_{n},\lambda\right)\widehat{e^{+}u}\left(\xi_{n}\right)d\xi_{n}.
\]
\end{rem}
\begin{defn}
Let $n_{1}$, $n_{2}$, $n_{3}$ and $n_{4}\in\mathbb{N}_{0}$. The
set of classical parameter-dependent Boutet de Monvel operators  on $\mathbb{R}_{+}^{n}$,
denoted by $\mathcal{B}_{n_{1},n_{2},n_{3},n_{4}}^{m,d}\left(\mathbb{R}^{n},\Lambda\right)$
for $m\in\mathbb{Z}$ and $d\in\mathbb{N}_{0}$, or just by $\mathcal{B}^{m,d}\left(\mathbb{R}^{n},\Lambda\right)$,
consists of all operators $A$ given by
\begin{equation}
A\left(\lambda\right)=\left(\begin{array}{cc}
P_{+}\left(\lambda\right)+G\left(\lambda\right) & K\left(\lambda\right)\\
T\left(\lambda\right) & S\left(\lambda\right)
\end{array}\right):\begin{array}{c}
\mathcal{S}\left(\mathbb{R}_{+}^{n}\right)^{n_{1}}\\
\oplus\\
\mathcal{S}\left(\mathbb{R}^{n-1}\right)^{n_{2}}
\end{array}\to\begin{array}{c}
\mathcal{S}\left(\mathbb{R}_{+}^{n}\right)^{n_{3}}\\
\oplus\\
\mathcal{S}\left(\mathbb{R}^{n-1}\right)^{n_{4}}
\end{array},\label{eq:DefAlambda}
\end{equation}
where: $P_{+}$ is a pseudodifferential operator
of order $m$ satisfying the transmission condition, 
$G$ is a singular Green operators of order
$m$ and type $d$, $K$ is a Poisson operator
of order $m$, $T$ is a trace operator of order
$m$ and type $d$ and $S$ is a pseudodifferential
operator of order $m$. All are parameter-dependent in the respective classes.
\end{defn}
The following algebra is also useful to prove spectral invariance: 
\begin{defn}
Let $n_{1}$, $n_{2}$, $n_{3}$, $n_{4}\in\mathbb{N}_{0}$ and $1<p<\infty$.
We define the set $\tilde{\mathcal{B}}_{n_{1},n_{2},n_{3},n_{4}}^{p}\left(\mathbb{R}^{n},\Lambda\right)$,
also denoted by $\tilde{\mathcal{B}}^{p}\left(\mathbb{R}^{n},\Lambda\right)$,
as the set of all operators $A$ of the form (\ref{eq:DefAlambda}),
where: $P_{+}$ is of order $0$, $G$ is of order
$0$ and type $0$, $K$ is
of order $\frac{1}{p}$, $T$ is of order $-\frac{1}{p}$ and type $0$ and $S$
is of order $0$. All are parameter-dependent in the respective classes.
\end{defn}

\begin{defn}
\label{def:Local boundary principal symbol}With $A\in\mathcal{B}_{n_{1},n_{2},n_{3},n_{4}}^{m,d}\left(\mathbb{R}^{n},\Lambda\right)$,
we associate the operator-valued principal boundary symbol 
$\sigma_{\partial}\left(A\right)$,
defined on 
$\mathbb{R}^{n-1}\times\left(\left(\mathbb{R}^{n-1}\times\Lambda\right)\backslash\left\{ 0\right\} \right)$. 
The operator 
\begin{eqnarray}\label{boundarysymboldefinition}
\sigma_\partial(A)(x',\xi',\lambda): 
\mathcal{S}\left(\mathbb{R}_{+}\right)^{n_{1}}\oplus\mathbb{C}^{n_{2}}
\to \mathcal{S}\left(\mathbb{R}_{+}\right)^{n_{3}}
\oplus
\mathbb{C}^{n_{4}}
\end{eqnarray}
is given by 
\begin{eqnarray}
\begin{pmatrix}
p_{\left(m\right)+}\left(x',0,\xi',D_{n},\lambda\right)+g_{(m-1)}\left(x',\xi',D_{n},\lambda\right) & k_{(m-1)}\left(x',\xi',D_{n},\lambda\right)\\
t_{(m)}\left(x',\xi',D_{n},\lambda\right) & s_{(m)}\left(x',\xi',\lambda\right)
\end{pmatrix}
\nonumber
\end{eqnarray}
where the entries are the matrix version of the operators in Definition
\ref{def:Poisson,TraceandGreen} and Remark \ref{rem:Poisson,TraceandGreencont}.

Similarly, with $A\in\tilde{\mathcal{B}}_{n_{1},n_{2},n_{3},n_{4}}^{p}\left(\mathbb{R}^{n},\Lambda\right)$,
we associate an operator $\sigma_{\partial}\left(A\right)\left(x',\xi',\lambda\right)$
acting as in \eqref{boundarysymboldefinition}, given as
\begin{eqnarray}
\begin{pmatrix}
p_{\left(0\right)+}\left(x',0,\xi',D_{n},\lambda\right)+g_{(-1)}\left(x',\xi',D_{n},\lambda\right) & k_{(\frac{1}{p}-1)}\left(x',\xi',D_{n},\lambda\right)\\
t_{(-\frac{1}{p})}\left(x',\xi',D_{n},\lambda\right) & s_{(0)}\left(x',\xi',\lambda\right)
\end{pmatrix}
\nonumber.
\end{eqnarray}
\end{defn}
Let now $M$ be a manifold with boundary, $E_{0}$ and $E_{1}$
two complex hermitian vector bundles over $M$ and $F_{0}$ and $F_{1}$
two complex hermitian vector bundles over $\partial M$. Let $U_{j}\subset M$,
$j=1,...,N$, be open cover of $M$ consisting of trivializing sets for the vector bundles, 
$\Phi_{1}$, \ldots, $\Phi_{N}\in C^{\infty}\left(M\right)$ be a partition
of unity subordinate to $U_{1}$, \ldots, $U_{N}$ and $\Psi_{1}$, \ldots,$\Psi_{N}\in C^{\infty}\left(M\right)$
be  supported in $U_{j}$ such that $\Psi_{j}\Phi_{j}=\Phi_{j}$.

A linear operator $A\left(\lambda\right):C^{\infty}\left(M,E_{0}\right)\oplus C^{\infty}\left(\partial M,F_{0}\right)\to C^{\infty}\left(M,E_{1}\right)\oplus C^{\infty}\left(\partial M,F_{1}\right)$
can always be written as 
\[
A\left(\lambda\right)=\sum_{j=1}^{N}\Phi_{j}A\left(\lambda\right)\Psi_{j}+\sum_{j=1}^{N}\Phi_{j}A\left(\lambda\right)\left(1-\Psi_{j}\right).
\]

Using the above definitions, we define the Boutet de Monvel algebra
on $M$:
\begin{defn}
A family  $A\left(\lambda\right):C^{\infty}\left(M,E_{0}\right)\oplus C^{\infty}\left(\partial M,F_{0}\right)\to C^{\infty}\left(M,E_{1}\right)\oplus C^{\infty}\left(\partial M,F_{1}\right)$,
$\lambda\in\Lambda$, is called a parameter-dependent Boutet de Monvel
operator of order $m\in\mathbb{Z}$ and class $d\in\mathbb{N}_{0}$,
if

1) The operators $\Psi A\left(\lambda\right)\Phi:\mathcal{S}\left(\mathbb{R}_{+}^{n}\right)^{n_{1}}\oplus\mathcal{S}\left(\mathbb{R}^{n-1}\right)^{n_{2}}\to\mathcal{S}\left(\mathbb{R}_{+}^{n}\right)^{n_{3}}\oplus\mathcal{S}\left(\mathbb{R}^{n-1}\right)^{n_{4}}$
belong to $\mathcal{B}^{m,d}\left(\mathbb{R}^{n},\Lambda\right)$ after localization.

2) The Schwartz kernels of the operators $\sum_{j=1}^{N}\Phi_{j}A\left(\lambda\right)\left(1-\Psi_{j}\right)$
belong to
\[
\begin{pmatrix}
\mathcal{S}\left(\Lambda,C^{\infty}\left(M\times M,Hom\left(\pi_{2}^{*} E_{0},\pi_{1}^{*}E_{1}\right)\right)\right) & \mathcal{S}\left(\Lambda,C^{\infty}\left(M\times\partial M,Hom\left(\pi_{2}^{*}F_{0},\pi_{1}^{*}E_{1}\right)\right)\right)\\
\mathcal{S}\left(\Lambda,C^{\infty}\left(\partial M\times M,Hom\left(\pi_{2}^{*}E_{0},\pi_{1}^{*}F_{1}\right)\right)\right) & \mathcal{S}\left(\Lambda,C^{\infty}\left(\partial M\times\partial M,Hom\left(\pi_{2}^{*}F_{0},\pi_{1}^{*}F_{1}\right)\right)\right)
\end{pmatrix},
\]
where $Hom$ indicates the space of homomorphisms and $\pi_{i}:M\times M\to M$
is given by $\pi_{i}\left(x_{1},x_{2}\right)=x_{i}$ for $i=1,2$. 

If $\partial M=\emptyset$, the algebra reduces to the classical parameter-dependent
pseudodifferential operators. The above definition is independent
of the partitions of unity and trivializing sets we choose.
\end{defn}
A central notion is parameter-ellipticity:
\begin{defn}
Given a   parameter-dependent Boutet de Monvel operator $A\in\mathcal{B}_{E_{0},F_{0},E_{1},F_{1}}^{m,d}\left(M,\Lambda\right)$
 we define:

1) The interior principal symbol $\sigma_{\psi}(A)\in C^{\infty}((T^{*}M\times\Lambda)\backslash\left\{ 0\right\} ,Hom(\pi_{T^{*}M\times\Lambda}^{*}E_{0},$ $\pi_{T^{*}M\times\Lambda}^{*}E_{1}))$,
where $\pi_{T^{*}M\times\Lambda}:T^{*}M\times\Lambda\to M$ is the
canonical projection. It is the principal symbol of the pseudodifferential
operator part of the operator $A$.

2) The boundary principal symbol $\sigma_\partial(A)$. For $(z,\lambda)\in\left(T^{*}\partial M\times\Lambda\right)\backslash\left\{ 0\right\} $ we let
\[
\sigma_{\partial}\left(A\right)\left(z\right)\left(\lambda\right):\pi_{T^{*}\partial M\times\Lambda}^{*}\begin{pmatrix}
\left.E_{0}\right|_{\partial M}\otimes\mathcal{S}\left(\mathbb{R}_{+}\right)\\
\oplus\\
F_{0}
\end{pmatrix}\to\pi_{T^{*}\partial M\times\Lambda}^{*}
\begin{pmatrix}
\left.E_{1}\right|_{\partial M}\otimes\mathcal{S}\left(\mathbb{R}_{+}\right)\\
\oplus\\
F_{1}
\end{pmatrix},
\]
where $\pi_{T^{*}\partial M\times\Lambda}:\left(T^{*}\partial M\times\Lambda\right)\backslash\left\{ 0\right\} \to\partial M$
is the canonical projection. After localization, it corresponds to
the symbol in Definition \ref{def:Local boundary principal symbol}.

We say that $A\left(\lambda\right)$ is parameter-elliptic if both
symbols are invertible. With obvious changes, 
we can also define parameter-ellipticity,
interior and boundary principal symbols of operators $A\in\tilde{\mathcal{B}}_{E_{0},F_{0},E_{1},F_{1}}^{p}\left(M,\Lambda\right)$. 
\end{defn}
The above operators act continuously on Bessel and Besov spaces. First
we fix a dyadic partition of unity $\left\{ \varphi_{j};\,j\in\mathbb{N}_{0}\right\} $.

\begin{defn}
Let $\varphi_{0}\in C_{c}^{\infty}(\mathbb{R}^{n})$ be supported 
$\left\{ \xi;\,\left|\xi\right|<2\right\} $,
$0\le\varphi_{0}\le1$ and $\varphi_{0}\left(\xi\right)=1$ in a neighborhood
of the closed unit ball. Define $\varphi_{j}\in C_{c}^{\infty}(\mathbb{R}^{n})$,
$j\ge1$, by $\varphi_{j}\left(\xi\right)=\varphi_{0}\left(2^{-j}\xi\right)-\varphi_{0}\left(2^{-j+1}\xi\right)$.
\end{defn}

\begin{rem}
\label{rem:Definicao conjuntos Kj} We use the following notation:
$K_{j}:=\left\{ \xi\in\mathbb{R}^{n};2^{j-1}\le\left|\xi\right|\le2^{j+1}\right\} $,
for $j\ge1$, and $K_{0}:=\left\{ \xi\in\mathbb{R}^{n};\left|\xi\right|\le2\right\} $.
The above definition implies that $\mbox{supp}\left(\varphi_{j}\right)\subset\text{interior}\left(K_{j}\right)$,
for $j\ge0$. Moreover, we see that $\varphi_{j}\left(\xi\right)=\varphi_{1}\left(2^{-j+1}\xi\right)$,
for $j\ge2$ and $\sum_{j=0}^{\infty}\varphi_{j}\left(\xi\right)=1$,
$\xi\in\mathbb{R}^{n}$.
\end{rem}
\begin{defn}
For each $s\in\mathbb{R}$, we define the operator $\left\langle D\right\rangle ^{s}:\mathcal{S}'\left(\mathbb{R}^{n}\right)\to\mathcal{S}'\left(\mathbb{R}^{n}\right)$
as the pseudodifferential operator with symbol $\xi\in\mathbb{R}^{n}\mapsto\left\langle \xi\right\rangle ^{s}$. Moreover, we write  $\varphi_{j}(D)u=op(\varphi_{j})u$.

1) The Bessel potential space $H_{p}^{s}\left(\mathbb{R}^{n}\right)=\left\{ u\in\mathcal{S}'\left(\mathbb{R}^{n}\right);\,\left\langle D\right\rangle ^{s}u\in L_{p}\left(\mathbb{R}^{n}\right)\right\} $,
for $1<p<\infty$ and $s\in\mathbb{R}$, is the Banach space with
norm $\left\Vert u\right\Vert _{H_{p}^{s}\left(\mathbb{R}^{n}\right)}:=\left\Vert \left\langle D\right\rangle ^{s}u\right\Vert _{L_{p}\left(\mathbb{R}^{n}\right)}$.

2) The Besov space $B_{p}^{s}\left(\mathbb{R}^{n}\right)$, for $s\in\mathbb{R}$
and $1<p<\infty$, is the Banach space of all tempered distributions
$f\in\mathcal{S}'\left(\mathbb{R}^{n}\right)$ that satisfy:
\[
\left\Vert f\right\Vert _{B_{p}^{s}\left(\mathbb{R}^{n}\right)}:=\Big(\sum_{j=0}^{\infty}2^{jsp}\left\Vert \varphi_{j}\left(D\right)f\right\Vert _{L^{p}\left(\mathbb{R}^{n}\right)}^{p}\Big)^{\frac{1}{p}}<\infty.
\]

For an open set $\Omega\subset\mathbb{R}^{n}$, we define the Bessel
potential spaces $H_{p}^{s}\left(\Omega\right)$, as the set of restrictions
of $H_{p}^{s}\left(\mathbb{R}^{n}\right)$ to $\Omega$ with norm
\[
\left\Vert u\right\Vert _{H_{p}^{s}\left(\Omega\right)}:=\left\{ \mbox{inf }\left\Vert v\right\Vert _{H_{p}^{s}\left(\mathbb{R}^{n}\right)};\,r_{\Omega}\left(v\right)=u\right\} .
\]
 Similarly, we define the Besov spaces $B_{p}^{s}\left(\Omega\right)$.
Together with partition of unity and local charts, this leads to the
spaces $H_{p}^{s}\left(M\right)$, $H_{p}^{s}\left(M,E\right)$, $B_{p}^{s}\left(\partial M\right)$
and $B_{p}^{s}\left(\partial M,E\right)$, where $E$ is a vector
bundle over $M$ or $\partial M$.
\end{defn}
\begin{rem}
\label{rem:BesovandBesselSpaces} Let $s\in\mathbb{R}$, $1<p<\infty$
and $\frac{1}{p}+\frac{1}{q}=1$.

1) There are continuous inclusions $C_{c}^{\infty}\left(\mathbb{R}^{n}\right)\hookrightarrow\mathcal{S}(\mathbb{R}^{n})\hookrightarrow B_{p}^{s}\left(\mathbb{R}^{n}\right)\hookrightarrow\mathcal{S}'\left(\mathbb{R}^{n}\right)$.
Moreover the spaces $C_{c}^{\infty}\left(\mathbb{R}^{n}\right)$ and
$\mathcal{S}\left(\mathbb{R}^{n}\right)$ are dense in $B_{p}^{s}\left(\mathbb{R}^{n}\right)$.
The same can be said of $H_{p}^{s}\left(\mathbb{R}^{n}\right)$.

2) The dual of $B_{p}^{s}\left(\mathbb{R}^{n}\right)$ is $B_{q}^{-s}\left(\mathbb{R}^{n}\right)$,
where the identification is given by the $L^{2}$ scalar product.
Again the same holds for $H_{p}^{s}\left(\mathbb{R}^{n}\right)$ and
$H_{q}^{-s}\left(\mathbb{R}^{n}\right)$.

3) A pseudodifferential operator with symbol $a\in S^{m}\left(\mathbb{R}^{n}\times\mathbb{R}^{n}\right)$
extends to continuous operators $op(a):H_{p}^{s}\left(\mathbb{R}^{n}\right)\to H_{p}^{s-m}\left(\mathbb{R}^{n}\right)$
and $op(a):B_{p}^{s}\left(\mathbb{R}^{n}\right)\to B_{p}^{s-m}\left(\mathbb{R}^{n}\right)$
for all $s\in\mathbb{R}$.

4) The following interpolation holds: $\left(L^{p}\left(\mathbb{R}^{n}\right),H_{p}^{1}\left(\mathbb{R}^{n}\right)\right)_{\theta,p}=B_{p}^{\theta}\left(\mathbb{R}^{n}\right)$,
for all $0<\theta<1$, where $\left(X,Y\right)_{\theta,p}$ denotes
the real interpolation space of the interpolation couple $\left(X,Y\right)$,
as in A. Lunardi \cite{Lunardi}.

5) If $M$ is a compact manifold (with or without boundary) and $E$
is a vector bundle over $M$, then $H_{p}^{s}(M,E)\hookrightarrow H_{p}^{s'}(M,E)$
and $B_{p}^{s}(M,E)\hookrightarrow B_{p}^{s'}(M,E)$ are compact inclusions,
whenever $s>s'$.

6) The trace functional $\gamma_{0}:\mathcal{S}(\mathbb{R}^{n})\to\mathcal{S}(\mathbb{R}^{n-1})$
extends to a continuous and surjective map $\gamma_{0}:H_{p}^{s}(\mathbb{R}^{n})\to B_{p}^{s-\frac{1}{p}}(\mathbb{R}^{n-1})$
when $s>\frac{1}{p}$.

7) The Besov spaces do not depend on the choice of the dyadic partition of unity; different partitions yield equivalent norms.
\end{rem}
\begin{rem}
Let $G$ be a $UMD$ Banach space that satisfies the property $\left(\alpha\right)$.
Using Bochner integrals, we can define $B_{p}^{s}\left(\mathbb{R},G\right)$
and $H_{p}^{s}\left(\mathbb{R},G\right)$ in the same way as before,
see, for instance, \cite{amannFunctionSpacesAnisotropic,DenkKaip}.
It is worth noting that $B_{p}^{s}\left(\mathbb{R}^{n}\right)$
and $B_{p}^{s}\left(\partial X,E\right)$ are $UMD$ spaces with the
property $\left(\alpha\right)$ for all $s\in\mathbb{R}$ and $1<p<\infty$.
Later, we also use that $B_{p}^{s}\left(\mathbb{R},G\right)\subset H_{p}^{1}\left(\mathbb{R},G\right):=\left\{ u\in L_{p}\left(\mathbb{R},G\right);\,\frac{du}{dt}\in L_{p}\left(\mathbb{R},G\right)\right\} $,
for all $0<s<1$.
\end{rem}
Let us now state the following properties of composition, adjoints
and continuity of Boutet de Monvel operators \cite{RempelSchulze,Grubbverde,Grubblp,Grubblpparameters}.
\begin{thm}
\label{thm:Propriedades de BdM}

{\rm1) (}Composition{\rm)} Let $A\in\mathcal{B}_{E_{1},F_{1},E_{2},F_{2}}^{m,d}\left(M,\Lambda\right)$, $B\in\mathcal{B}_{E_{0},F_{0},E_{1},F_{1}}^{m',d'}\left(M,\Lambda\right)$. Then $AB\in\mathcal{B}_{E_{0},F_{0},E_{2},F_{2}}^{m+m',d''}\left(M,\Lambda\right)$,
where $d'':=\max\left\{ m'+d,d'\right\} $. 
Similarly, if $A\in\tilde{\mathcal{B}}_{E_{1},F_{1},E_{2},F_{2}}^{p}\left(M,\Lambda\right)$
and $B\in\tilde{\mathcal{B}}_{E_{0},F_{0},E_{1},F_{1}}^{p}\left(M,\Lambda\right)$,
then $AB\in\tilde{\mathcal{B}}_{E_{0},F_{0},E_{2},F_{2}}^{p}\left(M,\Lambda\right)$.

{\rm2) (}Adjoint{\rm)}
Let $A\in\tilde{\mathcal{B}}_{E_{0},F_{0},E_{1},F_{1}}^{p}\left(M,\Lambda\right)$.
Then $A^{*}\in\tilde{\mathcal{B}}_{E_{1},F_{1},E_{0},F_{0}}^{q}\left(M,\Lambda\right)$, where
$\frac{1}{p}+\frac{1}{q}=1$ and $A^{*}$ is the only operator
that satisfies, for every $u\in C^{\infty}\left(M,E_{0}\right)\oplus C^{\infty}\left(\partial M,F_{0}\right)$
and $v\in C^{\infty}\left(M,E_{1}\right)\oplus C^{\infty}\left(\partial M,F_{1}\right)$,
the relation 
\[
\left(A\left(\lambda\right)u,v\right)_{L^{2}(M,E_{1})\oplus L^{2}(M,F_{1})}=\left(u,A^{*}\left(\lambda\right)v\right)_{L^{2}(M,E_{0})\oplus L^{2}(M,F_{0})}.
\]

{\rm3) (}Continuity{\rm)}
An operator $A\in\mathcal{B}_{E_{0},F_{0},E_{1},F_{1}}^{m,d}\left(M,\Lambda\right)$
induces bounded operators $A\left(\lambda\right):H_{p}^{s}\left(M,E_{0}\right)\oplus B_{p}^{s-\frac{1}{p}}\left(\partial M,F_{0}\right)\to H_{p}^{s-m}\left(M,E_{1}\right)\oplus B_{p}^{s-m-\frac{1}{p}}\left(\partial M,F_{1}\right)$
for all $s>d-1+\frac{1}{p}$. Similarly $A\in\tilde{\mathcal{B}}_{E_{0},F_{0},E_{1},F_{1}}^{p}\left(M,\Lambda\right)$
induces bounded operators $A\left(\lambda\right):H_{p}^{s}\left(M,E_{0}\right)\oplus B_{p}^{s}\left(\partial M,F_{0}\right)\to H_{p}^{s}\left(M,E_{1}\right)\oplus B_{p}^{s}\left(\partial M,F_{1}\right)$,
$\forall s>-1+\frac{1}{p}$.

{\rm4) (}Fredholm property{\rm)}
If $A\in\mathcal{B}_{E_{0},F_{0},E_{1},F_{1}}^{m,d}\left(M,\Lambda\right)$,
$d=\max\left\{ m,0\right\} $, is parameter-elliptic, then
there exists a $B\in\mathcal{B}_{E_{1},F_{1},E_{0},F_{0}}^{-m,d'}\left(M,\Lambda\right)$,
$d'=\max\left\{ -m,0\right\} $, such that
\begin{equation}
AB-I\in\mathcal{B}_{E_{1},F_{1},E_{1},F_{1}}^{-\infty,d'}\left(M,\Lambda\right)\,\,\,\,\,\mbox{and}\,\,\,\,\,BA-I\in\mathcal{B}_{E_{0},F_{0},E_{0},F_{0}}^{-\infty,d}\left(M,\Lambda\right).\label{eq:AB-IBeA-I}
\end{equation}

As a consequence,  $A\left(\lambda\right)$
is a Fredholm operator of index $0$ for each $\lambda\in\Lambda$, and there exists a constant
$\lambda_{0}>0$ such that $A\left(\lambda\right)$ is invertible, if $\left|\lambda\right|\ge\lambda_{0}$.

Similarly, if $A\in\tilde{\mathcal{B}}_{E_{0},F_{0},E_{1},F_{1}}^{p}\left(M,\Lambda\right)$
is parameter-elliptic, then there exists a $B\in\tilde{\mathcal{B}}_{E_{1},F_{1},E_{0},F_{0}}^{p}\left(M,\Lambda\right)$
such that Equation \eqref{eq:AB-IBeA-I} holds for $d=d'=0$.

\end{thm}

\subsection{The equivalence between ellipticity and Fredholm property}

In this section, we prove that the Fredholm property together with
some growth condition on $\lambda$ implies parameter-dependent ellipticity.
The use of Besov spaces makes the proofs a little more elaborate than
e.g. the proof in the parameter-independent $L^{2}$-case studied
by S. Rempel and B.-W. Schulze \cite{RempelSchulze}. To make it clearer,
we first study the pseudodifferential term on Besov spaces and then
the boundary terms.

\subsubsection{Pseudodifferential operators with parameters on a manifold without
boundary acting on Besov spaces. \label{subsec:Pseudodifferential-term.}}

In this section, we prove the following theorem:
\begin{thm}
\label{thm:Spectral-Invariance-in-Besov-Spaces}Let $M$ be a compact
manifold without boundary, $E$ and $F$ be vector bundles over $M$.
Let $A\left(\lambda\right):C^{\infty}\left(M,E\right)\to C^{\infty}\left(M,F\right)$,
$\lambda\in\Lambda$, be a classical parameter-dependent pseudodifferential
operator of order $0$. Then the following conditions are equivalent:
\begin{enumerate}\renewcommand{\labelenumi}{\roman{enumi})}
\item $A$ is parameter-elliptic.
\item
 There exist uniformly bounded operators $B_{j}\left(\lambda\right):B_{p}^{0}(M,F)\to B_{p}^{0}(M,E)$,
$\lambda\in\Lambda$, $j=1$ and $2$, such that
\[
B_{1}\left(\lambda\right)A\left(\lambda\right)=1+K_{1}\left(\lambda\right)\,\mbox{and}\,A\left(\lambda\right)B_{2}\left(\lambda\right)=1+K_{2}\left(\lambda\right).
\]
where $K_{1}\left(\lambda\right):B_{p}^{0}(M,E)\to B_{p}^{0}(M,E)$
and $K_{2}\left(\lambda\right):B_{p}^{0}(M,F)\to B_{p}^{0}(M,F)$
are compact operators for every $\lambda\in\Lambda$ and $\lim_{\left|\lambda\right|\to\infty}K_{j}\left(\lambda\right)=0$.

\item There exist bounded operators $B_{j}\left(\lambda\right):B_{p}^{0}(M,F)\to B_{p}^{0}(M,E)$,
$\lambda\in\Lambda$, $j=1$ and $2$, such that
\[
B_{1}\left(\lambda\right)A\left(\lambda\right)=1+K_{1}\left(\lambda\right)\,\mbox{and}\,A\left(\lambda\right)B_{2}\left(\lambda\right)=1+K_{2}\left(\lambda\right).
\]
where $K_{1}\left(\lambda\right):B_{p}^{0}(M,E)\to B_{p}^{0}(M,E)$
and $K_{2}\left(\lambda\right):B_{p}^{0}(M,F)\to B_{p}^{0}(M,F)$
are compact operators for every $\lambda\in\Lambda$.
Moreover, $\lim_{\left|\lambda\right|\to\infty}K_{j}\left(\lambda\right)=0$
and there exist $M\in\mathbb{N}_{0}$ and $C>0$ such that $\left\Vert B_{j}\left(\lambda\right)\right\Vert _{\mathcal{B}\left(B_{p}^{0}(M,F),B_{p}^{0}(M,E)\right)}\le C\left\langle \ln\left(\lambda\right)\right\rangle ^{M}$,
for $j=1$ and $2$.
\end{enumerate}
\end{thm}
The third item also holds if $\left\Vert B_{j}\left(\lambda\right)\right\Vert _{\mathcal{B}\left(B_{p}^{0}(M,F),B_{p}^{0}(M,E)\right)}\le C\left\langle \lambda\right\rangle ^{r}$,
for some sufficiently small $r$, as a careful study of our proof
shows.

We note that $A\left(\lambda\right)B_{2}\left(\lambda\right)=1+K_{2}\left(\lambda\right)$
is equivalent to $B_{2}\left(\lambda\right)^{*}A\left(\lambda\right)^{*}=1+K_{2}\left(\lambda\right)^{*}$,
where $*$ indicates the adjoint. This is the condition that we shall
need. Obviously condition $i)$ implies that $\mbox{dim}\left(E\right)=\mbox{dim}\left(F\right)$.

If $i)$ holds,
then we can find a parametrix to $A\left(\lambda\right)$ by Theorem \ref{thm:Propriedades de BdM}(4)  so that
$ii)$ is true, and $ii)$ trivially implies $iii)$. So we only need to prove that $iii)$ implies $i)$.
\begin{defn}
Let $s>0$, $0<\tau<\frac{1}{3}$ and $(y,\eta)\in\mathbb{R}^{n}\times\mathbb{R}^{n}$.
We define the operator $R_{s}(y,\eta):\mathcal{S}\left(\mathbb{R}^{n}\right)\to\mathcal{S}\left(\mathbb{R}^{n}\right)$,
also denoted just by $R_{s}$, by
\[
R_{s}u\left(x\right)=s^{\frac{\tau n}{p}}e^{isx\eta}u\left(s^{\tau}\left(x-y\right)\right).
\]
\end{defn}
Below we collect some well-known facts about the operators $R_{s}$.
The items 1, 2, 4, 5 and 6 can be found in \cite{RempelSchulze,SchroheSeilerSpectConical}.
As we are dealing also with Besov spaces, some estimates must be done
more carefully. The third item was not proven in the previous references.
Statement 7 is stronger than usual. Both are necessary, as $R_{s}$
is not an isometry in the space $B_{p}^{0}\left(\mathbb{R}^{n}\right)$.

\begin{lem}
\label{lem:Propriedades R_s}The operator $R_{s}=R_{s}\left(y,\eta\right)$
has the following properties:
\begin{enumerate}\renewcommand{\labelenumi}{\arabic{enumi})}
\item $\left\Vert R_{s}u\right\Vert _{L^{p}\left(\mathbb{R}^{n}\right)}=\left\Vert u\right\Vert _{L^{p}\left(\mathbb{R}^{n}\right)}$
for all $u\in\mathcal{S}\left(\mathbb{R}^{n}\right)$.

\item $\lim_{s\to\infty}R_{s}u=0$ weakly in $L_{p}\left(\mathbb{R}^{n}\right)$
for all $u\in\mathcal{S}\left(\mathbb{R}^{n}\right)$.

\item $R_{s}:B_{p}^{\theta}\left(\mathbb{R}^{n}\right)\to B_{p}^{\theta}\left(\mathbb{R}^{n}\right)$
is continuous for all $s>0$ and $\left\Vert R_{s}u\right\Vert _{B_{p}^{\theta}\left(\mathbb{R}^{n}\right)}\le C_{\theta}\left(1+s\left\langle \eta\right\rangle \right)^{\theta}\left\Vert u\right\Vert _{H_{p}^{1}\left(\mathbb{R}^{n}\right)}$,
for every $\theta\in\left]0,1\right[$, $s\ge1$ and $u\in\mathcal{S}\left(\mathbb{R}^{n}\right)$.
The constant $C_{\theta}$ depends on $\theta$, but not on $y$,
$\eta$ or $s$.

\item  The operator $R_{s}$ is invertible. Its inverse is given by
\[
R_{s}^{-1}u\left(x\right)=s^{-\frac{\tau n}{p}}e^{-is\left(y+s^{-\tau}x\right)\eta}u\left(y+s^{-\tau}x\right).
\]

\item The Fourier transform of $R_{s}u$ is given by
\[
\mathcal{F}\left(R_{s}u\right)\left(\xi\right)=s^{\frac{\tau n}{p}-n\tau}e^{-iy\left(\xi-s\eta\right)}\hat{u}\left(s^{-\tau}\left(\xi-s\eta\right)\right).
\]

\item Let $a\in S^{m}\left(\mathbb{R}^{n}\times\mathbb{R}^{n},\Lambda\right)$.
Then
\[
R_{s}^{-1}op(a)\left(s\lambda\right)R_{s}u(x)=op(a_{s})\left(\lambda\right)u(x),
\]
where $a_{s}(x,\xi,\lambda)=a\left(y+s^{-\tau}x,s\eta+s^{\tau}\xi,s\lambda\right)$.

\item Let $a\in S_{cl}^{0}\left(\mathbb{R}^{n}\times\mathbb{R}^{n},\Lambda\right)$
be  classical, $u\in\mathcal{S}\left(\mathbb{R}^{n}\right)$,
$\lambda\in\Lambda$ with $\left(\eta,\lambda\right)\ne\left(0,0\right)$
and  $0<r<\tau$. Then
\begin{equation}
\lim_{s\to\infty}s^{r}\left\Vert op(a)\left(s\lambda\right)R_{s}u-a_{(0)}(y,\eta,\lambda)R_{s}u\right\Vert _{B_{p}^{0}\left(\mathbb{R}^{n}\right)}=0.\label{eq:ConvergenciaBp}
\end{equation}
\end{enumerate}
\end{lem}
\begin{proof}
1), 4) and 5) are just simple computations, and 6) follows from 4),
5) and the definition of pseudodifferential operators.

In order to prove 2), we just have to note that $\lim_{s\to\infty}\int_{\mathbb{R}^{n}}R_{s}u\left(x\right)v\left(x\right)dx=0$
for all $u\in\mathcal{S}\left(\mathbb{R}^{n}\right)$ and $v\in\mathcal{S}\left(\mathbb{R}^{n}\right)$.
The proof follows then from the fact that $L_{p}(\mathbb{R}^{n})'\simeq L_{q}\left(\mathbb{R}^{n}\right)$,
for $\frac{1}{p}+\frac{1}{q}=1$, and that $R_{s}$ is an isometry.

3) The operator $R_{s}:B_{p}^{\theta}\left(\mathbb{R}^{n}\right)\to B_{p}^{\theta}\left(\mathbb{R}^{n}\right)$
is continuous for all $s\in\mathbb{R}$, as $R_{s}$ is the composition
of dilatation, translation and multiplication by $e^{is\eta x}$.
The estimate follows by interpolation. In fact, for $s\ge1$, it is
easy  to see that $\left\Vert R_{s}u\right\Vert _{H_{p}^{1}\left(\mathbb{R}^{n}\right)}\le\left(1+s\left\langle \eta\right\rangle \right)\left\Vert u\right\Vert _{H_{p}^{1}\left(\mathbb{R}^{n}\right)}$.
As $\left(L_{p}\left(\mathbb{R}^{n}\right),H_{p}^{1}\left(\mathbb{R}^{n}\right)\right)_{\theta,p}=B_{p}^{\theta}\left(\mathbb{R}^{n}\right)$,
we conclude (see A. Lunardi \cite[Corollary 1.1.7]{Lunardi}) that
there exists a constant $C_{\theta}$ such that
\[
\left\Vert R_{s}u\right\Vert _{B_{p}^{\theta}\left(\mathbb{R}^{n}\right)}\le C_{\theta}\left\Vert R_{s}u\right\Vert _{H_{p}^{1}\left(\mathbb{R}^{n}\right)}^{\theta}\left\Vert R_{s}u\right\Vert _{L^{p}\left(\mathbb{R}^{n}\right)}^{1-\theta}\le C_{\theta}\left(1+s\left\langle \eta\right\rangle \right)^{\theta}\left\Vert u\right\Vert _{H_{p}^{1}\left(\mathbb{R}^{n}\right)}.
\]

7) This is the longest statement we need to prove. We divide the proof
into several steps. Our first goal is the  $L_p$-convergence: 
\begin{equation}
\lim_{s\to\infty}s^{r}\left\Vert op(a_{s})\left(\lambda\right)u-a_{(0)}(y,\eta,\lambda)u\right\Vert _{L_{p}\left(\mathbb{R}^{n}\right)}=0,\,\text{ where }\,u\in\mathcal{S}\left(\mathbb{R}^{n}\right).\label{eq:Lpconvergence}
\end{equation}

In a {\bf first step} let us show that, for every $\left(x,\xi\right)\in\mathbb{R}^{n}\times\mathbb{R}^{n}$,
\begin{equation}
\left|a\left(y+s^{-\tau}x,s\eta+s^{\tau}\xi,s\lambda\right)-a_{\left(0\right)}\left(y,\eta,\lambda\right)\right|\le C_{\lambda,\eta}\left\langle x\right\rangle \left\langle \xi\right\rangle ^{2}s^{-\tau}.\label{eq:estamenosazero}
\end{equation}
Let $\chi:\mathbb{R}^{n}\times\Lambda\to\mathbb{C}$ be a smooth function
that is equal to $0$ in a neighborhood of the origin and  equal
to $1$ outside a closed ball centered at the origin that does
not contain $\left(\eta,\lambda\right)$. For $s\ge1$, we have
\begin{eqnarray}
|a\left(y+s^{-\tau}x,s\eta+s^{\tau}\xi,s\lambda\right)
&-&\chi\left(s\eta+s^{\tau}\xi,s\lambda\right)a_{\left(0\right)}\left(y+s^{-\tau}x,s\eta+s^{\tau}\xi,s\lambda\right)|
\nonumber\\
&\le&
\frac{C}{\left\langle s\eta+s^{\tau}\xi,s\lambda\right\rangle }\le Cs^{\tau}\left\langle s\eta,s\lambda\right\rangle ^{-1}\left\langle \xi\right\rangle ,\label{eq:estimativa1}
\end{eqnarray}
where we have used Peetre's inequality. Since $a_{\left(0\right)}\left(y,s\eta,s\lambda\right)=a_{\left(0\right)}\left(y,\eta,\lambda\right)$,
\begin{eqnarray}
\lefteqn{\left|\chi\left(s\eta+s^{\tau}\xi,s\lambda\right)a_{\left(0\right)}\left(y+s^{-\tau}x,s\eta+s^{\tau}\xi,s\lambda\right)-a_{\left(0\right)}\left(y,\eta,\lambda\right)\right|}
\nonumber\\
&\le&
\sum_{j=1}^{n}\left(\int_{0}^{1}s^{-\tau}\left|x_{j}\right|\left|\chi\left(s\eta+ts^{\tau}\xi,s\lambda\right)\left(\partial_{x_{j}}a_{\left(0\right)}\right)\left(y+ts^{-\tau}x,s\eta+ts^{\tau}\xi,s\lambda\right)\right|dt\right.
\nonumber\\
&&
\nonumber+
\left.s^{\tau}\int_{0}^{1}\left|\xi_{j}\partial_{\xi_{j}}\left(\chi a_{\left(0\right)}\right)\left(y+ts^{-\tau}x,s\eta+ts^{\tau}\xi,s\lambda\right)\right|dt\right)\\
&&\le\sum_{j=1}^{n}\Big(C_{1}s^{-\tau}\left|x_{j}\right|+C_{2}\frac{s^{2\tau}}{\left\langle s\eta,s\lambda\right\rangle }\left|\xi_{j}\right|\left\langle \xi\right\rangle \Big).\label{eq:estimativa2}
\end{eqnarray}
The estimates (\ref{eq:estimativa1}) and (\ref{eq:estimativa2})
imply (\ref{eq:estamenosazero}) for $\tau<\frac{1}{3}$.

In a {\bf second step} we are going to show the pointwise convergence of the integrand of (\ref{eq:Lpconvergence})
for all $u\in\mathcal{S}\left(\mathbb{R}^{n}\right)$ and  all
$x\in\mathbb{R}^{n}$. We know that
\begin{eqnarray}
\lefteqn{s^{r}\left(op\left(a_{s}\right)\left(\lambda\right)u\left(x\right)-a_{\left(0\right)}\left(y,\eta,\lambda\right)u\left(x\right)\right)}
\nonumber\\
&=&
(2\pi)^{-n}\int_{\mathbb{R}^{n}}e^{ix\xi}s^{r}\left(a\left(y+s^{-\tau}x,s\eta+s^{\tau}\xi,s\lambda\right)-a_{\left(0\right)}\left(y,\eta,\lambda\right)\right)\hat u \left(\xi\right)d\xi.
\nonumber
\end{eqnarray}
The integrand goes to zero, as we have seen in Equation (\ref{eq:estamenosazero}).
Moreover
\[
\left|s^{r}\left(a\left(y+s^{-\tau}x,s\eta+s^{\tau}\xi,s\lambda\right)-a_{\left(0\right)}\left(y,\eta,\lambda\right)\right)\hat u\left(\xi\right)\right|\le C_{\lambda,\eta}\left\langle x\right\rangle \left\langle \xi\right\rangle ^{2}\left|\hat{u}\left(\xi\right)\right|,
\]
is integrable with respect to $\xi$, so that the dominated convergence theorem applies.

In the {\bf third step} we will finally prove  \eqref{eq:Lpconvergence}. It is enough
to show that the integrand is dominated. Indeed, 
integration by parts shows that 
\begin{eqnarray}
\lefteqn{
s^{r}x^{\gamma}\left(op\left(a_{s}\right)\left(\lambda\right)u\left(x\right)-a_{\left(0\right)}\left(y,\eta,\lambda\right)u\left(x\right)\right)\label{eq:(*)}
\nonumber}\\
&=&\left(-1\right)^{|\gamma|}\sum_{\sigma\le\gamma}
\binom\gamma\sigma
s^{r}(2\pi)^{-n}\int_{\mathbb{R}^{n}}e^{ix\xi}
D_{\xi}^{\sigma}(a(y+s^{-\tau}x,s\eta+s^{\tau}\xi,s\lambda)
\\
&&-a_{(0)}(y,\eta,\lambda))D_{\xi}^{\gamma-\sigma}\hat{u}(\xi)d\xi.
\nonumber 
\end{eqnarray}
For $\sigma=0$, we recall  (\ref{eq:estamenosazero}); for
$\sigma\ne0$, we use that $r+2\tau\left|\sigma\right|-\left|\sigma\right|<0$
and obtain
\[
s^{r}\left|D_{\xi}^{\sigma}\left(a\left(y+s^{-\tau}x,s\eta+s^{\tau}\xi,s\lambda\right)\right)\right|\le C\left|\left(\eta,\lambda\right)\right|^{-\left|\sigma\right|}\left\langle \xi\right\rangle ^{\left|\sigma\right|}.
\]
As $\xi\mapsto \langle \xi\rangle ^{M}\hat{u}\left(\xi\right)$
is integrable for all $M>0$,  \eqref{eq:(*)}
can be estimated by $\tilde{C}_{\lambda,\eta,\gamma}\left\langle x\right\rangle $.
Hence, for arbitrary $N$, 
\[
s^{r}\left|op\left(a_{s}\right)\left(\lambda\right)u\left(x\right)-a_{\left(0\right)}\left(y,\eta,\lambda\right)u\left(x\right)\right|\le C_{\lambda,\eta,N}\left\langle x\right\rangle ^{-N}.
\]
The dominated convergence then shows the desired $L_p$-convergence.

Our next goal is to show $L_p$-convergence of the derivative:
\begin{equation}
\lim_{s\to\infty}s^{r}\left\Vert op(a_{s})\left(\lambda\right)u-a_{(0)}(y,\eta,\lambda)u\right\Vert _{H_{p}^{1}\left(\mathbb{R}^{n}\right)}=0,\,\text{ }\,u\in\mathcal{S}\left(\mathbb{R}^{n}\right).\label{eq:Convergencia H1}
\end{equation}
Let us first observe that
\[
\partial_{x_{j}}op(a_{s})\left(\lambda\right)=op\left(a_{s}\right)\left(\lambda\right)\partial_{x_{j}}u+s^{-\tau}op\left(\left(\partial_{x_{j}}a\right)_{s}\right)\left(\lambda\right)u.
\]
Using Equation (\ref{eq:Lpconvergence}) and the fact that $r<\tau$,
we conclude that 
\[
\lim_{s\to\infty}s^{r}\left\Vert op\left(a_{s}\right)\left(\lambda\right)\partial_{x_{j}}u-a_{(0)}(y,\eta,\lambda)\partial_{x_{j}}u\right\Vert _{L_{p}\left(\mathbb{R}^{n}\right)}=0
\]
and
\begin{eqnarray*}
\lefteqn{\lim_{s\to\infty}s^{r}\left\Vert s^{-\tau}op\left(\left(\partial_{x_{j}}a\right)_{s}\right)\left(\lambda\right)u\right\Vert _{L_{p}\left(\mathbb{R}^{n}\right)}}\\ 
& \le & \lim_{s\to\infty}s^{r-\tau}\left\Vert op\left(\left(\partial_{x_{j}}a\right)_{s}\right)\left(\lambda\right)u-\left(\partial_{x_{j}}a\right){}_{(0)}(y,\eta,\lambda)u\right\Vert _{L_{p}\left(\mathbb{R}^{n}\right)}\\
 &  & +\lim_{s\to\infty}s^{r-\tau}\left\Vert \left(\partial_{x_{j}}a\right){}_{(0)}(y,\eta,\lambda)u\right\Vert _{L_{p}\left(\mathbb{R}^{n}\right)}=0
\end{eqnarray*}
for all $u\in\mathcal{S}\left(\mathbb{R}^{n}\right)$. Hence 
\begin{equation}
\lim_{s\to\infty}s^{r}\left\Vert \partial_{x_{j}}op(a_{s})\left(\lambda\right)u-a_{(0)}(y,\eta,\lambda)\partial_{x_{j}}u\right\Vert _{L_{p}\left(\mathbb{R}^{n}\right)}=0.\label{eq:convergenciaderivada}
\end{equation}
Equations (\ref{eq:Lpconvergence}) and (\ref{eq:convergenciaderivada})
imply (\ref{eq:Convergencia H1}).

In order to finish the proof of item (7), choose $\theta>0$ such that $\theta+r<\tau$.
Then item (3) implies that
\begin{eqnarray*}\lefteqn{
s^{r}\left\Vert op(a)\left(s\lambda\right)R_{s}u-a_{(0)}(y,\eta,\lambda)R_{s}u\right\Vert _{B_{p}^{0}\left(\mathbb{R}^{n}\right)}}\\
&\le& s^{r}\left\Vert R_{s}\left(R_{s}^{-1}op(a)\left(s\lambda\right)R_{s}u-a_{(0)}(y,\eta,\lambda)u\right)\right\Vert _{B_{p}^{\theta}\left(\mathbb{R}^{n}\right)}\\
&\le& C_{\theta}\left(1+s\left\langle \eta\right\rangle \right)^{\theta}s^{r}\left\Vert op(a_{s})\left(\lambda\right)u-a_{(0)}(y,\eta,\lambda)u\right\Vert _{H_{p}^{1}\left(\mathbb{R}^{n}\right)}.
\end{eqnarray*}
As the last term goes to zero, we obtain (\ref{eq:ConvergenciaBp}).
\end{proof}

\begin{cor}
\label{cor:simbolo rmais}Let $a\in S_{cl}^{0}(\mathbb{R}^{n}\times\mathbb{R}^{n},\Lambda)$
satisfy the transmission condition and $u\in\mathcal{S}({\mathbb{R}_{+}^{n}})$.
Then
\[
\lim_{s\to\infty}s^{r}\left\Vert r^{+}op(a)\left(s\lambda\right)R_{s}\left(e^{+}u\right)-a_{(0)}(y,\eta,\lambda)r^{+}R_{s}\left(e^{+}u\right)\right\Vert _{L_{p}(\mathbb{R}_{+}^{n})}=0,
\]
 for $\left(y,\eta,\lambda\right)\in\overline{\mathbb{R}_{+}^{n}}\times\left(\left(\mathbb{R}^{n}\times\Lambda\right)\backslash\left\{ 0\right\} \right)$
and $0<r<\tau$, where $R_{s}=R_{s}\left(y,\eta\right)$.
\end{cor}

\begin{proof}
We use that  $r^{+}:L_{p}\left(\mathbb{R}^{n}\right)\to L_{p}\left(\mathbb{R}_{+}^{n}\right)$
is continuous, that $R_{s}:L_{p}\left(\mathbb{R}^{n}\right)\to L_{p}\left(\mathbb{R}^{n}\right)$
is an isometry  mapping $C_{c}^{\infty}(\mathbb{R}_{+}^{n})$
to $C_{c}^{\infty}(\mathbb{R}_{+}^{n})$, and
Equation (\ref{eq:Lpconvergence}). 
\end{proof}
In order to control the action of $R_{s}$ on Besov spaces, we recall
the equivalence of Besov norm and $L_{p}$ norm on certain subsets of $\mathcal{S}(\mathbb{R}^{n})$,
see e.g.\ \cite{LeopoldSchrohe}.

\begin{lem}
{\rm(}Besov space property{\rm)} \label{lem:LpBp}There is a constant $C>0$
such that
\[
C^{-1}\left\Vert u\right\Vert _{B_{p}^{0}\left(\mathbb{R}^{n}\right)}\le\left\Vert u\right\Vert _{L_{p}\left(\mathbb{R}^{n}\right)}\le C\left\Vert u\right\Vert _{B_{p}^{0}\left(\mathbb{R}^{n}\right)},
\]
for all $u\in\mathcal{S}'\left(\mathbb{R}^{n}\right)$ with $\mbox{supp}\, \mathcal{F}(u)\subset\cup_{k=m}^{m+2}K_{k}$
for some $m\ge0$. Here $C$ does not depend on $m$. In particular,
under these circumstances, $u\in L_{p}\left(\mathbb{R}^{n}\right)$ if and only if $ u\in B_{p}^{0}\left(\mathbb{R}^{n}\right)$.
\end{lem}
The number 2 could be replaced by a different one. We recall that
the sets $K_{j}$ were defined in Remark \ref{rem:Definicao conjuntos Kj}.
\begin{proof}
As $\varphi_{j}(\xi)=\varphi_{1}(2^{-j+1}\xi)$
for $j\ge1$, and $u=\sum_{j=m-1}^{m+3}\varphi_{j}(D)u$,
the estimate
\[
\Vert \varphi_{j}(D)u\Vert _{L^{p}\left(\mathbb{R}^{n}\right)}=\Vert \mathcal{F}^{-1}(\varphi_{j})*u\Vert _{L^{p}(\mathbb{R}^{n})}
\le
\Vert \mathcal{F}^{-1}\varphi_{1}\Vert _{L^{1}\left(\mathbb{R}^{n}\right)}\Vert u\Vert _{L^{p}\left(\mathbb{R}^{n}\right)},j\ge1,
\]
implies the result.
\end{proof}

The operator $R_{s}$ has important properties when acting  on functions whose Fourier transform is supported in 
$\tilde{K}:=\left\{ \xi\in\mathbb{R}^{n};\,\frac{1}{2}<\left|\xi\right|<1\right\} $.

\begin{lem}
\label{lem:Rs em supp de F na bola} There is a constant $s_{0}>0$,
that depends only on $\eta$, for which the operator $R_{s}=R_{s}\left(y,\eta\right)$
has the following properties:
\begin{enumerate}\renewcommand{\labelenumi}{\arabic{enumi})}
\item 
If $u\in\mathcal{S}'(\mathbb{R}^{n})$ and $\mbox{supp}\,(\mathcal{F}u)\subset\tilde{K}$,
then, for every $s\ge s_0$, there is an $m\in\mathbb{N}_{0}$ that depends on $s$, such
that 
$\mbox{supp}\,\mathcal{F}\left(R_{s}u\right)\subset\cup_{k=m}^{m+2}K_{k}$.

\item 
There exists a constant $C>0$ such that $C^{-1}\left\Vert u\right\Vert _{B_{p}^{0}\left(\mathbb{R}^{n}\right)}\le\left\Vert R_{s}u\right\Vert _{B_{p}^{0}\left(\mathbb{R}^{n}\right)}\le C\left\Vert u\right\Vert _{B_{p}^{0}\left(\mathbb{R}^{n}\right)}$
for all $s>s_{0}$ and all $u\in B_{p}^{0}\left(\mathbb{R}^{n}\right)$
with $\mbox{supp}\left(\mathcal{F}u\right)\subset\tilde{K}$.

\item 
For $u\in\mathcal{S}(\mathbb{R}^{n})$ with $\mbox{supp}(\mathcal{F}u)\subset\tilde{K}$,
 $\lim_{s\to\infty}R_{s}u=0$ weakly in $B_{p}^{0}\left(\mathbb{R}^{n}\right)$.
\end{enumerate}
\end{lem}

\begin{proof}
1) By item 5) of  Lemma \ref{lem:Propriedades R_s},  $\mathcal{F}\left(R_{s}u\right)\left(\xi\right)=0$,
unless $\frac{1}{2}<\left|s^{-\tau}\left(\xi-s\eta\right)\right|<1$.
If $\eta=0$, this means that $\frac{1}{2}s^{\tau}<\left|\xi\right|<s^{\tau}$.
If $\eta\ne0$, choose $s_{0}>0$ such that $2s^{\tau}<s\left|\eta\right|$,
for $s>s_{0}$. Then $\mbox{supp}\, \mathcal{F}\left(R_{s}\left(u\right)\right)\subset\left\{ \xi;\frac{1}{2}s\left|\eta\right|<\left|\xi\right|<2s\left|\eta\right|\right\} $,
for $s>s_{0}$. The result now follows easily.

2) As $\mbox{supp}\,\mathcal{F}(R_{s}u)\subset\cup_{k=m}^{m+2}K_{k}$
and $\mbox{supp}\left(\mathcal{F}\left(u\right)\right)\subset \tilde K$,
the result follows from Lemma \ref{lem:LpBp} and the fact that $R_{s}$
is an isometry in $L_{p}\left(\mathbb{R}^{n}\right)$.

3) From item 2) of Lemma \ref{lem:Propriedades R_s}, we know that
\[
\lim_{s\to\infty}\int_{\mathbb{R}^{n}}R_{s}u\left(x\right)v\left(x\right)dx=0,\,v\in\mathcal{S}\left(\mathbb{R}^{n}\right).
\]
However, $B_{q}^{0}\left(\mathbb{R}^{n}\right)\cong B_{p}^{0}\left(\mathbb{R}^{n}\right)'$,
for $\frac{1}{p}+\frac{1}{q}=1$, and $\mathcal{S}\left(\mathbb{R}^{n}\right)$
is dense in $B_{q}^{0}\left(\mathbb{R}^{n}\right)$. 
As, by item 2),  $\left\Vert R_{s}u\right\Vert _{B_{p}^{0}\left(\mathbb{R}^{n}\right)}$
is uniformly bounded in $s$ for all fixed $u\in\mathcal{S}\left(\mathbb{R}^{n}\right)$
such that $\mbox{supp}\left(\mathcal{F}u\right)\subset\tilde{K}$,
the result follows.
\end{proof}
We now prove Theorem \ref{thm:Spectral-Invariance-in-Besov-Spaces}.
The next simple lemma will be useful:
\begin{lem}
\label{lem:injetora-adj-implica-iso} Let $E$ and $F$ be Banach
spaces and $E'$ and $F'$ be their dual spaces. If $A:E\to F$ is
a bounded linear operator such that $A$ is injective, has closed
range and its adjoint $A^{*}:F'\to E'$ is also injective, then $A$
is an isomorphism.
\end{lem}
\begin{proof}
Suppose that the range $R\left(A\right)$ of $A$ is a proper subset
of $F$. By the Hahn-Banach Theorem, there is an $f\in F^{*}$, $f\ne0$,
such that $\left.f\right|_{R\left(A\right)}=0$. This implies that
$A^{*}\left(f\right)=f\circ A=0$. As $A^{*}:F'\to E'$ is injective,
we conclude that $f=0$, which is a contradiction.
\end{proof}
\begin{proof}
(of Theorem \ref{thm:Spectral-Invariance-in-Besov-Spaces})

As it suffices to prove the implication $iii)\implies i)$, consider
$A\left(\lambda\right)$, $B\left(\lambda\right)$ and $K\left(\lambda\right)$
as in $iii)$. 
Our aim is to prove that the principal symbol $p_{\left(0\right)}\left(z,\lambda\right)$
of $A$ is invertible for every $\left(z,\lambda\right)\in\left(T^{*}M\times\Lambda\right)\backslash\left\{ 0\right\} $.
We focus on a trivializing coordinate neighborhood $U$ containing $x=\pi(z)$. 
We choose smooth functions $\Phi, \Psi$ and $\mathrm H$ supported in $U$ 
such that $\Phi$ equals $1$ near $x$ and $\Psi\Phi= \Phi$, $\mathrm H\Psi=\Psi$.  
Denote by 
$\tilde A(\lambda)\in 
\mathcal B(B^0_p(\mathbb R^n)^{N_1}, B^0_p(\mathbb R^n)^{N_2})$ and $\tilde B(\lambda)\in \mathcal B(B^0_p(\mathbb R^n)^{N_2}, B^0_p(\mathbb R^n)^{N_1})$ the operators $\mathrm H A(\lambda)\Psi$ and $\Phi B(\lambda)\mathrm H$ in local coordinates.  
Then our assumptions imply that there are compact operators $\tilde K(\lambda)$, tending to zero in $\mathcal B(B^0_p(\mathbb R^n)^{N_1})$ as $|\lambda|\to \infty$ 
such that 
\begin{eqnarray}\label{eq:vaphipsi igual BA menos K}
\tilde B(\lambda)\tilde A(\lambda) = \tilde \Phi + \tilde K(\lambda),
\end{eqnarray}   
where  $\tilde \Phi$ is $\Phi$ in  local coordinates. 
Here we use the fact that $\tilde B(\lambda)$ has logarithmic growth and that 
$\Phi B(\lambda)\mathrm H^2A(\lambda) \Psi$ differs from 
$\Phi B(\lambda)A(\lambda)\mathrm \Psi$ by a compact operator whose norm tends to zero as $|\lambda| \to \infty$. 

Denote by $(y,\eta,\lambda)\in \mathbb R^n\times (\mathbb R^n\times \Lambda)\setminus \{0\}$ the point corresponding to $(z,\lambda)$ and fix an element 
$u=cv\in\mathcal{S}\left(\mathbb{R}^{n}\right)^{N_{1}}$,
where $c\in\mathbb{C}^{N_{1}}$
and $0\ne v\in\mathcal{S}\left(\mathbb{R}^{n}\right)$ with  
$\mbox{supp}\left(\mathcal{F}v\right)\subset\left\{ \xi;\,\frac{1}{2}<|{\xi}|<1\right\} $.
Equation (\ref{eq:vaphipsi igual BA menos K}) together with item
$ii)$ of Lemma \ref{lem:Rs em supp de F na bola} implies that
\begin{eqnarray}
\lefteqn{\left\Vert u\right\Vert _{B_{p}^{0}\left(\mathbb{R}^{n}\right)^{N_{1}}}\le C\left(\Vert \tilde{B}\left(s\lambda\right)\Vert _{\mathcal{B}\left(B_{p}^{0}\left(\mathbb{R}^{n}\right)^{N_{2}},B_{p}^{0}\left(\mathbb{R}^{n}\right)^{N_{1}}\right)}\Vert \tilde{A}\left(s\lambda\right)R_{s}u\Vert _{B_{p}^{0}\left(\mathbb{R}^{n}\right)^{N_{2}}}\right.}
\nonumber\\
&&+
\left.\Vert \tilde{K}\left(s\lambda\right)R_{s}u\Vert _{B_{p}^{0}\left(\mathbb{R}^{n}\right)^{ N_{1}}}+\Vert (1-\tilde \Phi)R_{s}u\Vert _{B_{p}^{0}\left(\mathbb{R}^{n}\right)^{ N_{1}}}\right).\label{eq:speccomp}
\end{eqnarray}
We claim that $\lim_{s\to\infty}\Vert \tilde{K}\left(s\lambda\right)R_{s}u\Vert _{B_{p}^{0}\left(\mathbb{R}^{n}\right)^{ N_{1}}}=0$: 
Indeed  $\Vert \tilde{K}\left(s\lambda\right)\Vert _{\mathcal{B}(B_{p}^{0}\left(\mathbb{R}^{n})^{ N_{1}}\right)}\to0$  for $\lambda\ne0$,
and $\Vert R_{s}u\Vert _{B_{p}^{0}\left(\mathbb{R}^{n}\right)^{ N_{1}}}\le C\left\Vert u\right\Vert _{B_{p}^{0}\left(\mathbb{R}^{n}\right)^{N_{1}}}$.
For $\lambda=0$, we use that $\tilde{K}\left(0\right)$
is compact and  the third item of Lemma \ref{lem:Rs em supp de F na bola},
which implies that $\lim_{s\to\infty}R_{s}u=0$ weakly in $B_{p}^{0}\left(\mathbb{R}^{n}\right)^{N_{1}}$.

Since $\tilde \Phi\in C_{c}^{\infty}\left(\mathbb{R}^{n}\right)$
is equal to $1$ in a neighborhood of $y$, 
$\lim_{s\to\infty}(1-\tilde \Phi)R_{s}(y,\eta)u=0$
in the topology of $\mathcal{S}\left(\mathbb{R}^{n}\right)$ and,
therefore, also in the topology of $B_{p}^{0}\left(\mathbb{R}^{n}\right)$.
We moreover estimate
\begin{eqnarray}
\lefteqn{\Vert \tilde{A}\left(s\lambda\right)R_{s}u\Vert _{B_{p}^{0}\left(\mathbb{R}^{n}\right)^{N_{2}}}
\le
\Vert \tilde{A}\left(s\lambda\right)R_{s}u-p_{\left(0\right)}\left(y,\eta,\lambda\right)R_{s}u\Vert _{B_{p}^{0}\left(\mathbb{R}^{n}\right)^{ N_{2}}}
}
\nonumber\\
&&
+C\left\Vert p_{\left(0\right)}\left(y,\eta,\lambda\right)c\right\Vert _{\mathcal{B}\left(\mathbb{C}^{N_{1}},\mathbb{C}^{N_{2}}\right)}\Vert v\Vert _{B_{p}^{0}(\mathbb{R}^{n})}.
\nonumber
\end{eqnarray}
Item 7 of Lemma \ref{lem:Propriedades R_s} implies that 
$\lim_{s\to\infty}s^{r}\Vert \tilde{A}\left(s\lambda\right)R_{s}u-p_{\left(0\right)}\left(y,\eta,\lambda\right)R_{s}u\Vert _{B_{p}^{0}\left(\mathbb{R}^{n}\right)^{N_{2}}}=0$ 
for $r$ sufficiently small. 
By assumption, $\Vert \tilde{B}(s\lambda)\Vert _{\mathcal{B}(B_{p}^{0}(\mathbb{R}^{n})^{N_{2}},B_{p}^{0}(\mathbb{R}^{n})^{N_{1}})}
\le\tilde{C}\langle \ln(s\lambda)\rangle ^{M}$.
Taking $s$ sufficiently large, we conclude that 
\begin{eqnarray}
\lefteqn{C\Big(\Vert \tilde{B}\left(s\lambda\right)\Vert _{\mathcal{B}\left(B_{p}^{0}\left(\mathbb{R}^{n}\right)^{ N_{2}},B_{p}^{0}\left(\mathbb{R}^{n}\right)^{ N_{1}}\right)}\Vert \tilde{A}\left(s\lambda\right)R_{s}u-p_{\left(0\right)}\left(y,\eta,\lambda\right)R_{s}u\Vert _{B_{p}^{0}\left(\mathbb{R}^{n}\right)^{ N_{2}}}}
\nonumber\\
&&+
\Vert \tilde{K}\left(s\lambda\right)R_{s}u\Vert _{B_{p}^{0}\left(\mathbb{R}^{n}\right)^{N_{1}}}+\Vert (1-\tilde \Phi)R_{s}u\Vert _{B_{p}^{0}\left(\mathbb{R}^{n}\right)^{N_{1}}}\Big)
\le\frac12\Vert u\Vert _{B_{p}^{0}\left(\mathbb{R}^{n}\right)^{N_{1}}}.
\nonumber
\end{eqnarray}
Hence, for sufficiently large $s$, we have
\[
\left\Vert c\right\Vert _{\mathbb{C}^{N_{1}}}\left\Vert v\right\Vert _{B_{p}^{0}\left(\mathbb{R}^{n}\right)}=\left\Vert u\right\Vert _{B_{p}^{0}\left(\mathbb{R}^{n}\right)^{N_{1}}}
\le
\tilde{C}\left\langle \ln\left(s\lambda\right)\right\rangle ^{M}\left\Vert p_{\left(0\right)}\left(y,\eta,\lambda\right)c\right\Vert _{\mathcal{B}(\mathbb{C}^{N_{1}},\mathbb{C}^{N_{2}})}\left\Vert v\right\Vert _{B_{p}^{0}\left(\mathbb{R}^{n}\right)}.
\]
As $v\ne0$, this clearly implies that $p_{\left(0\right)}\left(y,\eta,\lambda\right)$
is injective.

An analogous argument applies to the adjoint operator. We conclude
that $p_{\left(0\right)}\left(y,\eta,\lambda\right)^{*}$, that is,
the adjoint of $p_{\left(0\right)}\left(y,\eta,\lambda\right)$ and
the principal symbol of $A\left(\lambda\right)^{*}$, is also injective.
Lemma \ref{lem:injetora-adj-implica-iso} then tells us that $p_{\left(0\right)}\left(y,\eta,\lambda\right)$
is an isomorphism and, in particular, that $N_{2}=N_{1}$. Therefore
$A\left(\lambda\right)$ is an elliptic operator.
\end{proof}

\subsubsection{Boutet de Monvel operators with parameters acting on $L_{p}$-spaces}

\begin{thm}
\label{thm:equivalenciaFredholmelipticBdM}Let $M$ be a compact manifold
with boundary $\partial M$. Let $E_{0}$ and $E_{1}$ be vector bundles
over $M$, $F_{0}$ and $F_{1}$ be vector bundles over $\partial M$
and $A\in\tilde{\mathcal{B}}^p_{E_{0},F_{0},E_{1},F_{1}}\left(M,\Lambda\right)$.
Then the following conditions are equivalent:
\begin{enumerate} \renewcommand{\labelenumi}{\roman{enumi})}
\item 
The operator $A\left(\lambda\right)$ is an elliptic parameter-dependent
operator.

\item 
We find bounded operators $B_{1}\left(\lambda\right):L^{p}\left(M,E_{0}\right)\oplus B_{p}^{0}\left(M,F_{0}\right)\to L^{p}\left(M,E_{1}\right)\oplus B_{p}^{0}\left(M,F_{1}\right)$
and $B_{2}\left(\lambda\right):L^{p}\left(M,E_{1}\right)\oplus B_{p}^{0}\left(M,F_{1}\right)\to L^{p}\left(M,E_{0}\right)\oplus B_{p}^{0}\left(M,F_{0}\right)$
such that
\[
B_{1}\left(\lambda\right)A\left(\lambda\right)=1+K_{1}\left(\lambda\right)\,\,\,\text{and}\,\,\,A\left(\lambda\right)B_{2}\left(\lambda\right)=1+K_{2}\left(\lambda\right),\,\lambda\in\Lambda,
\]
where the $B_{j}\left(\lambda\right)$ are uniformly bounded in $\lambda$
and $K_{1}\left(\lambda\right):L^{p}\left(M,E_{0}\right)\oplus B_{p}^{0}\left(M,F_{0}\right)\to L^{p}\left(M,E_{0}\right)\oplus B_{p}^{0}\left(M,F_{0}\right)$
and $K_{2}\left(\lambda\right):L^{p}\left(M,E_{1}\right)\oplus B_{p}^{0}\left(M,F_{1}\right)\to L^{p}\left(M,E_{1}\right)\oplus B_{p}^{0}\left(M,F_{1}\right)$
are compact and  $\lim_{\left|\lambda\right|\to\infty}K_{j}\left(\lambda\right)=0$,
$j=1,2$.

\item 
Condition $ii)$ holds with the uniform boundedness of the $B_{j}\left(\lambda\right)$
replaced by the condition  that, for $j=1,2$ and some $M\in \mathbb{N}_0$,
\[
\left\Vert B_{j}\left(\lambda\right)\right\Vert _{\mathcal{B}\left(L^{p}\left(M,E_{1}\right)\oplus B_{p}^{0}\left(M,F_{1}\right),L^{p}\left(M,E_{0}\right)\oplus B_{p}^{0}\left(M,F_{0}\right)\right)}\le C\left\langle \ln\left(\lambda\right)\right\rangle ^{M}.
\]
\end{enumerate} 
\end{thm}

\begin{rem}
Let $A^{*}\left(\lambda\right)$ be the adjoint operator of $A\left(\lambda\right)$.
Theorem \ref{thm:Propriedades de BdM} tells us that $A\left(\lambda\right)B_{2}\left(\lambda\right)=1+K_{2}\left(\lambda\right)$
is equivalent to
\[
B_{2}^{*}\left(\lambda\right)A^{*}\left(\lambda\right)=1+K_{2}^{*}\left(\lambda\right),
\]
which is the condition that we will need later.
\end{rem}
Again a standard parametrix construction shows that $i)$ implies
$ii)$. As $ii)$ implies $iii)$ trivially, we only have to prove
that $iii)$ implies $i)$.

We fix a point $\left(y,\eta\right)\in\mathbb{R}^{n-1}\times\mathbb{R}^{n-1}$
and a constant $0<\tau<\frac{1}{3}$. For every $s>0$, we define
the isometries $R_{s}=R_{s}\left(y,\eta\right):L^{p}\left(\mathbb{R}^{n-1}\right)\to L^{p}\left(\mathbb{R}^{n-1}\right)$,
$S_{s}:L^{p}\left(\mathbb{R}_{+}\right)\to L^{p}\left(\mathbb{R}_{+}\right)$
and $R_{s}\otimes S_{s}:L^{p}\left(\mathbb{R}_{+}^{n}\right)\to L^{p}\left(\mathbb{R}_{+}^{n}\right)$ by
\begin{eqnarray*}
R_{s}v\left(x'\right)&=&s^{\frac{\tau\left(n-1\right)}{p}}e^{isx'\eta}v\left(s^{\tau}\left(x'-y\right)\right),\\
S_{s}w\left(x_{n}\right)&=&s^{\frac{1}{p}}w\left(sx_{n}\right),\\
R_{s}\otimes S_{s}u\left(x\right)&=&s^{\frac{\tau\left(n-1\right)}{p}}s^{\frac{1}{p}}e^{isx'\eta}v\left(s^{\tau}\left(x'-y\right),sx_{n}\right).
\end{eqnarray*}
The following simple proposition will be useful. It is very similar
to the results we have already seen.

\begin{prop}
The operator $R_{s}\otimes S_{s}:L^{p}\left(\mathbb{R}_{+}^{n}\right)\to L^{p}\left(\mathbb{R}_{+}^{n}\right)$
satisfies:
\begin{enumerate} \renewcommand{\labelenumi}{\arabic{enumi})}
\item  $\left\Vert R_{s}\otimes S_{s}u\right\Vert _{L^{p}\left(\mathbb{R}_{+}^{n}\right)}=\left\Vert u\right\Vert _{L^{p}\left(\mathbb{R}_{+}^{n}\right)}$,
$u\in L^{p}\left(\mathbb{R}_{+}^{n}\right)$.

\item 
$\lim_{s\to\infty}R_{s}\otimes S_{s}u=0$ in the weak topology
of $L^{p}\left(\mathbb{R}_{+}^{n}\right)$.
\end{enumerate}
\end{prop}

\begin{proof}
1) is easily verified.

2) Due to the first item and the fact that $L^{q}\left(\mathbb{R}_{+}^{n}\right)\cong L^{p}\left(\mathbb{R}_{+}^{n}\right)'$,
it is enough to prove that if $u\left(x\right)=u_{1}\left(x'\right)u_{2}\left(x_{n}\right)$
and $v\left(x\right)=v_{1}\left(x'\right)v_{2}\left(x_{n}\right)$,
where $u_{1},v_{1}\in C_{c}^{\infty}\left(\mathbb{R}^{n-1}\right)$
and $u_{2},v_{2}\in C_{c}^{\infty}\left(\overline{\mathbb{R}_{+}}\right)$,
then
\begin{eqnarray*}
\lefteqn{\lim_{s\to\infty}\int_{\mathbb{R}_{+}^{n}}R_{s}\otimes S_{s}u\left(x\right)v\left(x\right)dx}\\
&=&\lim_{s\to\infty}\left(\int_{\mathbb{R}^{n-1}}R_{s}u_{1}\left(x'\right)v_{1}\left(x'\right)dx'\right)\left(\int_{\mathbb{R}_{+}}S_{s}u_{2}\left(x_{n}\right)v_{2}\left(x_{n}\right)dx_{n}\right)=0.
\end{eqnarray*}
A simple computation shows that both terms on the right hand side
go to zero as $s\to\infty$.
\end{proof}

\begin{prop}
\label{prop:PseudoGreenPoissonTrace}Let $0<r<\tau$ and let $v\in\mathcal{S}\left(\mathbb{R}^{n-1}\right)$
be such that $\mathcal{F}\left(v\right)$ has compact support. 
Denote by $C\left(s\right)$ a function such that $\lim_{s\to\infty}C\left(s\right)=0$.
Then 

{\rm1)} {\rm (}Pseudodifferential operator in the interior{\rm)} Let $p\in S_{cl}^{0}\left(\mathbb{R}^{n}\times\mathbb{R}^{n},\Lambda\right)$
satisfy the transmission condition and $p\sim\sum_{j\in\mathbb{N}_{0}}p_{\left(-j\right)}$
be its asymptotic expansion. Then 
\begin{eqnarray*}\lefteqn{s^{r}\Big\Vert op\left(p\right)\left(s\lambda\right)\left(R_{s}v\otimes S_{s}w\right)}\\
&&-R_{s}\otimes S_{s}
\Big(\frac{v\left(x'\right)}{2\pi}\int_{\mathbb{R}}e^{ix_{n}\xi_{n}}p_{\left(0\right)}\left(y,x_{n},\eta,\xi_{n},\lambda\right)\mathcal{F}_{x_{n}\to\xi_{n}}\left(e^{+}w\right)\left(\xi_{n}\right)d\xi_{n}\Big)\Big\Vert _{L_{p}\left(\mathbb{R}_{+}^{n}\right)}\\
&\le& C\left(s\right)\left\Vert w\right\Vert _{L^{p}\left(\mathbb{R}_{+}\right)}, 
\quad w\in\mathcal{S}(\mathbb{R}_{+}).
\end{eqnarray*}

{\rm2)} {\rm (}Singular Green operators{\rm)} Let $ S_{cl}^{-1}(\mathbb{R}^{n-1},\mathcal{S}_{++},\Lambda)
\ni\tilde{g}\sim\sum_{j\in\mathbb{N}_{0}}\tilde{g}_{\left(-1-j\right)}$
and $G\left(\lambda\right):\mathcal{S}(\mathbb{R}_{+}^{n})\to\mathcal{S}(\mathbb{R}_{+}^{n})$
be defined by \eqref{eq:DefGreen}. 
Then, for $w\in\mathcal{S}\left(\mathbb{R}_{+}\right)$,
\begin{eqnarray*}\lefteqn{s^{r}\Big\Vert G\left(s\lambda\right)\left(R_{s}v\otimes S_{s}w\right)}\\
&&-R_{s}\otimes S_{s}\Big(v(x')\int_{\mathbb{R}_{+}}\tilde{g}_{\left(-1\right)}\left(y,x_{n},y_{n},\eta,\lambda\right)w(y_{n})dy_{n}\Big)\Big\Vert _{L_{p}\left(\mathbb{R}_{+}^{n}\right)}\le C\left(s\right)\Vert w\Vert _{L^{p}\left(\mathbb{R}_{+}\right)}.
\end{eqnarray*}

{\rm3)} {\rm(}Trace operators{\rm)} Let $S_{cl}^{-\frac{1}{p}}(\mathbb{R}^{n-1},\mathcal{S}_{+},\Lambda)\ni\tilde{t}\sim\sum_{j\in\mathbb{N}_{0}}\tilde{t}_{\left(-\frac{1}{p}-j\right)}$
and $T\left(\lambda\right):\mathcal{S}\left(\mathbb{R}_{+}^{n}\right)\to\mathcal{S}\left(\mathbb{R}^{n-1}\right)$
be defined by \eqref{eq:DefTrace}. Then for $w\in\mathcal{S}\left(\mathbb{R}_{+}\right)$,
\begin{eqnarray*}
&s^{r}\!\!\!&\Big\Vert T\left(s\lambda\right)\left(R_{s}v\otimes S_{s}w\right)
-R_{s}\Big(v\left(x'\right)\int_{\mathbb{R}_{+}}\tilde{t}_{\left(-\frac{1}{p}\right)}\left(y,x_{n},\eta,\lambda\right)w\left(x_{n}\right)dx_{n}\Big)\Big\Vert _{B_{p}^{0}(\mathbb{R}^{n-1})}\\
&&\le C\left(s\right)\left\Vert w\right\Vert _{L^{p}\left(\mathbb{R}_{+}\right)}.
\end{eqnarray*}

{\rm4)} {\rm(}Poisson operators{\rm)} Let 
$S_{cl}^{\frac{1}{p}-1}(\mathbb{R}^{n-1},\mathcal{S}_{+},\Lambda)
\ni\tilde{k}
\sim\sum_{j\in\mathbb{N}_{0}}\tilde{k}_{\left(\frac{1}{p}-1-j\right)}$ and $K\left(\lambda\right):\mathcal{S}(\mathbb{R}^{n-1})\to\mathcal{S}(\mathbb{R}_{+}^{n})$
be defined by (\ref{eq:DefPoisson}). Then for $w\in\mathcal{S}\left(\mathbb{R}_{+}\right)$,
\begin{eqnarray*}
\lim_{s\to\infty}s^{r}\left\Vert K\left(s\lambda\right)\left(R_{s}v\right)-R_{s}\otimes S_{s}\left(\tilde{k}_{\left(\frac{1}{p}-1\right)}\left(y,x_{n},\eta,\lambda\right)v\left(x'\right)\right)\right\Vert _{L_{p}\left(\mathbb{R}_{+}^{n}\right)}=0.
\end{eqnarray*}
\end{prop}
\begin{proof}
The items 1), 2) and 4) extend the results 
in \cite[Section 2.3.4.2]{RempelSchulze}.
They can be obtained by replacing the operators $R_{s}$ and $S_{s}$
in \cite{RempelSchulze} by the definitions given here and arguing similarly as for the third item. 

The
third item is more delicate, as the limit is taken in the Besov space:
Let $q$ be such that $\frac{1}{p}+\frac{1}{q}=1$. Using item
4, 5 and 6 of Lemma \ref{lem:Propriedades R_s},
we find  that
\begin{eqnarray*}\lefteqn{R_{s}^{-1}T\left(s\lambda\right)\left(R_{s}v\otimes S_{s}w\right)\left(x'\right)}\\
&=&\int_{\mathbb{R}^{n-1}}e^{ix'\xi'}\Big(\int_{\mathbb{R}_{+}}s^{-\frac{1}{q}}\tilde{t}\left(y+s^{-\tau}x',\frac{x_{n}}{s},s\eta+s^{\tau}\xi',s\lambda\right)\hat v(\xi')w\left(x_{n}\right)dx_{n}\Big)d\xi'.
\end{eqnarray*}
Fix $\left(y,\eta,\lambda\right)\in\mathbb{R}^{n-1}\times\mathbb{R}^{n-1}\times\Lambda$
such that $\left(\eta,\lambda\right)\ne\left(0,0\right)$. We will
use the simple fact that if $v\in\mathcal{S}(\mathbb{R}^{n-1})$
is such that $\text{supp}\,(\mathcal{F}(v))$
is compact, then for all $\theta\in\left]0,1\right[$ and for all
$\xi'\in\text{supp}\left(\mathcal{F}\left(v\right)\right)$, there
is a $s_{0}>0$ such that
\[
C^{-1}s^{M}\left|\left(\eta,\lambda\right)\right|^{M}\le\left\langle s\eta+\theta s^{\tau}\xi',s\lambda\right\rangle ^{M}\le Cs^{M}\left|\left(\eta,\lambda\right)\right|^{M},\,s\ge s_{0}.
\]
The constant $C$ does not depend on $\theta$, $s\ge s_{0}$ and
$\xi'\in\text{supp}\left(\mathcal{F}\left(v\right)\right)$.

We start by establishing {\bf $\mathbf L_p$-convergence}:
Let $0<r<\tau$
and $v\in\mathcal{S}(\mathbb{R}^{n-1})$ with $\text{supp}\left(\mathcal{F}(v)\right)$
compact. 
Then, for all $w\in\mathcal{S}\left(\mathbb{R}_{+}\right)$,
we have
\begin{eqnarray*}&s^{r}\!\!\!&\Big\Vert R_{s}^{-1}T\left(s\lambda\right)\left(R_{s}v\otimes S_{s}w\right)-v\left(x'\right)\int_{\mathbb{R}_{+}}\tilde{t}_{\left(-\frac{1}{p}\right)}\left(y,x_{n},\eta,\lambda\right)w\left(x_{n}\right)dx_{n}\Big\Vert _{L_{p}\left(\mathbb{R}^{n-1}\right)}\\
&&\le C\left(s\right)\left\Vert w\right\Vert _{L_{p}\left(\mathbb{R}_{+}\right)},\label{eq:Lpconvergencetrace}
\end{eqnarray*}
where $C\left(s\right)$ is a constant that depends on $s$, $\left(y,\eta,\lambda\right)$
and $v$ but not on $w$. Moreover, $\lim_{s\to\infty}C\left(s\right)=0$.

We divide the proof into steps, always assuming that $s\ge s_{0}$.
First we see that
\begin{eqnarray}
\lefteqn{s^{r}\Big|R_{s}^{-1}T\left(s\lambda\right)\left(R_{s}v\otimes S_{s}w\right)-v\left(x'\right)\int_{\mathbb{R}_{+}}\tilde{t}_{\left(-\frac{1}{p}\right)}\left(y,x_{n},\eta,\lambda\right)w\left(x_{n}\right)dx_{n}\Big|}
\nonumber\\
&&\le\left\Vert w\right\Vert _{L^{p}\left(\mathbb{R}_{+}\right)}
\Big(\int_{\mathbb{R}_{+}}\Big|\Big(\int_{\mathbb{R}^{n-1}}e^{ix'\xi'}s^{r-\frac{1}{q}}\Big(\tilde{t}(y+s^{-\tau}x',\frac{x_{n}}{s},s\eta+s^{\tau}\xi',s\lambda)\\
&&-s^{\frac{1}{q}}\tilde{t}_{\left(-\frac{1}{p}\right)}\left(y,x_{n},\eta,\lambda\right)\Big)\hat{v}\left(\xi'\right)d\xi'\Big)\Big|^{q}dx_{n}\Big)^{\frac{1}{q}}.\label{eq:tracemain}
\nonumber
\end{eqnarray}

In a {\bf first step} we will prove that, for all $\left(x',x_{n},\xi'\right)\in\mathbb{R}^{n-1}\times\mathbb{R}_{+}\times\mathbb{R}^{n-1}$
and $M\in\mathbb{N}_{0}$, there is a constant that depends on $\eta$,
$\lambda$ and $M$ such that
\begin{eqnarray}\lefteqn{\left|s^{r-\frac{1}{q}}\left(\tilde{t}\left(y+s^{-\tau}x',\frac{x_{n}}{s},s\eta+s^{\tau}\xi',s\lambda\right)-s^{\frac{1}{q}}\tilde{t}_{\left(-\frac{1}{p}\right)}\left(y,x_{n},\eta,\lambda\right)\right)\right|}\nonumber\\
&\le& C_{\eta,\lambda,M}\left\langle x_{n}\right\rangle ^{-M}s^{r-\tau},\,\xi'\in\text{supp}\left(\mathcal{F}\left(v\right)\right)\label{eq:pointwiseconvergencetrace-1}
\end{eqnarray}

Let us fix a function $\chi\in C^{\infty}\left(\mathbb{R}^{n-1}\times\Lambda\right)$
that is zero near the origin and equal to 1 outside
a closed ball that does not contain $\left(\eta,\lambda\right)$.
We note that
\begin{eqnarray}
\lefteqn{\Big|s^{r-\frac{1}{q}}x_{n}^{M}\Big(\tilde{t}\Big(y+s^{-\tau}x',\frac{x_{n}}{s},s\eta+s^{\tau}\xi',s\lambda\Big)\label{eq:eq1}}
\\
&&-\chi\left(s\eta+s^{\tau}\xi',s\lambda\right)
\tilde{t}_{\left(-\frac{1}{p}\right)}\Big(y+s^{-\tau}x',\frac{x_{n}}{s},s\eta+s^{\tau}\xi',s\lambda\Big)\Big)\Big|
\nonumber\\
&\le&
C_{1}s^{-\frac{1}{q}+r+M}\left\langle s\eta+s^{\tau}\xi',s\lambda\right\rangle ^{-\frac{1}{p}-M}\le C_{2}s^{-1+r}\left|\left(\eta,\lambda\right)\right|^{-\frac{1}{p}-M}
\nonumber
\end{eqnarray}
for $\xi'\in\text{supp}\left(\mathcal{F}\left(v\right)\right).$
We now study the term
\begin{eqnarray}
s^{r-\frac{1}{q}+M}&&\left(\frac{x_{n}}{s}\right)^{M}\big(\chi\left(s\eta+s^{\tau}\xi',s\lambda\right)\tilde{t}_{\left(-\frac{1}{p}\right)}\left(y+s^{-\tau}x',\frac{x_{n}}{s},s\eta+s^{\tau}\xi',s\lambda\right)\nonumber\\
&&-s^{\frac{1}{q}}\tilde{t}_{\left(-\frac{1}{p}\right)}\big(y,x_{n},\eta,\lambda\big)\big)\label{eq:boundaryterm-1}
\end{eqnarray}
Using the fact that $s^{-\frac{1}{q}}\tilde{t}_{\left(-\frac{1}{p}\right)}\left(y,\frac{x_{n}}{s},s\eta,s\lambda\right)=\tilde{t}_{\left(-\frac{1}{p}\right)}\left(y,x_{n},\eta,\lambda\right)$, and a Taylor expansion
we conclude that the expression (\ref{eq:boundaryterm-1}) is smaller
or equal to
\begin{eqnarray}
\lefteqn{s^{r-\tau-\frac{1}{q}+M}
\sum_{\left|\beta\right|=1}\left|x'^{\beta}\right|
\int_{0}^{1}\Big|\frac{x_{n}}{s}\Big|^{M}\Big|\partial_{x'}^{\beta}\big(\chi\tilde{t}_{\left(-\frac{1}{p}\right)}\big)\big(y+\theta s^{-\tau}x',\frac{x_{n}}{s},s\eta+\theta s^{\tau}\xi',s\lambda\big)\Big|d\theta}
\nonumber\\
&&\!\!\!\!\!\!\!\!+
s^{r+\tau-\frac{1}{q}+M}\sum_{\left|\beta\right|=1}\big|\xi'^{\beta}\big|
\int_{0}^{1}\Big|\frac{x_{n}}{s}\Big|^{M}
\Big|\partial_{\xi'}^{\beta}\big(\chi\tilde{t}_{\big(-\frac{1}{p}\big)}\big)\big(y+\theta s^{-\tau}x',\frac{x_{n}}{s},s\eta+\theta s^{\tau}\xi',s\lambda\big)\Big|d\theta
\nonumber\\
&&\mbox{\ \ \ \ \ \ \  }\le
C\left(s^{r-\tau}\left\langle x'\right\rangle \left|\left(\eta,\lambda\right)\right|^{-\frac{1}{p}+1-M}+s^{r+\tau-1}\left\langle \xi'\right\rangle \left|\left(\eta,\lambda\right)\right|^{-\frac{1}{p}-M}\right).\label{eq:eq2}
\end{eqnarray}

As $0<r<\tau<\frac{1}{3}$, we conclude that $-1+r<r-\tau$ and $r+\tau-1<r-\tau$.
Hence (\ref{eq:pointwiseconvergencetrace-1}) follows from the estimates
of (\ref{eq:eq1}) and (\ref{eq:eq2}).

In a {\bf second step}  we will next show that the limit of Equation (\ref{eq:tracemain}) as $s\to\infty$
is zero. This is true, as it is smaller than or equal to
\[
C_{\eta,\lambda}s^{r-\tau}\int_{\mathbb{R}^{n-1}}\left|\hat{v}\left(\xi'\right)\right|d\xi'\Big(\int_{\mathbb{R}_{+}}\left\langle x_{n}\right\rangle ^{-M}dx_{n}\Big)^{\frac{1}{q}},\,M>1.
\]

In a {\bf third step} we want to prove that, for all $M\in\mathbb{N}_{0}$, the
expression (\ref{eq:tracemain}) is bounded by $C_{M}\left\langle x'\right\rangle ^{-M}$,
for a constant $C_{M}>0$. Then Lebesgue's dominated convergence theorem
will imply that (\ref{eq:Lpconvergencetrace}) holds. In order to
do that, we note that
\[
x'^{\gamma}\int
e^{ix'\xi'}s^{r-\frac{1}{q}}\left(\tilde{t}\left(y+s^{-\tau}x',\frac{x_{n}}{s},s\eta+s^{\tau}\xi',s\lambda\right)-s^{\frac{1}{q}}\tilde{t}_{\left(-\frac{1}{p}\right)}\left(y,x_{n},\eta,\lambda\right)\right)\hat{v}\left(\xi'\right)d\xi'
\]
 is a linear combination of terms of the form
\[
\int_{\mathbb{R}^{n-1}}e^{ix'\xi'}s^{r-\frac{1}{q}}D_{\xi'}^{\sigma}\left(\tilde{t}\left(y+s^{-\tau}x',\frac{x_{n}}{s},s\eta+s^{\tau}\xi',s\lambda\right)-s^{\frac{1}{q}}\tilde{t}_{\left(-\frac{1}{p}\right)}\left(y,x_{n},\eta,\lambda\right)\right)D_{\xi'}^{\gamma-\sigma}\hat{v}\left(\xi'\right)d\xi'.
\]

If $\sigma=0$, we have already proven that the above expression is
smaller than $C_{\eta,\lambda,M}\left\langle x_{n}\right\rangle ^{-M}s^{r-\tau}$.
For $\sigma\ne0$, we estimate
\begin{eqnarray}
\lefteqn{\left|s^{r-\frac{1}{q}}x_{n}^{M}D_{\xi'}^{\sigma}\left(\tilde{t}\left(y+s^{-\tau}x',\frac{x_{n}}{s},s\eta+s^{\tau}\xi',s\lambda\right)\right)\right|}
\nonumber\\
&\le&
\left|s^{r-\frac{1}{q}+\tau\left|\sigma\right|+M}\left(\frac{x_{n}}{s}\right)^{M}\left(D_{\xi'}^{\sigma}\tilde{t}\right)\left(y+s^{-\tau}x',\frac{x_{n}}{s},s\eta+s^{\tau}\xi',s\lambda\right)\right|
\nonumber\\
&\le&
C_{1}s^{r-\frac{1}{q}+\tau\left|\sigma\right|+M}\left\langle s\eta+s^{\tau}\xi',s\lambda\right\rangle ^{-\frac{1}{p}+1-M-\left|\sigma\right|}
\nonumber\\
&\le& C_{2}s^{r+\left(\tau-1\right)\left|\sigma\right|}\left|\left(\eta,\lambda\right)\right|^{-\left|\sigma\right|-\frac{1}{p}+1-M}.
\end{eqnarray}
Hence $\big|s^{r-\frac{1}{q}}D_{\xi'}^{\sigma}\left(\tilde{t}\left(y+s^{-\tau}x',\frac{x_{n}}{s},s\eta+s^{\tau}\xi',s\lambda\right)\right)\big|\le C_{\eta,\lambda,M}\left\langle x_{n}\right\rangle ^{-M}s^{r-\tau}$.
The result now follows easily.

We will next establish the $\mathbf{L_p}${\bf-convergence of the derivative}.
Let $0<r<\tau$ and $v\in\mathcal{S}\left(\mathbb{R}^{n-1}\right)$
with $\text{supp}\left(\mathcal{F}\left(v\right)\right)$  compact. Then,
for all $w\in\mathcal{S}\left(\mathbb{R}_{+}\right)$, we have
\begin{eqnarray}s^{r}\Big\Vert R_{s}^{-1}T\left(s\lambda\right)
\left(R_{s}v\otimes S_{s}w\right)-v\left(x'\right)\int_{\mathbb{R}_{+}}\tilde{t}_{\left(-\frac{1}{p}\right)}\left(y,x_{n},\eta,\lambda\right)w\left(x_{n}\right)dx_{n}\Big\Vert _{H_{p}^{1}\left(\mathbb{R}^{n-1}\right)}\nonumber\\
\le C\left(s\right)\left\Vert w\right\Vert _{L_{p}\left(\mathbb{R}_{+}\right)},\label{eq:convH1BdM}
\end{eqnarray}
where $C\left(s\right)$ is a constant that depends on $s$, $\left(y,\eta,\lambda\right)$
and $v$ but not on $w$. Moreover, $\lim_{s\to\infty}C\left(s\right)=0$.

Let us first fix a notation. We denote by $\left(\partial_{x_{j}}T\right)\left(\lambda\right)$,
$j=1,...,n-1$, the operator:
\[
\left(\partial_{x_{j}}T\right)\left(\lambda\right)\left(u\right)\left(x'\right)
=\int_{\mathbb{R}^{n-1}}e^{ix'\xi'}\int_{\mathbb{R}_{+}}\partial_{x_{j}}\tilde{t}\left(x',x_{n},\xi',\lambda\right)\left(\mathcal{F}_{x'\to\xi'}u\right)\left(\xi',x_{n}\right)\,dx_{n}d\xi'.
\]
Now, let us first observe that, for $j=1,...,n-1$, 
\begin{eqnarray}
\lefteqn{\partial_{x_{j}}R_{s}^{-1}T\left(s\lambda\right)\left(R_{s}v\otimes S_{s}w\right)}\nonumber\\
&=&R_{s}^{-1}T\left(s\lambda\right)\left(R_{s}\left(\partial_{x_{j}}v\right)\otimes S_{s}w\right)+s^{-\tau}R_{s}^{-1}\left(\partial_{x_{j}}T\right)\left(s\lambda\right)\left(R_{s}v\otimes S_{s}w\right).\label{eq:eqi}
\end{eqnarray}
Using Equation (\ref{eq:Lpconvergencetrace}) and the fact that $r<\tau$,
we conclude that 
\begin{eqnarray}
\lefteqn{s^{r}\Big\Vert R_{s}^{-1}T\left(s\lambda\right)\left(R_{s}\left(\partial_{x_{j}}v\right)\otimes S_{s}w\right)}\nonumber\\
&&-\partial_{x_{j}}v\left(x'\right)\int_{\mathbb{R}_{+}}\tilde{t}_{\left(-\frac{1}{p}\right)}\left(y,x_{n},\eta,\lambda\right)w\left(x_{n}\right)dx_{n}\Big\Vert _{L_{p}\left(\mathbb{R}^{n-1}\right)}
\le C\left(s\right)\left\Vert w\right\Vert _{L_{p}\left(\mathbb{R}_{+}\right)}\label{eq:eqii}
\end{eqnarray}
and
\begin{eqnarray}
\lefteqn{s^{r}\left\Vert s^{-\tau}R_{s}^{-1}\left(\partial_{x_{j}}T\right)\left(s\lambda\right)\left(R_{s}v\otimes S_{s}w\right)\right\Vert _{L_{p}\left(\mathbb{R}^{n-1}\right)}}
\nonumber\\
&\le&
s^{r-\tau}\Big\Vert R_{s}^{-1}\left(\partial_{x_{j}}T\right)\left(s\lambda\right)\left(R_{s}v\otimes S_{s}w\right)\nonumber \\
&&-v\left(x'\right)\int_{\mathbb{R}_{+}}\left(\partial_{x_{j}}\tilde{t}\right){}_{\left(-\frac{1}{p}\right)}\left(y,x_{n},\eta,\lambda\right)w\left(x_{n}\right)dx_{n}\Big\Vert _{L_{p}\left(\mathbb{R}^{n-1}\right)}
\nonumber\\
&&+
s^{r-\tau}\left\Vert v\right\Vert _{L_{p}\left(\mathbb{R}^{n-1}\right)}\left\Vert x_{n}\mapsto\left(\partial_{x_{j}}\tilde{t}\right){}_{\left(-\frac{1}{p}\right)}\left(y,x_{n},\eta,\lambda\right)\right\Vert _{L_{q}\left(\mathbb{R}_{+}\right)}\left\Vert w\right\Vert _{L_{p}\left(\mathbb{R}_{+}\right)}\nonumber\\
&\le& C\left(s\right)\left\Vert w\right\Vert _{L_{p}\left(\mathbb{R}_{+}\right)}.\label{eq:eqiii}
\end{eqnarray}
The expressions (\ref{eq:eqi}), (\ref{eq:eqii}) and (\ref{eq:eqiii})
imply that
\begin{eqnarray}
\lefteqn{
s^{r}\Big\Vert \partial_{x_{j}}\Big(R_{s}^{-1}T\left(s\lambda\right)\left(R_{s}v\otimes S_{s}w\right)}\nonumber\\
&&-v\left(x'\right)\int_{\mathbb{R}_{+}}\tilde{t}_{\left(-\frac{1}{p}\right)}\left(y,x_{n},\eta,\lambda\right)w\left(x_{n}\right)dx_{n}\Big)\Big\Vert _{L_{p}\left(\mathbb{R}^{n-1}\right)}
\le C\left(s\right)\left\Vert w\right\Vert _{L_{p}\left(\mathbb{R}_{+}\right)}.\label{eq:eq3}
\end{eqnarray}
Finally,  \eqref{eq:convH1BdM} is a consequence of
Equations \eqref{eq:eq3} and \eqref{eq:Lpconvergencetrace}.

We are now in the position to prove item 3. Choose $0<\theta<\theta+r<\tau$. Then
\begin{eqnarray}
\lefteqn{s^{r}\Big\Vert T\left(s\lambda\right)\left(R_{s}v\otimes S_{s}w\right)-\left(R_{s}v\right)\left(x'\right)\int_{\mathbb{R}_{+}}\tilde{t}_{\left(-\frac{1}{p}\right)}\left(y,x_{n},\eta,\lambda\right)w\left(x_{n}\right)dx_{n}\Big\Vert _{B_{p}^{0}\left(\mathbb{R}^{n-1}\right)}}
\nonumber\\
&\le&
s^{r}\Big\Vert R_{s}\Big(R_{s}^{-1}T(s\lambda)\left(R_{s}v\otimes S_{s}w\right)-v\left(x'\right)\int_{\mathbb{R}_{+}}\tilde{t}_{\left(-\frac{1}{p}\right)}\left(y,x_{n},\eta,\lambda\right)w\left(x_{n}\right)dx_{n}\Big)\Big\Vert _{B_{p}^{\theta}\left(\mathbb{R}^{n-1}\right)}
\nonumber\\
&\le&
C_{\theta}\left(1+s\left\langle \eta\right\rangle \right)^{\theta}s^{r}\Big\Vert R_{s}^{-1}T\left(s\lambda\right)\left(R_{s}v\otimes S_{s}w\right)\nonumber\\
&&-v\left(x'\right)\int_{\mathbb{R}_{+}}\tilde{t}_{\left(-\frac{1}{p}\right)}\left(y,x_{n},\eta,\lambda\right)w\left(x_{n}\right)dx_{n}\Big\Vert _{H_{p}^{1}\left(\mathbb{R}^{n}\right)}
\le
C\left(s\right)\left\Vert w\right\Vert _{L_{p}\left(\mathbb{R}_{+}\right)}.
\nonumber
\end{eqnarray}
\end{proof}
We also need to understand the action of the singular Green and trace operators on the operators $R_{s}=R_{s}\left(y,\eta\right)$ for
$\left(y,\eta\right)\in\overline{\mathbb{R}_{+}^{n}}\times\mathbb{R}^{n}$.
Notice that $\left(y,\eta\right)\in\overline{\mathbb{R}_{+}^{n}}\times\mathbb{R}^{n}$
instead of $\mathbb{R}^{n-1}\times\mathbb{R}^{n-1}$ as in the previous
proposition. 
\begin{prop}
\label{lem:GreenTrace}Let $R_{s}=R_{s}\left(y,\eta\right)$, where
$\eta=\left(\eta',\eta_{n}\right)\in\mathbb{R}^{n-1}\times\mathbb{R}$
and $y=\left(y',0\right)\in\mathbb{R}^{n-1}\times\mathbb{R}$. For
$u\in C_{c}^{\infty}\left(\mathbb{R}_{+}^{n}\right)$ the following
properties hold:

{\rm1) (}Green{\rm)} For $\tilde{g}\in S_{cl}^{-1}\left(\mathbb{R}^{n-1},\mathcal{S}_{++},\Lambda\right)$ define
$G\left(\lambda\right):\mathcal{S}\left(\mathbb{R}_{+}^{n}\right)\to\mathcal{S}\left(\mathbb{R}_{+}^{n}\right)$
by Equation \eqref{eq:DefGreen}. Then $\lim_{s\to\infty}s^{r}\left\Vert G\left(s\lambda\right)R_{s}\left(e^{+}u\right)\right\Vert _{L_{p}\left(\mathbb{R}_{+}^{n}\right)}=0$ for all $r>0$.

{\rm2) (}Trace{\rm)} For $\tilde{t}\in S_{cl}^{-\frac{1}{p}}\left(\mathbb{R}^{n-1},\mathcal{S}_{+},\Lambda\right)$ define
$T\left(\lambda\right):\mathcal{S}\left(\mathbb{R}_{+}^{n}\right)\to\mathcal{S}\left(\mathbb{R}^{n-1}\right)$
by Equation \eqref{eq:DefTrace}. Then $\lim_{s\to\infty}s^{r}\left\Vert T\left(s\lambda\right)R_{s}\left(e^{+}u\right)\right\Vert _{B_{p}^{0}\left(\mathbb{R}^{n-1}\right)}=0$ for all $r>0$.
\end{prop}
\begin{proof}
The proof is analogous to that of Proposition
\ref{prop:PseudoGreenPoissonTrace}. Let us sketch the proof of $2)$
as $1)$ is similar. 

Let $\frac{1}{p}+\frac{1}{q}=1$, $R_{s}^{-1}:=R_{s}^{-1}\left(y',\eta'\right):\mathcal{S}\left(\mathbb{R}^{n-1}\right)\to\mathcal{S}\left(\mathbb{R}^{n-1}\right)$
and $R_{s}:=R_{s}\left(y,\eta\right):\mathcal{S}\left(\mathbb{R}^{n}\right)\to\mathcal{S}\left(\mathbb{R}^{n}\right)$.
Using item 6 of Lemma \ref{lem:Propriedades R_s}
in $\left(x',\xi'\right)$ and the definition of $R_{s}$,
we obtain that
\begin{eqnarray*}
R_{s}^{-1}T\left(s\lambda\right)\left(R_{s}u\right)\left(x'\right)
&=&\int_{\mathbb{R}^{n-1}}e^{ix'\xi'}
\Big(\int_{\mathbb{R}_{+}}e^{i\tau s^{1-\tau}x_{n}\eta_{n}}s^{-\frac{\tau}{q}}\\
&\times& \tilde{t}\left(y+s^{-\tau}x',\frac{x_{n}}{s^{\tau}},s\eta+s^{\tau}\xi',s\lambda\right)\mathcal{F}_{x'\to\xi'}u\left(\xi',x_{n}\right)dx_{n}\Big)d\xi'.
\end{eqnarray*}
Now, we note that
\[
\left(\frac{x_{n}}{s^{\tau}}\right)^{N}\left|\tilde{t}\left(y+s^{-\tau}x',\frac{x_{n}}{s^{\tau}},s\eta+s^{\tau}\xi',s\lambda\right)\right|\le C_{N}\left\langle s\eta+s^{\tau}\xi',s\lambda\right\rangle ^{-\frac{1}{p}+1-N},
\]
On the support of $u$, we have $x_{n}\ge R>0$ for a certain
constant $R>0$. 
Hence
\[
\left|\tilde{t}\left(y+s^{-\tau}x',\frac{x_{n}}{s^{\tau}},s\eta+s^{\tau}\xi',s\lambda\right)\right|\le C_{N}\left\langle s\eta+s^{\tau}\xi',s\lambda\right\rangle ^{-\frac{1}{p}+1-N}s^{\tau N}R^{-N}
\]
\[
\le C_{N}s^{\left(\frac{1}{p}-1\right)\left(\tau-1\right)+\left(2\tau-1\right)N}\left\langle \eta,\lambda\right\rangle ^{-\frac{1}{p}+1-N}\left\langle \xi'\right\rangle ^{N+\frac{1}{p}-1}R^{-N}.
\]

As $2\tau-1<0$, we can always choose $N\in\mathbb{N}_{0}$
so large that, for all $\left(x',x_{n},\xi',\lambda\right)\in\mathbb{R}^{n-1}\times\mathbb{R}_{+}\times\mathbb{R}^{n-1}\times\Lambda$
such that $x_{n}\ge R$ and for all $r>0$, we have
\[
\lim_{s\to\infty}s^{r}\left(s^{-\frac{\tau}{q}}\tilde{t}\left(y+s^{-\tau}x',\frac{x_{n}}{s^\tau},s\eta+s^{\tau}\xi',s\lambda\right)\right)=0.
\]
For large $N\in\mathbb{N}_{0}$, the dominated convergence theorem implies that
\[
\lim_{s\to\infty}s^{r}\left(R_{s}^{-1}T\left(s\lambda\right)\left(R_{s}u\right)\left(x'\right)\right)=0,\,\,\, r>0.
\]

Now, to finish the proof, we just study $L_{p}$ and $H_{p}^{1}$
convergence. Using integration by parts in the expression $x'^{\gamma}R_{s}^{-1}T\left(s\lambda\right)\left(R_{s}u\right)\left(x'\right)$,
we see that we can dominate $R_{s}^{-1}T\left(s\lambda\right)\left(R_{s}u\right)\left(x'\right)$
by $\left\langle x'\right\rangle ^{-N}$ for every $N$. Hence 
\[
\lim_{s\to\infty}s^{r}\left\Vert R_{s}^{-1}T\left(s\lambda\right)\left(R_{s}u\right)\right\Vert _{L_{p}\left(\mathbb{R}^{n-1}\right)}=0.
\]
If we take derivatives of first order in $x'$, we find that 
\[
\lim_{s\to\infty}s^{r}\left\Vert R_{s}^{-1}T\left(s\lambda\right)\left(R_{s}u\right)\right\Vert _{H_{p}^{1}\left(\mathbb{R}^{n-1}\right)}=0.
\]
The estimate of the norm of $R_{s}$ on Besov space and the same argument
with interpolation of Proposition \ref{prop:PseudoGreenPoissonTrace}
lead us to the conclusion that 
\[
\lim_{s\to\infty}s^{r}\left\Vert T\left(s\lambda\right)\left(R_{s}u\right)\right\Vert _{B_{p}^{0}\left(\mathbb{R}^{n-1}\right)}=0.
\]
\end{proof}
Finally, we prove the main Theorem of this sub-section.

\begin{proof}
(of Theorem \ref{thm:equivalenciaFredholmelipticBdM})
Let $A=\left(\begin{array}{cc}
P_{+}+G & K\\
T & S
\end{array}\right)\in\tilde{\mathcal{B}}^{p}_{E_{0},F_{0},E_{1},F_{1}}\left(M,\Lambda\right)$
and $B_{1}$, $B_{2}$,
$K_{1}$ and $K_{2}$ be as in Theorem \ref{thm:equivalenciaFredholmelipticBdM}.iii). 
Write
$B_{1}=\left(\begin{array}{cc}
B_{11} & B_{12}\\
B_{21} & B_{22}
\end{array}\right)$ and 
decompose similarly $K_1$.

Next we choose smooth functions $\Phi, \Psi$ and $\mathrm H$, supported in a trivializing neighborhood $U$ of $x=\pi(z)$, 
such that $\Phi$ equals $1$ near $x$ and $\Psi\Phi= \Phi$, $\mathrm H\Psi=\Psi$.  
We denote by 
$\tilde P_+(\lambda), \tilde G(\lambda) \in 
\mathcal B(L^p(\mathbb R^n_+)^{n_1}, L^p(\mathbb R^n_+)^{n_3})$,
$\tilde T(\lambda) \in \mathcal B(L^p(\mathbb R^n_+)^{n_1},B^0_p(\mathbb R^{n-1})^{n_4})$,
$\tilde B_{11}(\lambda)\in \mathcal B(L_p(\mathbb R^n_+)^{n_3}, L_p(\mathbb R^n)^{n_1})$ 
and $\tilde B_{12}(\lambda)\in \mathcal B(B^0_p(\mathbb R^{n-1})^{n_4}, L_p(\mathbb R^n_+)^{n_1})$ 
the operators $\mathrm H P_+(\lambda)\Psi$, $\mathrm H G(\lambda)\Psi$, $\Phi B_{11}(\lambda)\mathrm H$,  $\mathrm H T(\lambda) \Psi$, and $\Phi B_{12}(\lambda)\mathrm H$ 
in local coordinates.  

The identity $B_{1}A=I+K_{1}$ implies that 
\begin{eqnarray}\label{eq:Proof1}
\tilde B_{11}(\lambda) (\tilde P_+(\lambda) +\tilde G(\lambda))
+ \tilde B_{12} (\lambda)\tilde T(\lambda) = \tilde \Phi + \tilde K(\lambda),
\end{eqnarray}
where $\tilde \Phi$ is the function $\Phi$ in local coordinates and $\tilde K(\lambda)$ is 
the operator which collects the terms arising from the localizations of 
$\Phi K_{11}(\lambda)\mathrm H$, 
$\Phi B_{11}(\lambda)(1-\mathrm H^2)(P_+(\lambda)+ G(\lambda))\Psi$ 
and $\Phi B_{12}(\lambda) (1-\mathrm H^2) T(\lambda)\Psi$.  
As the latter two operators have smooth integral kernels, with seminorms rapidly 
decreasing with respect to $\lambda$, $\tilde K(\lambda)$ is compact and its norm  tends to zero as $|\lambda|\to\infty$. 

\subsubsection*{The interior principal symbol}
In order to prove the invertibility of the interior principal symbol $p_{\left(0\right)}\left(z,\lambda\right):\pi_{T^{*}M\times\Lambda}^{*}\left(E_{0}\right)\to\pi_{T^{*}M\times\Lambda}^{*}\left(E_{1}\right)$
for $(z,\lambda)\in\left(T^{*}M\times\Lambda\right)\backslash\left\{ 0\right\} $, 
fix $u=cv\in C_{c}^{\infty}(\mathbb{R}_{+}^{n})^{n_{1}}$,
where $c\in\mathbb{C}^{n_{1}}$
and $0\neq v\in C_{c}^{\infty}(\mathbb{R}_{+}^{n})$.
Denote by $(y,\eta)\in \overline{\mathbb R}^n_+\times \mathbb R^n$ the point corresponding 
to $z$ in local coordinates. 
For $R_{s}=R_{s}\left(y,\eta\right)$ we note
that $R_s(e^+u)\in C^\infty_c(\mathbb R^n_+)$, since 
$\text{supp} \,R_s (e^+u)  \subset \mathbb R^n_+$.  
In particular $\Vert u\Vert _{L_{p}(\mathbb{R}_{+}^{n})^{n_{1}}}=\Vert r^{+}R_{s}(e^{+}u)\Vert _{L_{p}(\mathbb{R}_{+}^{n})^{n_{1}}}$.
Hence we obtain from \eqref{eq:Proof1}
\begin{eqnarray}
\lefteqn{\Vert u\Vert _{L_{p}(\mathbb{R}_{+}^{n})^{n_{1}}}
\le \Vert {B}_{11}(s\lambda)\Vert _{\mathcal{B}(L_{p}(\mathbb{R}_{+}^{n})^{ n_{3}},L_{p}(\mathbb{R}_{+}^{n})^{ n_{1}})}
\Vert \tilde{P}(s\lambda)R_{s}(e^{+}u)\Vert _{L_{p}(\mathbb{R}_{+}^{n})^{ n_{3}}}
}
\nonumber\\
&&
+\Vert (\tilde B_{11}\tilde G+ \tilde B_{12}\tilde T)(s\lambda)R_{s}(e^{+}u)\Vert _{L_{p}(\mathbb{R}_{+}^{n})^{ n_{1}}}
+\Vert \tilde K (s\lambda)R_{s}(e^{+}u)\Vert _{L_{p}(\mathbb{R}_{+}^{n})^{ n_{1}}}
\nonumber\\
&&+\Vert (1-\tilde \Phi)R_{s}(e^{+}u)\Vert _{L_{p}(\mathbb{R}_{+}^{n})^{ n_{1}}}.
\label{eq:speccomp-1-1}
\end{eqnarray}
On the right hand side of Equation \eqref{eq:speccomp-1-1},
we estimate
\begin{eqnarray}
\Vert \tilde{P}(s\lambda)R_{s}(e^{+}u)\Vert _{L_{p}(\mathbb{R}_{+}^{n})^{ n_{3}}}&\le&\Vert \tilde{P}(s\lambda)R_{s}(e^{+}u)-p_{(0)}(y,\eta,\lambda)R_{s}(e^{+}u)\Vert _{L_{p}(\mathbb{R}_{+}^{n})^{ n_{3}}}\nonumber\\
&&+C\Vert p_{(0)}(y,\eta,\lambda)c\Vert _{\mathcal{B}(\mathbb{C}^{n_{1}},\mathbb{C}^{n_{3}})}\Vert v\Vert _{L_{p}(\mathbb{R}_{+}^{n})}\label{eq:PRs}
\end{eqnarray}
and note that Corollary \ref{cor:simbolo rmais} implies that 
\[
\lim_{s\to\infty}s^{r}\Vert \tilde{P}(s\lambda)R_{s}(e^{+}u)-p_{(0)}(y,\eta,\lambda)R_{s}(e^{+}u)\Vert _{L_{p}(\mathbb{R}_{+}^{n})^{n_{3}}}=0.
\]
We claim that also $\tilde K(s\lambda)R_s(e^+u)$ tends to zero: For $\lambda=0$ we infer this from  the fact that $\tilde K(0)$ is compact, while   $R_s(e^+u)$ weakly tends to zero. 
For $\lambda \not=0$ the norm of $\tilde K(s\lambda)$ tends to zero as $s\to\infty$, whereas $R_s(e^+u)$ is bounded. 
Finally, it is easy to check that $\lim_{s\to\infty}(1-\tilde \Phi)R_{s}(e^{+}u)=0$
in $\mathcal{S}\left(\mathbb{R}^{n}\right)^{ n_{1}}$ and therefore
also in $L_{p}\left(\mathbb{R}_{+}^{n}\right)^{ n_{1}}$. 

If we assume, for an instant, that also the second summand on the right hand side of 
\eqref{eq:speccomp-1-1} tends to zero as $s\to \infty$, 
then, taking $s$ sufficiently large, the boundedness of $R_{s}$, Inequality \eqref{eq:PRs} and  Equation (\ref{eq:speccomp-1-1})
imply together with the assumption that  $\Vert \tilde{B}_{11}(s\lambda)\Vert _{\mathcal{B}(L_{p}(\mathbb{R}_{+}^{n})^{n_{3}},L_{p}(\mathbb{R}_{+}^{n})^{ n_{1}})}\le C\langle \text{ln}(s\lambda)\rangle ^{M}$,
that
\[
\left\Vert c\right\Vert _{\mathbb{C}^{n_{1}}}\left\Vert v\right\Vert _{L_{p}(\mathbb{R}_{+}^{n})}=\left\Vert u\right\Vert _{L_{p}\left(\mathbb{R}_{+}^{n}\right)^{n_{1}}}\le\tilde{C}\left\Vert p_{\left(0\right)}\left(y,\eta,\lambda\right)c\right\Vert _{\mathcal{B}\left(\mathbb{C}^{n_{1}},\mathbb{C}^{n_{3}}\right)}\left\Vert v\right\Vert _{L_{p}\left(\mathbb{R}_{+}^{n}\right)}.
\]
Hence $p_{\left(0\right)}\left(y,\eta,\lambda\right)$ is injective. 
The same argument,  applied to the adjoint operator, shows the injectivity of
$p_{\left(0\right)}\left(y,\eta,\lambda\right)^{*}$ and thus the invertibility of 
$p_{\left(0\right)}\left(y,\eta,\lambda\right)$.
In particular,  $n_{1}=n_{3}$.
In order to establish the convergence to zero of the second summand in \eqref{eq:speccomp-1-1},
we distinguish two cases. 

\subsubsection*{Case 1: $x\notin\partial M$.} 
Then $U$ can be taken as a subset of the interior of $M$. 
According to the rules of the calculus,  $\tilde T(s\lambda)$ and $\tilde G(s\lambda)$ are regularizing elements in their
respective classes; in particular, they are compact. For $\lambda\neq 0$, their operator norms  are rapidly decreasing as $s\to\infty$. Arguing as for $\tilde K$ above, we obtain the assertion from the assumptions on $B$. 
  
\subsubsection*{Case 2: $x\in \partial M$}
Here, statements 1) and 2) of Proposition \ref{lem:GreenTrace} assert that, 
for every $r>0$, the norms of
$s^r\tilde G(s\lambda)R_s(e^+u) $ and $s^r\tilde  T(s\lambda) R_s(e^+u) $  go to zero in the corresponding spaces as $s\to\infty$. The assertion then follows from the fact that the norm of $B(s\lambda)$ grows at most logarithmically in $s$ by assumption.  

\subsubsection*{The boundary principal symbol}
We have to show that, for any given  $\left(z,\lambda\right)\in\left(T^{*}\partial M\times\Lambda\right)\backslash\{0\}$, $\sigma_\partial(A)(z,\lambda)$  is invertible in  
$$Hom(\pi_{\partial M}^{*}\left(\left(\left.E_{0}\right|_{\partial M}\otimes\mathcal{S}\left(\mathbb{R}_{+}\right)\right)\oplus F_{0}\right),
\pi_{\partial M}^{*}\left(\left(\left.E_{1}\right|_{\partial M}\otimes\mathcal{S}\left(\mathbb{R}_{+}\right)\right)\oplus F_{1}\right))).$$
Let $\tilde{B}$ and $\tilde{A}$ be the operators $\mathrm{H}A\Psi$ and $\Phi B \mathrm H$ in local coordinates, respectively. Write the principal boundary symbol of $\tilde A$ in the form 
\begin{equation}
\left(\begin{array}{cc}
p_{\left(0\right)}{}_{+}(x',0,\xi',D_{n},\lambda)+g_{\left(-1\right)}(x',\xi',D_{n},\lambda) & k_{\left(\frac{1}{p}-1\right)}(x',\xi',D_{n},\lambda)\\
t_{\left(-\frac{1}{p}\right)}(x',\xi',D_{n},\lambda) & s_{\left(0\right)}(x',\xi',\lambda)
\end{array}\right)
\label{eq:bps}
\end{equation}
and let  $\left(y,\eta\right)\in\mathbb{R}^{n-1}\times\mathbb{R}^{n-1}$
be the point that corresponds to $z$ in local coordinates.

Fix a function $0\neq u'\in\mathcal{S}\left(\mathbb{R}^{n-1}\right)$
with $\mbox{supp}\left(\mathcal{F}u'\right)\subset\left\{ \xi;\,\frac{1}{2}<|{\xi}|<1\right\} $.
For $u=\left(u_{1},...,u_{n_{1}}\right)\in\mathcal{S}\left(\mathbb{R}_{+}\right)^{n_{1}}$
and $v=\left(v_{1},...,v_{n_{3}}\right)\in\mathbb{C}^{n_{2}}$, not both zero,  
denote by $u'\otimes u$ and $u'\otimes v$ the functions $\mathbb{R}_{+}^{n}\ni \left(x',x_{n}\right)\mapsto\left(u'\left(x'\right)u_{1}\left(x_{n}\right),...,u'\left(x'\right)u_{n_{1}}\left(x_{n}\right)\right)$
and  $\mathbb{R}^{n-1} \ni x' \mapsto \left(u'\left(x'\right)v_{1},...,u'\left(x'\right)v_{n_{2}}\right)$, respectively.
According to Lemmas \ref{lem:Propriedades R_s}, \ref{lem:LpBp} and \ref{lem:Rs em supp de F na bola} there are constants 
such that
\begin{eqnarray*}
\lefteqn{\left\Vert u'\right\Vert _{L_{p}\left(\mathbb{R}^{n-1}\right)}
=\left\Vert R_{s}u'\right\Vert _{L_{p}\left(\mathbb{R}^{n-1}\right)}}\\
&\le& C_{1}\left\Vert u'\right\Vert _{B_{p}^{0}\left(\mathbb{R}^{n-1}\right)}\le C_{2}\left\Vert R_{s}u'\right\Vert _{B_{p}^{0}\left(\mathbb{R}^{n-1}\right)}\le C_{3}\left\Vert u'\right\Vert _{L_{p}\left(\mathbb{R}^{n-1}\right)}, \quad s\ge 1.
\end{eqnarray*}

Writing $\left\Vert .\right\Vert _{L_{p}B_{p}^{0}}$ for the
norm in $L_{p}(\mathbb{R}_{+}^{n})^{ n_{1}}\oplus B_{p}^{0}(\mathbb{R}^{n-1})^{ n_{2}}$, in analogy with Equation \ref{eq:Proof1} conclude from the identity $B_1A=I+K_1$ that  
\begin{eqnarray}
\lefteqn{\left\Vert u'\right\Vert _{L^{p}(\mathbb R^{n-1})}\left\Vert \binom{u}{v}\right\Vert _{L^{p}\left(\mathbb{R}_{+}\right)^{ n_{1}}\oplus \mathbb{C}^{n_{2}}}}
\nonumber\\
&\le&
C\left\Vert \tilde \Phi\binom{R_{s}u'\otimes S_{s}u}{R_{s}u'\otimes v}\right\Vert _{L_{p}B_{p}^{0}} +
C\left\Vert \left(1-\tilde \Phi\right)
\binom{R_{s}u'\otimes S_{s}u}{R_{s}u'\otimes v}\right\Vert _{L_{p}B_{p}^{0}}
\nonumber\\
&\le&
C\left(\left\Vert \tilde{B}\left(s\lambda\right)\tilde{A}\left(s\lambda\right)
\begin{pmatrix}
R_{s}\otimes S_{s} & 0\\
0 & R_{s}
\end{pmatrix} 
\binom{u'\otimes u}{u'\otimes v}
-\tilde{B}\left(s\lambda\right)
\begin{pmatrix}
R_{s}\otimes S_{s} & 0\\
0 & R_{s}
\end{pmatrix}
\right.\right.
\nonumber\\
&&\times
\left.
\begin{pmatrix}
p_{\left(0\right)}{}_{+}(x',0,\xi',D_{n},\lambda)+g_{\left(-1\right)}(x',\xi',D_{n},\lambda) & k_{\left(\frac{1}{p}-1\right)}(x',\xi',D_{n},\lambda)\\
t_{\left(-\frac{1}{p}\right)}(x',\xi',D_{n},\lambda) & s_{\left(0\right)}(x',\xi',\lambda)
\end{pmatrix}
\binom{u'\otimes u}{u'\otimes v}
\right\Vert _{L_{p}B_{p}^{0}}
\nonumber\\
&&+
\left\Vert \tilde{B}\left(s\lambda\right)
\begin{pmatrix} 
R_{s}\otimes S_{s} & 0\\
0 & R_{s}
\end{pmatrix}\right. \nonumber\\
&&\times
\left.\begin{pmatrix}
p_{\left(0\right)}{}_{+}(x',0,\xi',D_{n},\lambda)+g_{\left(-1\right)}(x',\xi',D_{n},\lambda) & k_{\left(\frac{1}{p}-1\right)}(x',\xi',D_{n},\lambda)\\
t_{\left(-\frac{1}{p}\right)}(x',\xi',D_{n},\lambda) & s_{\left(0\right)}(x',\xi',\lambda)
\end{pmatrix}\binom{u'\otimes u}{u'\otimes v}
\right\Vert _{L_{p}B_{p}^{0}}
\nonumber\\
&&+
\left.\left\Vert \tilde{K}\left(s\lambda\right)
\binom{R_{s}u'\otimes S_{s}u}{R_{s}u'\otimes v}
\right\Vert _{L_{p}B_{p}^{0}}
+\left\Vert \left(1-\tilde \Phi\right)
\binom{R_{s}u'\otimes S_{s}u}{R_{s}u'\otimes v}\right\Vert _{L_{p}B_{p}^{0}}\right).
\nonumber
\end{eqnarray}
Let us first consider the case where $\lambda\ne0$. 
We infer from Proposition  \ref{prop:PseudoGreenPoissonTrace} and the fact that 
the norm of $\tilde B(s\lambda)$ is $O(\langle\ln(s\lambda)\rangle^M)$ that the first summand on the right hand side 
is $o(\|(u'\otimes u )\oplus (u'\otimes v)\|)$. 
The same is true for the third summand, since the norm of $\tilde K(s\lambda)$ 
tends to zero as $s\to \infty$. The fourth summand tends to zero in 
$\mathcal  S(\mathbb R^n_+)^{n_1} \oplus S(\mathbb R^{n-1}_+)^{n_2}$, 
a fortiori in the $L_pB^0_p$-norm. Taking $s$ sufficiently large, we may achieve that 
the sum of the first, the third and the fourth summand is $\le  \frac12 (\|(u'\otimes u )\oplus (u'\otimes v)\|)$.
From the boundedness of $\tilde B(s\lambda)$, $R_s$ and $S_s$ for this fixed value of $s$, 
we conclude that, with norms taken in  $L^{p}\left(\mathbb{R}_{+}\right)^{n_{1}}\oplus\mathbb{C}^{n_{2}}$ and $L^{p}\left(\mathbb{R}_{+}\right)^{n_{3}}\oplus\mathbb{C}^{n_{4}}$,
\begin{eqnarray*}\lefteqn{\left\Vert \binom {u}{v} \right\Vert 
}\\
&\le&\!\!\!  C\left\Vert \begin{pmatrix}
p_{\left(0\right)+}(x',0,\xi',D_{n},\lambda)+g_{\left(-1\right)}(x',\xi',D_{n},\lambda) & k_{\left(\frac{1}{p}-1\right)}(x',\xi',D_{n},\lambda)\\
t_{\left(-\frac{1}{p}\right)}(x',\xi',D_{n},\lambda) & s_{\left(0\right)}(x',\xi',\lambda)
\end{pmatrix}\binom{u}{v}\right\Vert 
\end{eqnarray*}
In case  $\lambda=0$, we obtain the same conclusion using the compactness of $\tilde{K}\left(0\right)$. 

Hence the operator from Equation (\ref{eq:bps}) is injective and
has closed range. As the same can be said of the adjoint, we conclude
from Lemma \ref{lem:injetora-adj-implica-iso} that the principal
boundary symbol is an isomorphism.
\end{proof}

\subsection{The spectral invariance of the parameter-dependent Boutet de Monvel
algebra\label{subsec:SpectralInvarianceBMparameters}}

\begin{thm}
\label{thm:Teoremabmcomparametros} Let $A\in\tilde{\mathcal{B}}_{E_{0},F_{0},E_{1},F_{1}}^{p}\left(M,\Lambda\right)$
be a parameter-dependent operator. Suppose that, for each $\lambda\in\Lambda$,
the operator
\[
A\left(\lambda\right):L_{p}\left(M,E_{0}\right)\oplus B_{p}^{0}\left(\partial M,F_{0}\right)\to L_{p}\left(M,E_{1}\right)\oplus B_{p}^{0}\left(\partial M,F_{1}\right)
\]
is invertible. If there are constants $C>0$ and $M\in\mathbb{N}_{0}$
such that 
\[
\left\Vert A\left(\lambda\right)^{-1}\right\Vert _{\mathcal{B}\left(L_{p}\left(M,E_{1}\right)\oplus B_{p}^{0}\left(\partial M,F_{1}\right),L_{p}\left(M,E_{0}\right)\oplus B_{p}^{0}\left(\partial M,F_{0}\right)\right)}\le C\left\langle \ln\left(\lambda\right)\right\rangle ^{M}, \lambda\in\Lambda,
\]
then $A\left(\lambda\right)^{-1}\in\tilde{\mathcal{B}}_{E_{1},F_{1},E_{0},F_{0}}^{p}\left(M,\Lambda\right)$.
\end{thm}

\begin{proof}
By Theorem \ref{thm:equivalenciaFredholmelipticBdM} $A$ is parameter-elliptic.
Hence we find a parametrix $B\in\tilde{\mathcal{B}}_{E_{1},F_{1},E_{0},F_{0}}^{p}\left(M,\Lambda\right)$
and  $K_{1}\in\mathcal{B}_{E_{1},F_{1},E_{1},F_{1}}^{-\infty,0}\left(M,\Lambda\right)$
and $K_{2}\in\mathcal{B}_{E_{0},F_{0},E_{0},F_{0}}^{-\infty,0}\left(M,\Lambda\right)$
such that $AB=I+K_{1}$
and $BA=I+K_{2}$.
We conclude that
\[
A\left(\lambda\right)^{-1}=B\left(\lambda\right)-K_{2}\left(\lambda\right)A\left(\lambda\right)^{-1}=B\left(\lambda\right)-K_{2}\left(\lambda\right)\left(B\left(\lambda\right)-A\left(\lambda\right)^{-1}K_{1}\left(\lambda\right)\right).
\]
As $K_{2}B\in\mathcal{B}_{E_{1},F_{1},E_{0},F_{0}}^{-\infty,0}\left(M,\Lambda\right)$, $A\left(\lambda\right)^{-1}$ grows at most as $\left\langle \ln\left(\lambda\right)\right\rangle ^{M}$
in $\lambda$ and $K_{j}\left(\lambda\right)$, $j=1,2$,  are integral operators with smooth kernels whose derivatives decay rapidly with respect to $\lambda$, we see that $K_{2}A^{-1}K_{1}\in\mathcal{B}_{E_{1},F_{1},E_{0},F_{0}}^{-\infty,0}\left(M,\Lambda\right)$
and $A^{-1}\in\tilde{\mathcal{B}}_{E_{1},F_{1},E_{0},F_{0}}^{p}\left(M,\Lambda\right)$.
\end{proof}

\section{Boundary value problems on manifolds with conical singularities}

In this section, we provide the  definitions and results concerning manifolds
with boundary and conical singularities that we shall need. Details
can be found in
\cite{SchroheSchulzeConicalBoundaryI,SchroheSchulzeConicalBoundaryII}.
\begin{defn}
A compact manifold with boundary and conical singularities of dimension
$n$ is a triple $\left(D,\Sigma,\mathcal{F}\right)$ formed by:

1) A compact Hausdorff topological space $D$.

2) A finite subset $\Sigma\subset D$, which we call conical points,
such that $D\backslash\Sigma$ is an $n$-dimensional smooth manifold
with boundary.

3) A set of functions $\mathcal{F}_{\sigma}=\left\{ \varphi:U_{\sigma}\to X_{\sigma}\times\left[0,1\right[/X_{\sigma}\times\left\{ 0\right\} ,\,\sigma\in\Sigma\right\} $
such that:

i) The sets $U_{\sigma}\subset D$ are open and disjoint sets. Moreover,
each $U_{\sigma}$ is a neighborhood of $\sigma\in\Sigma$.

ii) $X_{\sigma}$ is a compact smooth manifold with boundary for each $\sigma\in\Sigma$.

iii) The function $\varphi_{\sigma}:U_{\sigma}\to X_{\sigma}\times\left[0,1\right[/X_{\sigma}\times\left\{ 0\right\} $
is a homeomorphism, $\varphi_{\sigma}(\sigma)=X_{\sigma}\times\left\{ 0\right\} /X_{\sigma}\times\left\{ 0\right\} $
and $\varphi_{\sigma}:U_{\sigma}\backslash\left\{ \sigma\right\} \to X_{\sigma}\times\left]0,1\right[$
is a diffeomorphism.
\end{defn}
\begin{rem}
For each $\sigma\in\Sigma$, we could use a different function $\tilde{\varphi}_{\sigma}:U_{\sigma}\to X_{\sigma}\times\left[0,1\right[/X_{\sigma}\times\left\{ 0\right\} $
with the same properties as in item iii), as long as, for each $\sigma$,
\[
\tilde{\varphi}_{\sigma}\circ\varphi_{\sigma}^{-1}:X_{\sigma}\times\left]0,1\right[\to X_{\sigma}\times\left]0,1\right[
\]
extends to a diffeomorphism $\tilde{\varphi}_{\sigma}\circ\varphi_{\sigma}^{-1}:X_{\sigma}\times\left]-1,1\right[\to X_{\sigma}\times\left]-1,1\right[$.
These are the changes of variables that we allow to do near the singularities.
\end{rem}
For the analysis of the typical (pseudo-) differential boundary value problems on these
manifolds, we introduce the Fuchs type boundary value problems on a manifold with corners $\mathbb{D}$. It is obtained by gluing the sets $X_{\sigma}\times\left[0,1\right[$
in place of $U_{\sigma}$, using the functions $\varphi_{\sigma}$. In this
way, the singularities are identified with the sets $X_{\sigma}\times\left\{ 0\right\} $.
The above remark ensures that the use of different functions $\tilde{\varphi}_{\sigma}$
instead of $\varphi_{\sigma}$ leads to diffeomorphic manifolds with corners.
In order to avoid unnecessary complications with the notation, we
shall consider manifolds with just one point singularity. A neighborhood of the conical point will always be identified with $X\times\left[0,1\right[/X\times\left\{0\right\}$ and a neighborhood of the corner will always be identified with $X\times\left[0,1\right[$, where $X$ is a compact manifold with boundary. For a finite
number of singularities the definitions and arguments are analogous.

We will denote by $\text{int}\left(\mathbb{D}\right)$ the manifold
with boundary $\mathbb{D}\backslash\left(X\times\left\{ 0\right\} \right)$.
By $\text{int}\left(\mathbb{B}\right)$, we denote the boundary of $\text{int}\left(\mathbb{D}\right)$. In a neighborhood of the singularity, it can be identified with $\partial X\times\left]0,1\right[$. Finally $\mathbb{B}$
is the manifold with boundary given by $\text{int}\left(\mathbb{B}\right)\cup\left(\partial X\times\left\{ 0\right\} \right)$.
In particular, in a neighborhood of the singularity, it can be identified
with $\partial X\times\left[0,1\right[$. We will also use $2\mathbb{D}$
to denote a manifold with boundary in which $\mathbb{D}$ is embedded.
The boundary of $2\mathbb{D}$ is $2\mathbb{B}$, a
manifold without boundary.

We divide our presentation into two parts. First we define the classes of functions
and distributions and then the operators.  The operators
acting on a neighborhood of the singularity will be defined as operators
on $X\times\left]0,1\right[$. We denote by $E_{0}$ and $E_{1}$ two vector
bundles over $\mathbb{D}$ and by $F_{0}$ and $F_{1}$ two vector
bundles over $\mathbb{B}$. Let $\pi_{X}:X\times\left[0,1\right[\to X$
be the projection operator, then there are vector bundles
$E_{0}'$ and $E_{1}'$ over $X$ such that $E_{0}$ and $E_{1}$
can be identified with $\pi_{X}^{*}\left(E_{0}'\right)$ and $\pi_{X}^{*}\left(E_{1}'\right)$,
respectively. Similarly, if $\pi_{\partial X}:\partial X\times\left[0,1\right[\to\partial X$
is the projection operator, then there are vector bundles
$F_{0}'$ and $F_{1}'$ over $\partial X$ such that $F_{0}$ and
$F_{1}$ can be identified with $\pi_{\partial X}^{*}\left(F_{0}'\right)$
and $\pi_{\partial X}^{*}\left(F_{1}'\right)$, respectively. $E_{0}$ will denote $E_{0}$ and $E_{0}'$ and the same will be done for $E_{1}$, $F_{0}$ and $F_{1}$. We also
denote by $2E_{0}$, $2F_{0}$, ... the vector bundles over $2\mathbb{D}$
and $2\mathbb{B}$, whose restriction to $\mathbb{D}$ and $\mathbb{B}$
are $E_{0}$ and $F_{0}$.

Finally, a cut-off function $\omega\in C_{c}^{\infty}\left(\overline{\mathbb{R}_{+}}\right)$
is a smooth nonnegative function that is equal to $1$ in
a neighborhood of $0$ and equal to $0$, outside $\left[0,1\right]$.

\subsection{Classes of functions and distributions}

In the following sections, $X$ is a manifold endowed with a Riemannian metric and with boundary $\partial X$. All vector bundles are assumed to be hermitian. We use the notation $X^{\wedge}:=\mathbb{R}_{+}\times X$
and $\partial X^{\wedge}:=\mathbb{R}_{+}\times\partial X$ and we will denote by $E$, $E_{0}$ and $E_{1}$ vector bundles
over $X$ or $\mathbb{D}$ and by $F$, $F_{0}$ and $F_{1}$ vector bundles over $\partial X$ or $\mathbb{B}$. The vector bundles
$E$, $E_{0}$, $E_{1}$, $F$, $F_{0}$ and $F_{1}$ will also refer
to the pullback bundles in $X\times\mathbb{R}$,
$X^{\wedge}$, $\partial X\times\mathbb{R}$ and $\partial X^{\wedge}$.
Finally we denote by $C^{\infty}(X,E,F)$ the set $C^{\infty}\left(X,E\right)\oplus C^{\infty}\left(\partial X,F\right)$.

\begin{defn}
Let $W$ be a Fréchet space and $\gamma\in\mathbb{R}$. We define
the Fréchet space $\mathcal{T}_{\gamma}\left(\mathbb{R}_{+},W\right)$
as the space of all functions $\varphi\in C^{\infty}\left(\mathbb{R}_{+},W\right)$
that satisfy
\[
\mbox{sup}\left\{ \left\langle \ln\left(t\right)\right\rangle ^{l}p\left(t^{\frac{1}{2}-\gamma}\left(t\partial_{t}\right)^{k}\varphi\left(t\right)\right),\,t\in\mathbb{R}_{+}\right\} <\infty,
\]
for all $k.\,l\in\mathbb{N}_{0}$ and for all continuous seminorms
$p$ of $W$. We write $\mathcal{T}_{\gamma}\left(\mathbb{R}_{+}\right)$
when $W=\mathbb{C}$.
\end{defn}
\begin{defn}
Let $\omega\in C^{\infty}\left(\left[0,1\right[\right)$ be a cut-off function. The space of functions $C_{\gamma}^{\infty}\left(\mathbb{D}\right)$, $\gamma \in \mathbb{R}$, consists of all functions $u\in C^{\infty}\left(\text{int}\left(\mathbb{D}\right)\right)$ such that $\omega u\in\mathcal{T}_{\gamma-\frac{n}{2}}\left(X^{\wedge}\right)$. Similarly, $C_{\gamma}^{\infty}\left(\mathbb{B}\right)$ are all the functions $u\in C^{\infty}\left(\text{int}\left(\mathbb{B}\right)\right)$ such that $\omega u\in\mathcal{T}_{\gamma-\frac{n-1}{2}}\left(\partial X^{\wedge}\right)$.
\end{defn}

\begin{defn}
Let $X=\cup_{j=1}^{M}U_{j}$  be a cover of $X$ consisting of trivializing sets
and $\varphi_{j}:U_{j}\subset X\to V_{j}\subset\overline{\mathbb{R}_{+}^{n}}$
be coordinate charts and $\left(\psi_{j}\right)_{j=1}^{M}$ be a partition
of unity subordinate to $U_{j}$, $j=1,...,M$. The space $H_{p}^{s}\left(X\times\mathbb{R},E\right)$
is defined as the set of distributions $\mathcal{D}'\left(\mathbb{R}\times X,E\right)$
such that $\left(t,x\right)\in\mathbb{R}\times\mathbb{R}_{+}^{n}\mapsto\left(\psi_{j}u\right)\left(t,\varphi_{j}^{-1}\left(x\right)\right)$
belong to $H_{p}^{s}\left(\mathbb{R}\times\mathbb{R}_{+}^{n},\mathbb{C}^{N}\right)$,
where $N$ is the dimension of $E$, with norm given by:
\[
\left\Vert u\right\Vert _{H_{p}^{s}\left(X\times\mathbb{R},E\right)}=\sum_{j=1}^{M}\left\Vert \left(\psi_{j}u\right)\left(t,\varphi_{j}^{-1}\left(x\right)\right)\right\Vert _{H_{p}^{s}\left(\mathbb{R}\times\mathbb{R}_{+}^{n},\mathbb{C}^{N}\right)}.
\]
The space $\mathcal{H}_{p}^{s,\gamma}\left(X^{\wedge},E\right)$
is the space of all distributions $u\in\mathcal{D}'\left(X^{\wedge},E\right)$
such that $u\left(t,x\right)=t^{-\frac{n+1}{2}+\gamma}v\left(\ln\left(t\right),x\right)$,
where $v\in H_{p}^{s}\left(X\times\mathbb{R},E\right)$. Its norm
is given by $\left\Vert u\right\Vert _{\mathcal{H}_{p}^{s,\gamma}\left(X^{\wedge},E\right)}:=\left\Vert v\right\Vert _{H_{p}^{s}\left(X\times\mathbb{R},E\right)}$.

Similarly, using the space $B_{p}^{s}\left(\mathbb{R}^{n},\mathbb{C}^{N}\right)$
instead of $H_{p}^{s}\left(\mathbb{R}\times\mathbb{R}_{+}^{n},\mathbb{C}^{N}\right)$,
we define the space $B_{p}^{s}\left(\partial X\times\mathbb{R},F\right)$,
where $N$ is the dimension of $F$. Associated to it is the space $\mathcal{B}_{p}^{s,\gamma}\left(\partial X^{\wedge},F\right)$ of all distributions $u\in\mathcal{D}'\left(\partial X^{\wedge},F\right)$
such that $u\left(t,x\right)=t^{-\frac{n}{2}+\gamma}v\left(\text{ln}\left(t\right),x\right),$
where $v\in B_{p}^{s}\left(\partial X\times\mathbb{R},F\right)$.
Its norm is given by $\left\Vert u\right\Vert _{\mathcal{B}_{p}^{s,\gamma}\left(\partial X^{\wedge},F\right)}:=\left\Vert v\right\Vert _{B_{p}^{s}\left(\partial X\times\mathbb{R},F\right)}.$
\end{defn}
\begin{rem}
\label{rem:s=00003D1 cone} The above definition implies $u\mapsto\left\Vert t\partial_{t}u\right\Vert _{L_{p}\left(X^{\wedge},E,dx\frac{dt}{t}\right)}+\left\Vert u\right\Vert _{L_{p}\left(\mathbb{R}_{+},H_{p}^{1}\left(X,E\right),\frac{dt}{t}\right)}$
is an equivalent norm for $\mathcal{H}_{p}^{1,\frac{n+1}{2}}\left(X^{\wedge},E\right)$.
\end{rem}
Finally, we need Bessel and Besov spaces with asymptotics. First let
us define asymptotic types.
\begin{defn}
We say that $P=\left\{ \left(p_{j},m_{j},L_{j}\right);\,j\in\left\{ 1,...,M\right\} \right\} $
is an asymptotic type for $C^{\infty}\left(X,E\right)$ with weight
$\left(\gamma,k\right)\in\mathbb{R}\times\mathbb{N}_{0}$ if $p_{j}\in\mathbb{C}$,
$\frac{n+1}{2}-\gamma-k<\text{Re}\left(p_{j}\right)<\frac{n+1}{2}-\gamma$,
are distinct numbers, $m_{j}\in\mathbb{N}_{0}$ and $L_{j}\subset C^{\infty}\left(X,E\right)$
are finite dimensional spaces. The set of all asymptotic types is
denoted by $As\left(X,E,\gamma,k\right)$. Similarly, we say that
$Q=\left\{ \left(p_{j},m_{j},L_{j}\right);\,j\in\left\{ 1,...,M\right\} \right\} $
is an asymptotic type for $C^{\infty}\left(\partial X,F\right)$ with
weight $\left(\gamma,k\right)\in\mathbb{R}\times\mathbb{N}_{0}$ and write $Q\in As\left(\partial X,F,\gamma,k\right)$, if $p_{j}\in\mathbb{C}$,
$\frac{n}{2}-\gamma-k<\text{Re}\left(p_{j}\right)<\frac{n}{2}-\gamma$,
are distinct numbers, $m_{j}\in\mathbb{N}_{0}$ and $L_{j}\subset C^{\infty}\left(\partial X,F\right)$
are finite dimensional spaces.
\end{defn}
\begin{defn}
The Bessel potential and Besov space with asymptotics, respectively,
are defined as follows:

1) Let $P=\left\{ \left(p_{j},m_{j},L_{j}\right);\,j\in\left\{ 1,...,M\right\} \right\} \in As\left(X,E,\gamma,k\right)$.
We define 
$$\mathcal{H}_{p,P}^{s,\gamma}\left(\mathbb{D},E\right)=\cap_{\epsilon>0}\mathcal{H}_{p}^{s,\gamma+k-\epsilon}\left(\mathbb{D},E\right)\oplus\mathcal{E}_{P}\left(X\right),$$
where $\mathcal{E}_{P}:=\left\{X^{\wedge}\ni \left(t,x\right)\mapsto \omega\left(t\right)\sum_{j=1}^{M}\sum_{k=0}^{m_{j}}t^{-p_{j}}\ln^{k}\left(t\right)v_{jk}\left(x\right),\,v_{jk}\in L_{j}\right\} $.

2) Let $\tilde{P}=\{ (\tilde{p}_{j},\tilde{m}_{j},\tilde{L}_{j});\,j\in\{ 1,...,M\} \} \in As \left(\partial X,F,\gamma,k\right)$.
We define 
$$\mathcal{B}_{p,\tilde{P}}^{s,\gamma} \left(\mathbb{B},F\right) 
=\cap_{\epsilon>0}\mathcal{B}_{p}^{s,\gamma+k-\epsilon}\left(\mathbb{B},F\right)\oplus\mathcal{E}_{\tilde{P}}\left(\partial X\right),$$
where 
$\mathcal{E}_{\tilde{P}}
:=\big\{\partial X^{\wedge}\ni\left(t,x\right) \mapsto \omega\left(t\right)\sum_{j=1}^{M}\sum_{k=0}^{\tilde{m}_{j}}t^{-\tilde{p}_{j}}\ln^{k}\left(t\right)v_{jk}\left(x\right),\,v_{jk}\in\tilde{L}_{j}\big\} $
and $\omega$ is a cut-off function.
\end{defn}
\begin{rem}

1) The scalar product of $L^{2}\left(\partial X^{\wedge},F,dx\frac{dt}{t}\right)$
allows the identification $\mathcal{B}_{p}^{s,\frac{n}{2}}\left(\partial X^{\wedge},F\right)'\cong\mathcal{B}_{q}^{-s,\frac{n}{2}}\left(\partial X^{\wedge},F\right)$,
$\frac{1}{p}+\frac{1}{q}=1$. As $\mathcal{B}_{p}^{s,\gamma}\left(\partial X^{\wedge},F\right)=t^{\gamma-\frac{n}{2}}\mathcal{B}_{p}^{s,\frac{n}{2}}\left(\partial X^{\wedge},F\right)$
and $\mathcal{B}_{q}^{-s,-\gamma}\left(\partial X^{\wedge},F\right)=t^{-\gamma-\frac{n}{2}}\mathcal{B}_{q}^{-s,\frac{n}{2}}\left(\partial X^{\wedge},F\right)$,
we conclude that $\mathcal{B}_{p}^{s,\gamma}\left(\partial X^{\wedge},F\right)'\cong\mathcal{B}_{q}^{-s,-\gamma}\left(\partial X^{\wedge},F\right)$,
if we use the scalar product of $L^{2}\left(\partial X^{\wedge},F,t^{n-1}dtdx\right)$.

2) In the same way, $\mathcal{H}_{p}^{s,\frac{n+1}{2}}\left(X^{\wedge},E\right)'\cong\mathcal{H}_{q}^{-s,\frac{n+1}{2}}\left(X^{\wedge},E\right)$,
if we use the scalar product of $L^{2}\left(X^{\wedge},E,dx\frac{dt}{t}\right)$,
and $\mathcal{H}_{p}^{s,\gamma}\left(X^{\wedge},E\right)'\cong\mathcal{H}_{q}^{-s,-\gamma}\left(X^{\wedge},E\right)$,
if we use that
of $L^{2}\left(X^{\wedge},E,t^{n}dtdx\right)$.
\end{rem}
\begin{defn}
Let $s,\gamma\in\mathbb{R}$ and $1<p<\infty$.

1) We define $\mathcal{H}_{p}^{s,\gamma}\left(\mathbb{D},E\right)$
as the space of all distributions $u\in H_{p,{\rm loc}}^{s}\left(\text{int}\left(\mathbb{D}\right),E\right)$
such that, for any cut-off function $\omega$, here considered as a function on $\mathbb D$, we have $\omega u\in\mathcal{H}_{p}^{s,\gamma}\left(X^{\wedge},E\right)$.
Its norm is given by
\[
\left\Vert u\right\Vert _{\mathcal{H}_{p}^{s,\gamma}\left(\mathbb{D},E\right)}:=\left\Vert \omega u\right\Vert _{\mathcal{H}_{p}^{s,\gamma}\left(X^{\wedge},E\right)}+\left\Vert \left(1-\omega\right)u\right\Vert _{H_{p,{\rm loc}}^{s}\left(\text{int}\left(\mathbb{D}\right),E\right)}.
\]

2) Similarly, we obtain  $\mathcal{B}_{p}^{s,\gamma}\left(\mathbb{B},F\right)$
from $\mathcal{B}_{p}^{s,\gamma}\left(\partial X^{\wedge},F\right)$ and $B_{p,{\rm loc}}^{s}\left(\text{int}\left(\mathbb{B}\right),F\right)$.
\end{defn}

\subsection{Classes of operators}

We are going to use the natural identification 
\[
\mathcal{T}_{\gamma}\left(\mathbb{R}_{+},C^{\infty}\left(X,E,F\right)\right)\cong\mathcal{T}_{\gamma}\left(\mathbb{R}_{+},C^{\infty}\left(X,E\right)\right)\oplus\mathcal{T}_{\gamma}\left(\mathbb{R}_{+},C^{\infty}\left(\partial X,F\right)\right)
\]
and write $\Gamma_{\sigma}:=\left\{ z\in\mathbb{C};\,Re\left(z\right)=\sigma\right\} $. The latter set will be obviously identified with $\mathbb{R}$, when it is convenient to do so.

\begin{defn}
The weighted Mellin transform is the continuous linear operator $\mathcal{M}_{\gamma}:\mathcal{T}_{\gamma}\left(\mathbb{R}_{+},C^{\infty}\left(X,E,F\right)\right)\to\mathcal{S}\left(\Gamma_{\frac{1}{2}-\gamma},C^{\infty}\left(X,E,F\right)\right)$
defined by
\[
\mathcal{M}_{\gamma}\varphi\left(z\right)=\int_{0}^{\infty}t^{z}\varphi\left(t\right)\frac{dt}{t},\,\,\,\,z\in\Gamma_{\frac{1}{2}-\gamma}.
\]
It is an invertible operator, whose inverse is given by
\[
\mathcal{M}_{\gamma}^{-1}\varphi\left(t\right)=\frac{1}{2\pi i}\int_{\Gamma_{\frac{1}{2}-\gamma}}t^{-z}\varphi\left(z\right)dz=\frac{1}{2\pi}\int t^{-\left(\frac{1}{2}-\gamma+i\tau\right)}\varphi\left(\frac{1}{2}-\gamma+i\tau\right)d\tau.
\]
\end{defn}

\begin{defn}
For $m\in\mathbb{Z}$ and $d\in\mathbb{N}_{0}$, $M\mathcal{B}_{E_{0},F_{0},E_{1},F_{1}}^{m,d}\left(X,\mathbb{R}_{+};\Gamma_{\gamma}\right)$ is the space of all functions $h\in C^{\infty}\left(\mathbb{R}_{+},\mathcal{B}_{E_{0},F_{0},E_{1},F_{1}}^{m,d}\left(X,\Gamma_{\gamma}\right)\right)$
that satisfy
\[
\sup\left\{ p\left(\left(t\partial_{t}\right)^{k}h\left(t\right)\right),\,t\in\mathbb{R}_{+}\right\} <\infty,
\]
for all continuous seminorms $p$ of $\mathcal{B}_{E_{0},F_{0},E_{1},F_{1}}^{m,d}\left(X,\Gamma_{\gamma}\right)$.
In a similar way we define $M\tilde{\mathcal{B}}_{E_{0},F_{0},E_{1},F_{1}}^{p}\left(X,\mathbb{R}_{+};\Gamma_{\gamma}\right)$.

To a function in $M\mathcal{B}_{E_{0},F_{0},E_{1},F_{1}}^{m,d}\left(X,\mathbb{R}_{+};\Gamma_{\gamma}\right)$
or $M\tilde{\mathcal{B}}_{E_{0},F_{0},E_{1},F_{1}}^{p}\left(X,\mathbb{R}_{+};\Gamma_{\gamma}\right)$ we  associate 
the Mellin operator 
$$op_{M}^{\gamma}\left(h\right):\mathcal{T}_{\gamma}\left(\mathbb{R}_{+},C^{\infty}\left(X,E_{0},F_{0}\right)\right)\to\mathcal{T}_{\gamma}\left(\mathbb{R}_{+},C^{\infty}\left(X,E_{1},F_{1}\right)\right)$$
by
\[
\left[op_{M}^{\gamma}\left(h\right)\varphi\right]\left(t\right)=\frac{1}{2\pi}\int_{\mathbb{R}}t^{-\left(\frac{1}{2}-\gamma+i\tau\right)}h\left(t,\frac{1}{2}-\gamma+i\tau\right)\left(\mathcal{M}_{\gamma}\varphi\right)\left(\frac{1}{2}-\gamma+i\tau\right)d\tau.
\]
\end{defn}
We also need to define the discrete Mellin asymptotic types:
\begin{defn}
A discrete Mellin asymptotic type of order $d\in\mathbb{N}_{0}$ is a set
\[
P=\left\{ \left(p_{j},m_{j},L_{j}\right)\right\} _{j\in\mathbb{Z}},
\]
where $p_{j}\in\mathbb{C}$ satisfy $Re\left(p_{j}\right)\to\pm\infty$
as $j\to\mp\infty$, $m_{j}\in\mathbb{N}_{0}$ and $L_{j}$ are finite-dimensional
subspaces of operators of finite rank in $\mathcal{B}_{E_{0},F_{0},E_{1},F_{1}}^{-\infty,d}\left(X\right)$.
The collection of all these asymptotic types is denoted by $As\left(\mathcal{B}_{E_{0},F_{0},E_{1},F_{1}}^{-\infty,d}\left(X\right)\right)$.
Moreover, we let $\pi_{\mathbb{C}}P:=\left\{ p_{j}:\,j\in\mathbb{Z}\right\} \subset\mathbb{C}$.
\end{defn}
The asymptotic types are used to define the following meromorphic
functions.
\begin{defn}
The space $M_{P\,E_{0},F_{0},E_{1},F_{1}}^{m,d}\left(X\right)$, $P\in As\left(\mathcal{B}_{E_{0},F_{0},E_{1},F_{1}}^{-\infty,d}\left(X\right)\right)$,
is the space of all meromorphic functions $a:\mathbb{C}\backslash\pi_{\mathbb{C}}P\to\mathcal{B}_{E_{0},F_{0},E_{1},F_{1}}^{m,d}\left(X\right)$
such that:

i) For every $p_{j}\in\pi_{\mathbb{C}}P$, there is a neighborhood
of $p_{j}$ where $a$ can be written as
\[
a\left(z\right)=\sum_{k=0}^{m_{j}}\nu_{jk}\left(z-p_{j}\right)^{-k-1}+a_{0}\left(z\right).
\]

Above, $a_{0}$ is a holomorphic function near $p_{j}$, with values
in $\mathcal{B}_{E_{0},F_{0},E_{1},F_{1}}^{m,d}\left(X\right)$ and
$\nu_{jk}\in L_{j}$, for $k=0,...,m_{j}$.

ii) For every $N\in\mathbb{N}_{0}$, the function $\gamma\in\left[-N,N\right]\mapsto a_{N}\left(\gamma+i.\right)\in\mathcal{B}_{E_{0},F_{0},E_{1},F_{1}}^{m,d}\left(X,\mathbb{R}\right)$
is continuous, where
\[
a_{N}\left(z\right):=a\left(z\right)-\sum_{\left|Re\left(p_{j}\right)\right|\le N}\sum_{k=0}^{m_{j}}\nu_{jk}\left(z-p_{j}\right)^{-k-1}.
\]

For $P\in As\left(\mathcal{B}_{E_{0},F_{0},E_{1},F_{1}}^{-\infty,0}\left(X\right)\right)$,
we can also define $\tilde{M}_{P\,E_{0},F_{0},E_{1},F_{1}}^{p}\left(X\right)$
replacing $\mathcal{B}_{E_{0},F_{0},E_{1},F_{1}}^{m,d}\left(X\right)$
by $\tilde{\mathcal{B}}_{E_{0},F_{0},E_{1},F_{1}}^{p}\left(X\right)$.
When $P=\emptyset$, we also use the notations $M_{\mathcal{O}\,E_{0},F_{0},E_{1},F_{1}}^{m,d}\left(X\right)$
and $\tilde{M}_{\mathcal{O}\,E_{0},F_{0},E_{1},F_{1}}^{p}\left(X\right)$.
\end{defn}
The last operator that we need are the Green ones.
\begin{defn}
We define $C_{G\,E_{0},F_{0},E_{1},F_{1}}^{d}\left(\mathbb{D};\gamma,\gamma',k\right)$
as the space of operators of the form
\[
\mathcal{G}=\sum_{j=0}^{d}\mathcal{G}_{j}\left(\begin{array}{cc}
D^{j} & 0\\
0 & 1
\end{array}\right)+\mathcal{G}_{0},
\]
where, for each $\mathcal{G}_{j}$, there exist asymptotic types $P\in As\left(X,E_{1},\gamma',k\right)$ and 
$P'\in As\left(X,E_{0},-\gamma,k\right)$, $Q\in As\left(\partial X,F_{1},\gamma'-\frac{1}{2},k\right)$
and $Q'\in As\left(\partial X,F_{0},-\gamma-\frac{1}{2},k\right)$,
such that $\mathcal{G}_{j}$ and its formal adjoint with respect to
$\mathcal{H}_{2}^{0,0}\left(\mathbb{D},E_{j}\right)\oplus\mathcal{B}_{2}^{-\frac{1}{2},-\frac{1}{2}}\left(\mathbb{B},F_{j}\right)$,
$j=0,1$, define continuous operators:
\[
\mathcal{G}_{j}:\!\!\!\begin{array}{c}
\mathcal{\mathcal{H}}_{p}^{s,\gamma}\left(\mathbb{D},E_{0}\right)\\
\oplus\\
\mathcal{\mathcal{B}}_{p}^{r,\gamma-\frac{1}{2}}\left(\mathbb{B},F_{0}\right)
\end{array}
\!\!\!\to\!\!\!
\begin{array}{c}
\mathcal{\mathcal{H}}_{p,P}^{\infty,\gamma'}\left(\mathbb{D},E_{1}\right)\\
\oplus\\
\mathcal{\mathcal{B}}_{p,Q}^{\infty,\gamma'-\frac{1}{2}}\left(\mathbb{B},F_{1}\right)
\end{array}
\text{, }\mathcal{G}_{j}^{*}:\!\!\!\begin{array}{c}
\mathcal{\mathcal{H}}_{q}^{s,-\gamma'}\left(\mathbb{D},E_{1}\right)\\
\oplus\\
\mathcal{\mathcal{B}}_{q}^{r,-\gamma'-\frac{1}{2}}\left(\mathbb{B},F_{1}\right)
\end{array}
\!\!\!\to\!\!\!
\begin{array}{c}
\mathcal{\mathcal{H}}_{q,P'}^{\infty,-\gamma}\left(\mathbb{D},E_{0}\right)\\
\oplus\\
\mathcal{\mathcal{B}}_{q,Q'}^{\infty,-\gamma-\frac{1}{2}}\left(\mathbb{B},F_{0}\right)
\end{array}
\]
for all $r\in\mathbb{R}$, $s>-1+\frac{1}{p}$ on the left hand side and
$s>-1+\frac{1}{q}$ on the right hand side. Near the boundary $\mathbb{B}$ of $\mathbb{D}$, the operators $D^{j}$ coincide
with $\left(-i\partial_{\nu}\right)^{j}$ where $\partial_{\nu}$
is the normal derivative.

Similarly, $\tilde{C}_{G\,E_{0},F_{0},E_{1},F_{1}}^{p}\left(\mathbb{D};k\right)$
denotes the space of all operators $\mathcal{G}$ for which there exist asymptotic
types $P\in As\left(X,E_{1},\frac{n+1}{2},k\right)$, $P'\in As\left(X,E_{0},\frac{n+1}{2},k\right)$,
$Q\in As\left(\partial X,F_{1},\frac{n}{2},k\right)$ and $Q'\in As\left(\partial X,F_{0},\frac{n}{2},k\right)$,
such that $\mathcal{G}$ and its formal adjoint $\mathcal{G}^{*}$
with respect to $\mathcal{H}_{2}^{0,\frac{n+1}{2}}\left(\mathbb{D},E_{j}\right)\oplus\mathcal{B}_{2}^{0,\frac{n}{2}}\left(\mathbb{B},F_{j}\right)$,
$j=0,1$, define continuous operators:
\[
\mathcal{G}_{j}:\!\!\!\begin{array}{c}
\mathcal{\mathcal{H}}_{p}^{s,\frac{n+1}{2}}\left(\mathbb{D},E_{0}\right)\\
\oplus\\
\mathcal{\mathcal{B}}_{p}^{r,\frac{n}{2}}\left(\mathbb{B},F_{0}\right)
\end{array}\to\begin{array}{c}
\mathcal{\mathcal{H}}_{p,P}^{\infty,\frac{n+1}{2}}\left(\mathbb{D},E_{1}\right)\\
\oplus\\
\mathcal{\mathcal{B}}_{p,Q}^{\infty,\frac{n}{2}}\left(\mathbb{B},F_{1}\right)
\end{array}\text{, }\mathcal{G}_{j}^{*}:\!\!\!\begin{array}{c}
\mathcal{\mathcal{H}}_{q}^{s,\frac{n+1}{2}}\left(\mathbb{D},E_{1}\right)\\
\oplus\\
\mathcal{\mathcal{B}}_{q}^{r,\frac{n}{2}}\left(\mathbb{B},F_{1}\right)
\end{array}\to\begin{array}{c}
\mathcal{\mathcal{H}}_{q,P'}^{\infty,\frac{n+1}{2}}\left(\mathbb{D},E_{0}\right)\\
\oplus\\
\mathcal{\mathcal{B}}_{q,Q'}^{\infty,\frac{n}{2}}\left(\mathbb{B},F_{0}\right)
\end{array}
\]
for all $r\in\mathbb{R}$, $s>-1+\frac{1}{p}$ on the left hand side and
$s>-1+\frac{1}{q}$ on the right hand side.
\end{defn}

It is an immediate consequence of the embedding properties for cone Sobolev spaces that, for $r\in \mathbb R$, $s>d+1/p-1$ and arbitrary $r',s'\in \mathbb R$, an operator 
$$C_{G\,E_{0},F_{0},E_{1},F_{1}}^{d}\left(\mathbb{D};\gamma,\gamma',k\right)
\ni\mathcal{G}:
\begin{array}{c}
\mathcal{\mathcal{H}}_{p}^{s,\gamma}\left(\mathbb{D},E_{0}\right)\\
\oplus\\
\mathcal{\mathcal{B}}_{p}^{r,\gamma-\frac{1}{2}}\left(\mathbb{B},F_{0}\right)
\end{array}
\!\!\!\to\!\!\!
\begin{array}{c}
\mathcal{\mathcal{H}}_{p}^{s',\gamma'}\left(\mathbb{D},E_{1}\right)\\
\oplus\\
\mathcal{\mathcal{B}}_{p}^{r',\gamma'-\frac{1}{2}}\left(\mathbb{B},F_{1}\right)
\end{array}
$$
is compact. An analogous statement applies to operators in 
$\tilde{C}_{G\,E_{0},F_{0},E_{1},F_{1}}^{p}\left(\mathbb{D};k\right)$.

Finally, we can define the cone algebra for boundary value problems.
\begin{defn}
\label{def:boutetconical}For $\gamma\in\mathbb{R}$, $m\in\mathbb{Z}$,
$d\in\mathbb{N}_{0}$ and $k\in\mathbb{N}_{0}$ we define the space $C_{E_{0},F_{0},E_{1},F_{1}}^{m,d}\left(\mathbb{D},\left(\gamma,\gamma-m,k\right)\right)$
of all  operators $A:C_{\gamma}^{\infty}\left(\mathbb{D},E_{0}\right)\oplus C_{\gamma-\frac{1}{2}}^{\infty}\left(\mathbb{B},F_{0}\right)\to C_{\gamma-m}^{\infty}\left(\mathbb{D},E_{1}\right)\oplus C_{\gamma-m-\frac{1}{2}}^{\infty}\left(\mathbb{B},F_{1}\right)$
of the form
\begin{equation}
A=\omega_{0}A_{M}\omega_{1}+\left(1-\omega_{2}\right)A_{\psi}\left(1-\omega_{3}\right)+M+G, \label{eq:CBVP}
\end{equation} where $\omega_{1},\,...,\,\omega_{4}\in C^{\infty}\left[0,1\right[$
are cut-off functions. The operator $A_{M}$ is a Mellin operator:
$A_{M}=t^{-m}op_{M}^{\gamma-\frac{n}{2}}\left(h\right)$, with $h\in C^{\infty}(\overline{\mathbb{R}_{+}},M_{O\,E_{0},F_{0},E_{1},F_{1}}^{m,d}\left(X\right))$.
The operator $A_{\psi}$ is a Boutet de Monvel operator $A_{\psi}\in\mathcal{B}_{2E_{0},2F_{0},2E_{1},2F_{1}}^{m,d}\left(2\mathbb{D}\right)$.
The operator $M$ is a smoothing Mellin operator: $M=\omega_{0}\left(\sum_{l=0}^{k-1}t^{-m+l}op_{M}^{\gamma_{l}-\frac{n}{2}}\left(h_{l}\right)\right)\omega_{1}$
with $h_{l}\in M_{P_{l}\,E_{0},F_{0},E_{1},F_{1}}^{-\infty,d}\left(X\right)$,
$\pi_{\mathbb{C}}P_{l}\cap\Gamma_{\frac{n+1}{2}-\gamma_{l}}=\emptyset$,
and $\gamma-l\le\gamma_{l}\le\gamma$. The operator $G$ is a Green
operator: $G\in C_{G\,E_{0},F_{0},E_{1},F_{1}}^{d}\left(\mathbb{D};\gamma,\gamma-m,k\right)$.

Similarly, the algebra $\tilde{C}_{E_{0},F_{0},E_{1},F_{1}}^{p}\left(\mathbb{D},k\right)$
is defined as the space of all continuous operators $A:C_{\frac{n+1}{2}}^{\infty}\left(\mathbb{D},E_{0}\right)\oplus C_{\frac{n}{2}}^{\infty}\left(\mathbb{B},F_{0}\right)\to C_{\frac{n+1}{2}}^{\infty}\left(\mathbb{D},E_{1}\right)\oplus C_{\frac{n}{2}}^{\infty}\left(\mathbb{B},F_{1}\right)$
of the form (\ref{eq:CBVP}), where $A_{M}=op_{M}^{\frac{1}{2}}\left(h\right)$, with $h\in C^{\infty}\left(\overline{\mathbb{R}_{+}},\tilde{M}_{O\,E_{0},F_{0},E_{1},F_{1}}^{p}\left(X\right)\right)$, $A_{\psi}\in\tilde{\mathcal{B}}_{2E_{0},2F_{0},2E_{1},2F_{1}}^{p}\left(2\mathbb{D}\right)$, $M=\omega_{0}\left(\sum_{l=0}^{k-1}t^{l}op_{M}^{\gamma_{l}-\frac{n}{2}}\left(h_{l}\right)\right)\omega_{1}$
with $h_{l}\in M_{P_{l}\,E_{0},F_{0},E_{1},F_{1}}^{-\infty,0}\left(X\right)$,
$\pi_{\mathbb{C}}P_{l}\cap\Gamma_{\frac{n+1}{2}-\gamma_{l}}=\emptyset$,
and $\frac{n+1}{2}-l\le\gamma_{l}\le\frac{n+1}{2}$, and $G\in\tilde{C}_{G\,E_{0},F_{0},E_{1},F_{1}}^{p}\left(\mathbb{D},k\right)$.
\end{defn}
\begin{defn}\label{defn:ellipticity}
(Ellipticity) Using the notation of Definition \ref{def:boutetconical}\label{def:Ellipticity-cone},
we say that $A\in C_{E_{0},F_{0},E_{1},F_{1}}^{m,d}\left(\mathbb{D},\left(\gamma,\gamma-m,k\right)\right)$, $d\le\max\{0,m\}$,
is elliptic if:

1) Outside the singularity $X\times\left\{ 0\right\} $, $A$ is an elliptic
Boutet de Monvel operator in $\mathcal{B}_{E_{0},F_{0},E_{1},F_{1}}^{m,d}\left(\text{int}\left(\mathbb{D}\right)\right)$:
Its interior symbol and boundary symbol are invertible at each point.

2) Its conormal symbol $\sigma_{M}\left(A\right)\left(z\right):=h\left(0,z\right)+h_{0}\left(z\right):H_{p}^{s}\left(X,E_{0}\right)\oplus B_{p}^{s-\frac{1}{p}}\left(X,F_{0}\right)\to H_{p}^{s-m}\left(X,E_{1}\right)\oplus B_{p}^{s-\frac{1}{p}-m}\left(X,F_{1}\right)$,
$s>d-1+\frac{1}{p}$, is invertible for each $z\in\Gamma_{\frac{n+1}{2}-\gamma}$
and its inverse belongs to $B_{E_{1},F_{1},E_{0},F_{0}}^{-m,d'}(X,\Gamma_{\frac{n+1}{2}-\gamma})$, $d'=\max\{-m,0\}$.

Similarly, we say that $A\in\tilde{C}_{E_{0},F_{0},E_{1},F_{1}}^{p}\left(\mathbb{D},k\right)$
is an elliptic operator if, outside the singularity $X\times\left\{ 0\right\} $,
$A$ is an elliptic operator in $\tilde{\mathcal{B}}_{E_{0},F_{0},E_{1},F_{1}}^{p}\left(\text{int}\left(\mathbb{D}\right)\right)$
and its conormal symbol $\sigma_{M}\left(A\right)\left(z\right):=h\left(0,z\right)+h_{0}\left(z\right):H_{p}^{s}\left(X,E_{0}\right)\oplus B_{p}^{s}\left(X,F_{0}\right)\to H_{p}^{s}\left(X,E_{1}\right)\oplus B_{p}^{s}\left(X,F_{1}\right)$
is invertible for each $z\in\Gamma_{\frac{n+1}{2}}$, $s>d-1+\frac{1}{p}$,
and its inverse belongs to $\tilde{B}_{E_{1},F_{1},E_{0},F_{0}}^{p}(X,\Gamma_{\frac{n+1}{2}})$.
\end{defn}

\begin{rem}\label{ellipticity}
Definition \ref{defn:ellipticity} follows \cite{SchroheSchulzeConicalBoundaryII}. 
Instead one might ask that
\begin{enumerate}
\item The principal pseudodifferential symbol $\sigma_\psi(A)$ is invertible on 
$T^*(\text{int}\,\mathbb D)\setminus \{0\}$ and, in local coordinates $(t,x,\tau,\xi)$ for the 
cotangent space in a collar neighborhood of the conical point, 
$t^m \sigma_\psi(A)(t,x, \tau/t, \xi)$ is smoothly invertible up to $t=0$. 
\item The boundary principal symbol $\sigma_\partial (A)$ is invertible on 
$T^*(\text{int}\, \mathbb B)\setminus \{0\}$ and, in local coordinates $(t,y, \tau, \eta) $ 
for the cotangent space in a collar neighborhood of the conical point, 
$t^m \sigma_\partial(A)(t,y, \tau/t, \eta)$ is smoothly invertible up to $t=0$. 
\item The conormal symbol is pointwise invertible. 
\end{enumerate} 
See \cite[Section 6.2.1]{HarutyunyanSchulze} for details. 
%
\end{rem}

\begin{prop}
\label{prop:propriedadesconecalculus}The operators in $C_{E_{0},F_{0},E_{1},F_{1}}^{m,d}\left(\mathbb{D},\left(\gamma,\gamma-m,k\right)\right)$
and $\tilde{C}_{E_{0},F_{0},E_{1},F_{1}}^{p}\left(\mathbb{D},k\right)$
have the following properties:

1) If $B\in C_{E_{0},F_{0},E_{1},F_{1}}^{m_{0},d_{0}}\left(\mathbb{D},\left(\gamma,\gamma-m_{0},k\right)\right)$
and $A\in C_{E_{1},F_{1},E_{2},F_{2}}^{m_{1},d_{1}}\left(\mathbb{D},\left(\gamma-m_{0},\gamma-m_{2},k\right)\right)$
then $AB\in C_{E_{0},F_{0},E_{2},F_{2}}^{m_{2},d_{2}}\left(\mathbb{D},\left(\gamma,\gamma-m_{2},k\right)\right)$,
where $m_{2}:=m_{0}+m_{1}$ and $d_{2}:=\max\left\{ m_{0}+d_{1},d_{0}\right\} $.
If $B\in\tilde{C}_{E_{0},F_{0},E_{1},F_{1}}^{p}\left(\mathbb{D},k\right)$
and $A\in\tilde{C}_{E_{1},F_{1},E_{2},F_{2}}^{p}\left(\mathbb{D},k\right)$,
then $AB\in\tilde{C}_{E_{0},F_{0},E_{2},F_{2}}^{p}\left(\mathbb{D},k\right)$.

2) If $A\in C_{E_{0},F_{0},E_{1},F_{1}}^{m,d}\left(\mathbb{D},\left(\gamma,\gamma-m,k\right)\right)$,
then $A$ extends to a continuous operator:
\[
A:\begin{array}{c}
\mathcal{H}_{p}^{s,\gamma}\left(\mathbb{D},E_{0}\right)\\
\oplus\\
\mathcal{B}_{p}^{s-\frac{1}{p},\gamma-\frac{1}{2}}\left(\mathbb{B},F_{0}\right)
\end{array}\to\begin{array}{c}
\mathcal{H}_{p}^{s-m,\gamma-m}\left(\mathbb{D},E_{1}\right)\\
\oplus\\
\mathcal{B}_{p}^{s-m-\frac{1}{p},\gamma-m-\frac{1}{2}}\left(\mathbb{B},F_{1}\right)
\end{array},\,\,s>d-1+\frac{1}{p}.
\]

If $A\in\tilde{C}_{E_{0},F_{0},E_{1},F_{1}}^{p}\left(\mathbb{D},k\right)$,
then $A$ extends to a continuous operator:
\[
A:\begin{array}{c}
\mathcal{H}_{p}^{s,\frac{n+1}{2}}\left(\mathbb{D},E_{0}\right)\\
\oplus\\
\mathcal{B}_{p}^{s,\frac{n}{2}}\left(\mathbb{B},F_{0}\right)
\end{array}\to\begin{array}{c}
\mathcal{H}_{p}^{s,\frac{n+1}{2}}\left(\mathbb{D},E_{1}\right)\\
\oplus\\
\mathcal{B}_{p}^{s,\frac{n}{2}}\left(\mathbb{B},F_{1}\right)
\end{array},\,\,s>-1+\frac{1}{p}.
\]

3) If $A\in\tilde{C}_{E_{0},F_{0},E_{1},F_{1}}^{p}\left(\mathbb{D},k\right)$,
then its formal adjoint with respect to the inner product in $\mathcal{H}_{2}^{0,\frac{n+1}{2}}\left(\mathbb{D},E_{j}\right)\oplus\mathcal{B}_{2}^{0,\frac{n}{2}}\left(\mathbb{B},F_{j}\right)$
belongs to $\tilde{C}_{E_{1},F_{1},E_{0},F_{0}}^{q}\left(\mathbb{D},k\right)$,
for $\frac{1}{p}+\frac{1}{q}=1$. If $A$ is elliptic, so is its adjoint.

4) If $A\in C_{E_{0},F_{0},E_{1},F_{1}}^{m,d}\left(\mathbb{D},\left(\gamma,\gamma-m,k\right)\right)$
is elliptic, $d:=\max\left\{ m,0\right\} $, then there
is an operator $B\in C_{E_{1},F_{1},E_{0},F_{0}}^{-m,d'}\left(\mathbb{D},\left(\gamma-m,\gamma,k\right)\right)$,
$d':=\max\left\{ -m,0\right\} $, such that
\[
BA-I\in C_{G\,E_{0},F_{0},E_{0},F_{0}}^{d}\left(\mathbb{D},\left(\gamma,\gamma,k\right)\right); AB-I\in C_{G\,E_{1},F_{1},E_{1},F_{1}}^{d'}\left(\mathbb{D},\left(\gamma-m,\gamma-m,k\right)\right).
\]

Similarly, if $\tilde{A}\in\tilde{C}_{E_{0},F_{0},E_{1},F_{1}}^{p}\left(\mathbb{D},k\right)$
is elliptic, then there is an operator $\tilde{B}\in\tilde{C}_{E_{1},F_{1},E_{0},F_{0}}^{p}\left(\mathbb{D},k\right)$,
such that
\[
\tilde B\tilde A-I\in\tilde{C}_{G\,E_{0},F_{0},E_{0},F_{0}}^{p}\left(\mathbb{D},k\right)\,\,\,\text{and}\,\,\,\tilde A\tilde B-I\in\tilde{C}_{G\,E_{1},F_{1},E_{1},F_{1}}^{p}\left(\mathbb{D},k\right).
\]
In particular, $A$ and $\tilde A$ are then Fredholm operators.

5) (Existence of order reducing operators) For $m\in\mathbb{Z}$, $m'\in \mathbb R$
and $\gamma\in\mathbb{R}$, there are elliptic  operators $A\in t^m C_{E_{0},E_{0}}^{m,0}\left(\mathbb{D},\left(\gamma,\gamma-m,k\right)\right)$
and $B\in t^{m'}C_{F_{0},F_{0}}^{m',0}\left(\mathbb{B},\left(\gamma,\gamma-m',k\right)\right)$,
such that $A:\mathcal{H}_{p}^{s,\gamma}\left(\mathbb{D},E_{0}\right)\to\mathcal{H}_{p}^{s-m,\gamma-m}\left(\mathbb{D},E_{0}\right)$
and $B:\mathcal{B}_{p}^{s-\frac{1}{p},\gamma-\frac{1}{2}}\left(\mathbb{B},F_{0}\right)\to\mathcal{B}_{p}^{s-m'-\frac{1}{p},\gamma-m'-\frac{1}{2}}\left(\mathbb{B},F_{0}\right)$
are invertible for all $s>-1+\frac{1}{p}$, see \cite[Section 6.4]{HarutyunyanSchulze}.
\end{prop}

\subsection{The equivalence between the Fredholm property and the ellipticity}

\begin{thm}\label{thm:EquivFredEll}
Let $A\in\tilde{C}_{E_{0},F_{0},E_{1},F_{1}}^{p}\left(\mathbb{D},k\right)$.
Then the following conditions are equivalent:

i) $A$ is elliptic.

ii) $A:\begin{array}{c}
\mathcal{H}_{p}^{0,\frac{n+1}{2}}\left(\mathbb{D},E_{0}\right)\\
\oplus\\
\mathcal{B}_{p}^{0,\frac{n}{2}}\left(\mathbb{B},F_{0}\right)
\end{array}\to\begin{array}{c}
\mathcal{H}_{p}^{0,\frac{n+1}{2}}\left(\mathbb{D},E_{1}\right)\\
\oplus\\
\mathcal{B}_{p}^{0,\frac{n}{2}}\left(\mathbb{B},F_{1}\right)
\end{array}$ is Fredholm.
\end{thm}
That $i)$ implies $ii)$ follows from the existence of a parametrix
of an elliptic operator, as it is stated in item $4)$ of Proposition
\ref{prop:propriedadesconecalculus}. It remains to prove that $ii)$
implies $i)$. If $A$ is Fredholm, then  condition $1)$ of Definition
\ref{def:Ellipticity-cone} holds by Theorem \ref{thm:equivalenciaFredholmelipticBdM}.
In fact, the proof of Theorem \ref{thm:equivalenciaFredholmelipticBdM}
is local, so it applies in this context.
In the next two subsections, we will show that condition 
$2)$ of Definition \ref{def:Ellipticity-cone} holds.
We rely on the arguments in \cite[Section 3.1]{SchroheSeilerSpectConical};
however, the Besov space estimates need more attention. 
Before, however, we note the following consequence.

%

\begin{cor}\label{equivalence}
For $A\in C_{E_{0},F_{0},E_{1},F_{1}}^{m,d}(\mathbb D,\gamma, \gamma-m, k)$, 
$m\in \mathbb Z$, $d\le \max\{m,0\}$, $p\in\, ]1,\infty[$, $s\in \mathbb Z$, $s\ge d$, the following are equivalent:

{\rm i)} $A$ is elliptic.

{\rm ii)}
\begin{eqnarray}\label{eqn:FredholmA}
A:\begin{array}{c}
\mathcal{H}_{p}^{s,\gamma}\left(\mathbb{D},E_{0}\right)\\
\oplus\\
\mathcal{B}_{p}^{s-\frac{1}{p},\gamma-\frac{1}{2}}\left(\mathbb{B},F_{0}\right)
\end{array}\to\begin{array}{c}
\mathcal{H}_{p}^{s-m,\gamma-m}\left(\mathbb{D},E_{1}\right)\\
\oplus\\
\mathcal{B}_{p}^{s-m-\frac{1}{p},\gamma-m-\frac{1}{2}}\left(\mathbb{B},F_{1}\right)
\end{array}\text{ is Fredholm.}
\end{eqnarray}
 
In particular, the Fredholm property is independent of $p$ and $s$, subject to the condition $s\in \mathbb Z$, $s\ge d$. The same is then true for the kernel and the index.
\end{cor}

\begin{proof}
According to item 4) of Proposition \ref{prop:propriedadesconecalculus}, ellipticity implies the Fredholm property. 
In order to see the converse, we note that, by item 5 of Proposition \ref{prop:propriedadesconecalculus}, we find   operators 
$P^{-s}\in t^{-s}C_{E_0,E_0}^{-s,0}(\mathbb D,\frac{n+1}2, \frac{n+1}2+s,k)$ and 
$Q^{s-m}\in t^{s-m} C^{s-m,0}_{E_1,E_1}(\mathbb D,\frac{n+1}2, \frac{n+1}2-s+m,k)$ defined on $\mathbb{D}$, 
$\tilde P^{-s+1/p}\in t^{-s+1/p} C^{-s+1/p}_{F_0,F_0}(\mathbb B,\frac n2, \frac n2+s-\frac1p,k)$ and 
$\tilde Q^{s-m+1/p}\in t^{s-m+1/p} C^{s-m+1/p}_{F_1,F_1}(\mathbb B,\frac n2, \frac n2-s+m-\frac 1p, k)$, defined on $\mathbb{B}$, such that 
$P^{-s}:\mathcal{H}_{p}^{0,\frac{n+1}{2}}\left(\mathbb{D},E_{0}\right)\to\mathcal{H}_{p}^{s,\frac{n+1}{2}}\left(\mathbb{D},E_{0}\right)$,
$\tilde{P}^{-s+\frac{1}{p}}:\mathcal{B}_{p}^{0,\frac{n}{2}}\left(\mathbb{B},F_{0}\right)\to\mathcal{B}_{p}^{s-\frac{1}{p},\frac{n}{2}}\left(\mathbb{B},F_{0}\right)$,
$Q^{s-m}:\mathcal{H}_{p}^{s-m,\frac{n+1}{2}}\left(\mathbb{D},E_{1}\right)\to\mathcal{H}_{p}^{0,\frac{n+1}{2}}\left(\mathbb{D},E_{1}\right)$, and
$\tilde{Q}^{s-m+\frac{1}{p}}:\mathcal{B}_{p}^{s-m-\frac{1}{p},\frac{n}{2}}\left(\mathbb{B},F_{1}\right)\to\mathcal{B}_{p}^{0,\frac{n}{2}}\left(\mathbb{B},F_{1}\right)$
are invertible. Here we use that $s\in \mathbb Z$. Since $A$ is a Fredholm operator,  the operator 
$\tilde A\in  \tilde C^p_{E_0,F_0,E_1,F_1}(\mathbb D,k)$, defined by
\begin{eqnarray}\label{eq:reduction}
\tilde A=
\begin{pmatrix}
Q^{s-m} & 0\\
0 & \tilde{Q}^{s-m+\frac{1}{p}}
\end{pmatrix}t^{\frac{n+1}{2}-\gamma+m}At^{-\frac{n+1}{2}+\gamma}
\begin{pmatrix}
P^{-s} & 0\\
0 & \tilde{P}^{-s+\frac{1}{p}}
\end{pmatrix}
\end{eqnarray}
is a Fredholm operator in $\mathcal B(\mathcal{H}_{p}^{0,\frac{n+1}{2}}(\mathbb{D},E_{0})\oplus
\mathcal{B}_{p}^{0,\frac{n}{2}}(\mathbb{B},F_{0}),
\mathcal{H}_{p}^{0,\frac{n+1}{2}}(\mathbb{D},E_{1})
\oplus
\mathcal{B}_{p}^{0,\frac{n}{2}}(\mathbb{B},F_{1}))$. 
By Theorem \ref{thm:EquivFredEll}, $\tilde A$ is elliptic,  hence so is $A$. 
As a consequence, the Fredholm property is independent of $p$ and $s$.

Suppose $A$ is elliptic and 
$u\in \mathcal{H}_{p}^{s,\gamma}(\mathbb{D},E_{0})\oplus
\mathcal{B}_{p}^{s-\frac{1}{p},\gamma-\frac{1}{2}}(\mathbb{B},F_{0})$ belongs to the
kernel of $A$. Then the existence of a parametrix and the mapping properties of the 
Green operators imply that, for some $\epsilon>0$ and all $t\in \mathbb R$, 
$u\in \mathcal{H}_{p}^{t,\gamma+\epsilon}(\mathbb{D},E_{0})\oplus
\mathcal{B}_{p}^{t-\frac{1}{p},\gamma+\epsilon-\frac{1}{2}}(\mathbb{B},F_{0})$.
Thus $u$ also is an element of 
$ \mathcal{H}_{q}^{t,\gamma}(\mathbb{D},E_{0})\oplus
\mathcal{B}_{q}^{t-\frac{1}{p},\gamma-\frac{1}{2}}(\mathbb{B},F_{0})$ for $1<q<\infty$ 
and $t\in \mathbb R$ and belongs to the kernel of $A$ on that space. This shows the independence of the kernel on $s$ and $p$. 

We next consider the formal adjoint  ${\tilde A}'$ of $\tilde A$ in the sense of item 3) of Proposition \ref{prop:propriedadesconecalculus}, which is an elliptic element of 
$\tilde C^q_{E_1,F_1,E_0,F_0}(\mathbb D,k)$, where $1/p+1/q=1$. 
Its extension to an operator in 
$\mathcal B(\mathcal{H}_{q}^{0,\frac{n+1}{2}}(\mathbb{D},E_{1})\oplus
\mathcal{B}_{q}^{0,\frac{n}{2}}(\mathbb{B},F_{1}),
\mathcal{H}_{q}^{0,\frac{n+1}{2}}(\mathbb{D},E_{0})
\oplus
\mathcal{B}_{q}^{0,\frac{n}{2}}(\mathbb{B},F_{0}))$
furnishes the adjoint to the operator $\tilde A$ acting as in \eqref{eq:reduction}. 
The index of $\tilde A$ then is the difference of the kernel dimensions of $\tilde A$ and 
${\tilde A}'$. By the same argument as above, these are independent of $p$ and $q$.
Hence the index of $\tilde A$ is independent of $p$ and the index of $A$ is independent 
of $s$ and $p$.  
\end{proof}

\subsection{Besov-space preliminaries}
 
Given dyadic partitions of unity $\left\{ \varphi_{j}\right\} _{j\in\mathbb{N}_{0}}\subset C_{c}^{\infty}(\mathbb{R})$
and $\left\{ \tilde{\varphi}_{j}\right\} _{j\in\mathbb{N}_{0}}\subset C_{c}^{\infty}(\mathbb{R}^{n-1})$
of $\mathbb{R}$ and $\mathbb{R}^{n-1}$, respectively, we define a dyadic partition of unity $\left\{ \psi_{j}\right\} _{j\in\mathbb{N}_{0}}\subset C_{c}^{\infty}(\mathbb{R}^{n})$ of $\mathbb{R}^{n}$ by
\begin{eqnarray*}
\psi_{0}\left(t,x\right)&:=&\varphi_{0}\left(t\right)\tilde{\varphi}_{0}\left(x\right),\\
\psi_{j}\left(t,x\right)&:=&\varphi_{j}\left(t\right)\Big(\sum_{k=0}^{j}\tilde{\varphi}_{k}\left(x\right)\Big)+\Big(\sum_{k=0}^{j-1}\varphi_{k}\left(t\right)\Big)\tilde{\varphi}_{j}\left(x\right)\\
&=&\psi_{0}\left(2^{-j}t,2^{-j}x\right)-\psi_{0}\left(2^{-j+1}t,2^{-j+1}x\right),\,\,\,j\ge1.
\end{eqnarray*}

Then $\text{supp}\left(\psi_{0}\right)\subset\left\{ \left(t,x\right)\in\mathbb{R}^{n};\left\Vert \left(t,x\right)\right\Vert_{N}<2\right\}$  and $\text{supp}\left(\psi_{j}\right)\subset\{ \left(t,x\right)\in\mathbb{R}^{n};$ $2^{j-1}<\left\Vert \left(t,x\right)\right\Vert_{N}<2^{j+1}\}$, for $j\ge1$. Here $\left\Vert \left(t,x\right)\right\Vert _{N}$ denotes the norm
\[
\left\Vert \left(t,x\right)\right\Vert _{N}=\max\left\{ \left|x\right|,\left|t\right|\right\} ,
\]
where $\left|x\right|$ denotes the Euclidean norm of $x\in\mathbb{R}^{n-1}$
and $\left|t\right|$ denotes the modulus of $t\in\mathbb{R}$. 

Item 8 of Remark \ref{rem:BesovandBesselSpaces} implies that we can choose the
following norm for $B_{p}^{s}\left(\mathbb{R}^{n}\right)$:
\[
\left\Vert f\right\Vert _{B_{p}^{s}\left(\mathbb{R}^{n}\right)}:=\Big(\sum_{j=0}^{\infty}2^{jsp}\left\Vert \psi_{j}\left(D\right)f\right\Vert _{L_{p}\left(\mathbb{R}^{n}\right)}^{p}\Big)^{\frac{1}{p}}.
\]

The next spaces are very useful for computations.
\begin{defn}
Let $G$ be a Banach space that is a $UMD$ Banach space with the
property $\left(\alpha\right)$. We define $\mathcal{B}_{p}^{s,\frac{1}{2}}\left(\mathbb{R}_{+},G\right)$
as the space of all $u\in\mathcal{D}'\left(\mathbb{R}_{+},G\right)$
such that $\left(\mathbb{R}\ni t\mapsto u\left(e^{-t}\right)\right)\in B_{p}^{s}\left(\mathbb{R},G\right)$
and $\mathcal{H}_{p}^{s,\frac{1}{2}}\left(\mathbb{R}_{+},G\right)$
as the set of all $u\in\mathcal{D}'\left(\mathbb{R}_{+},G\right)$
such that $\left(\mathbb{R}\ni t\mapsto u\left(e^{-t}\right)\right)\in H_{p}^{s}\left(\mathbb{R},G\right)$.
In particular, $\mathcal{H}_{p}^{0,\frac{1}{2}}\left(\mathbb{R}_{+},G\right)=L_{p}\left(\mathbb{R}_{+},G,\frac{dt}{t}\right)$
and $\mathcal{H}_{p}^{1,\frac{1}{2}}\left(\mathbb{R}_{+},G\right)=\left\{ u\in L_{p}\left(\mathbb{R}_{+},G,\frac{dt}{t}\right);\,t\partial_{t}u\in L_{p}\left(\mathbb{R}_{+},G,\frac{dt}{t}\right)\right\} $.
\end{defn}
\begin{prop}
\label{prop:Estimativa Besov do cone}There is a constant $C>0$ such
that 
\[
\left\Vert u\right\Vert _{\mathcal{B}_{p}^{0,\frac{n}{2}}\left(\partial X^{\wedge},F\right)}\le C\left\Vert u\right\Vert _{\mathcal{H}_{p}^{1,\frac{1}{2}}\left(\mathbb{R}_{+},B_{p}^{0}\left(\partial X,F\right)\right)},
\]
for all $u\in\mathcal{T}_{\frac{1}{2}}\left(\mathbb{R}_{+},C^{\infty}\left(\partial X,F\right)\right)$.
\end{prop}
\begin{proof}
In order to prove the proposition, we fix a constant $\theta\in\left]0,1\right[$
and a constant $C_{\theta}>1$ such that $j+1\le C_{\theta}2^{\theta j}$,
for all $j\in\mathbb{N}_{0}$. The Hölder inequality implies that,
for every non-negative real numbers $a_{0}$, ..., $a_{j}$, we have
\[
\left(\sum_{k=0}^{j}a_{k}\right)^{p}\le\left(j+1\right)^{p-1}\sum_{k=0}^{j}a_{k}^{p}\le C_{\theta}^{p}2^{j\theta p}\sum_{k=0}^{j}a_{k}^{p}.
\]

Now, let us first prove that $\left\Vert u\right\Vert _{B_{p}^{0}\left(\mathbb{R}^{n}\right)}\le C\left\Vert u\right\Vert _{H_{p}^{1}\left(\mathbb{R},B_{p}^{0}\left(\mathbb{R}^{n-1}\right)\right)}$:
\begin{eqnarray}
\lefteqn{\left\Vert u\right\Vert _{B_{p}^{0}\left(\mathbb{R}^{n}\right)}
=\Big(\sum_{j=0}^{\infty}\Vert \varphi_{j}\left(D_{t}\right)
\sum_{k=0}^{j}\tilde{\varphi}_{k}\left(D_{x}\right)u+\sum_{k=0}^{j-1}\varphi_{k}\left(D_{t}\right)\tilde{\varphi}_{j}\left(D_{x}\right)u\Vert _{L_{p}\left(\mathbb{R}^{n}\right)}^{p}\Big)^{\frac{1}{p}}}
\nonumber\\
&\le&
\Big(\sum_{j=0}^{\infty}\Big(\sum_{k=0}^{j}\left\Vert \varphi_{j}\left(D_{t}\right)\tilde{\varphi}_{k}\left(D_{x}\right)u\right\Vert _{L_{p}\left(\mathbb{R}^{n}\right)}\Big)^{p}\Big)^{\frac{1}{p}}\nonumber\\
&&+
\Big(\sum_{j=0}^{\infty}\Big(\sum_{k=0}^{j-1}\left\Vert \varphi_{k}\left(D_{t}\right)\tilde{\varphi}_{j}\left(D_{x}\right)u\right\Vert _{L_{p}\left(\mathbb{R}^{n}\right)}\Big)^{p}\Big)^{\frac{1}{p}}
\nonumber\\
&\le&
2C_{\theta}\Big(\sum_{j=0}^{\infty}\sum_{k=0}^{\infty}2^{j\theta p}\left\Vert \varphi_{j}\left(D_{t}\right)\tilde{\varphi}_{k}\left(D_{x}\right)u\right\Vert _{L_{p}\left(\mathbb{R}^{n}\right)}^{p}\Big)^{\frac{1}{p}}
\nonumber\\
&=&2C_{\theta}\Big(\sum_{j=0}^{\infty}2^{j\theta p}\int_{\mathbb{R}}\sum_{k=0}^{\infty}\left\Vert \varphi_{j}\left(D_{t}\right)\tilde{\varphi}_{k}\left(D_{x}\right)u\right\Vert _{L_{p}\left(\mathbb{R}_{x}^{n-1}\right)}^{p}dt\Big)^{\frac{1}{p}}
\nonumber\\
&=&
2C_{\theta}\Big(\sum_{j=0}^{\infty}2^{j\theta p}\left\Vert \varphi_{j}\left(D_{t}\right)u\right\Vert _{L_{p}\left(\mathbb{R},B_{p}^{0}\left(\mathbb{R}^{n-1}\right)\right)}^{p}\Big)^{\frac{1}{p}}=2C_{\theta}\left\Vert u\right\Vert _{B_{p}^{\theta}\left(\mathbb{R},B_{p}^{0}\left(\mathbb{R}^{n-1}\right)\right)}.
\nonumber
\end{eqnarray}
Choosing $\theta<1$, we conclude that
\[
\left\Vert u\right\Vert _{B_{p}^{0}\left(\mathbb{R}^{n}\right)}\le C_{\theta}\left\Vert u\right\Vert _{B_{p}^{\theta}\left(\mathbb{R},B_{p}^{0}\left(\mathbb{R}^{n-1}\right)\right)}\le C_{\theta}\left\Vert u\right\Vert _{H_{p}^{1}\left(\mathbb{R},B_{p}^{0}\left(\mathbb{R}^{n-1}\right)\right)}.
\]
Using a change of variable $t\mapsto e^{-t}$, we obtain that
\[
\left\Vert u\right\Vert _{\mathcal{B}_{p}^{0,\frac{n}{2}}\left(\mathbb{R}_{+}^{n}\right)}\le C_{\theta}\left\Vert u\right\Vert _{\mathcal{H}_{p}^{1,\frac{1}{2}}\left(\mathbb{R}_{+},B_{p}^{0}\left(\mathbb{R}^{n-1}\right)\right)},
\]
where $\mathcal{B}_{p}^{0,\gamma}(\mathbb{R}_{+}^{n})=\left\{v\left(\ln\left(t\right),x\right);\,v\in B_{p}^{0}\left(\mathbb{R}^{n}\right)\right\} $. Finally, using partition of unity and localization, we obtain the assertion.
\end{proof}
For the following proposition we write $\mathcal{HB}_{pE_{j},F_{j}}\left(X^{\wedge}\right):=\mathcal{H}_{p}^{0,\frac{n+1}{2}}\left(X^{\wedge},E_{j}\right)\oplus\mathcal{B}_{p}^{0,\frac{n}{2}}\left(\partial X^{\wedge},F_{j}\right)$
and $\mathcal{HB}_{pE_{j},F_{j}}\left(\mathbb{D}\right):=\mathcal{H}_{p}^{0,\frac{n+1}{2}}\left(\mathbb{D},E_{j}\right)\oplus\mathcal{B}_{p}^{0,\frac{n}{2}}\left(\mathbb{B},F_{j}\right)$, for $j=0,1$. We denote by $K_{j}$, $j\in\mathbb{N}_{0}$, the sets introduced in Remark \ref{rem:Definicao conjuntos Kj} for $n=1$.

\begin{prop}
\label{prop:LpBpMellin}There exists a constant $C$, independent of $m$, such that for all $u\in\mathcal{T}_{\frac{1}{2}}\left(\mathbb{R}_{+}\right)$
and all $v\in C^{\infty}\left(X,E,F\right)$ with
$supp(\tau\mapsto(\mathcal{M}_{\frac{1}{2}}u)(i\tau))\subset K_{m}$
\begin{eqnarray*}
\lefteqn{\frac{1}{C}\frac{1}{\left(m+1\right)}\left\Vert u\right\Vert _{L_{p}\left(\mathbb{R}_{+},\frac{dt}{t}\right)}\left\Vert v\right\Vert _{L_{p}\left(X,E\right)\oplus B_{p}^{0}\left(\partial X,F\right)}}\\
&\le&\left\Vert u\otimes v\right\Vert _{\mathcal{HB}_{pE,F}\left(X^{\wedge}\right)}\le C\left(m+1\right)\left\Vert u\right\Vert _{L_{p}\left(\mathbb{R}_{+},\frac{dt}{t}\right)}\left\Vert v\right\Vert _{L_{p}\left(X,E\right)\oplus B_{p}^{0}\left(\partial X,F\right)}.\end{eqnarray*}
\end{prop}
In order to make the proof more transparent, we first
prove the following lemma.

\begin{lem}\label{lem:lpb0b0}
There exists a constant $C>0$, independent of $m$, such that
for $u\in\mathcal{S}\left(\mathbb{R}\right)$ and $v\in\mathcal{S}(\mathbb{R}^{n-1})$
with $\text{supp}\left(\mathcal{F}u\right)\subset K_{m}$. 
\[
\frac{1}{C}\frac{1}{\left(m+1\right)}\left\Vert u\right\Vert _{L_{p}\left(\mathbb{R}\right)}\left\Vert v\right\Vert _{B_{p}^{0}\left(\mathbb{R}^{n-1}\right)}\le\left\Vert u\otimes v\right\Vert _{B_{p}^{0}\left(\mathbb{R}^{n}\right)}\le C\left(m+1\right)\left\Vert u\right\Vert _{L_{p}\left(\mathbb{R}\right)}\left\Vert v\right\Vert _{B_{p}^{0}\left(\mathbb{R}^{n-1}\right)}.
\]
\end{lem}
\begin{proof}
Let $\left(t,x\right)\in\mathbb{R}\times\mathbb{R}^{n-1}$ and $C>0$  such that $\left\Vert \varphi_{k}\left(D_{t}\right)u\right\Vert _{L_{p}\left(\mathbb{R}\right)}\le C\left\Vert u\right\Vert _{L_{p}\left(\mathbb{R}\right)}$
and $\left\Vert \tilde{\varphi}_{k}\left(D_{x}\right)u\right\Vert _{L_{p}\left(\mathbb{R}^{n-1}\right)}\le C\left\Vert u\right\Vert _{L_{p}\left(\mathbb{R}^{n-1}\right)}$
for all $k\in\mathbb{N}_{0}$ and for all Schwartz functions $u$.
This constant exists, as we have seen in the proof of Lemma \ref{lem:LpBp}.
In particular, $\left\Vert \varphi_{k}\left(D_{t}\right)\tilde{\varphi}_{j}\left(D_{x}\right)\left(u\right)\right\Vert _{L_{p}\left(\mathbb{R}\times\mathbb{R}^{n-1}\right)}\le C^{2}\left\Vert u\right\Vert _{L_{p}\left(\mathbb{R}\times\mathbb{R}^{n-1}\right)}$,
for all $k,j\in\mathbb{N}_{0}$ and $u\in\mathcal{S}(\mathbb{R}\times\mathbb{R}^{n-1})$.

Using the conventions $\varphi_{k}=0$, $\tilde{\varphi}_{k}=0$,
$\psi_{k}=0$ and $K_{k}=\emptyset$, whenever $k\le-1$, we see that
\begin{eqnarray}
\lefteqn{\left\Vert u\otimes v\right\Vert _{B_{p}^{0}\left(\mathbb{R}^{n}\right)}=
\Big(\sum_{j=0}^{\infty}\left\Vert \psi_{j}\left(D\right)\left(u\otimes v\right)\right\Vert _{L_{p}\left(\mathbb{R}^{n}\right)}^{p}\Big)^{\frac{1}{p}}}
\nonumber\\
&=&
\Big(\sum_{j=0}^{\infty}\Vert \varphi_{j}(D_{t})u\sum_{k=0}^{j}\tilde{\varphi}_{k}\left(D_{x}\right)v+\sum_{k=0}^{j-1}\varphi_{k}(D_{t})u\tilde{\varphi}_{j}\left(D_{x}\right)v\Vert _{L_{p}\left(\mathbb{R}^{n}\right)}^{p}\Big)^{\frac{1}{p}}
\nonumber\\
&\le&
\Big(\sum_{j=m-1}^{m+1}\left\Vert \varphi_{j}(D_{t})u\right\Vert _{L_{p}\left(\mathbb{R}\right)}^{p}\Big\Vert \sum_{k=0}^{j}\tilde{\varphi}_{k}(D_{x})v\Big\Vert _{L_{p}\left(\mathbb{R}^{n-1}\right)}^{p}\Big)^{\frac{1}{p}}
\nonumber\\
&&+\Big(\sum_{j=0}^{\infty}\Big\Vert \sum_{k=m-1}^{m+1}\varphi_{k}(D_{t})u\Big\Vert _{L_{p}\left(\mathbb{R}\right)}^{p}\left\Vert \tilde{\varphi}_{j}(D_{x})v\right\Vert _{L_{p}\left(\mathbb{R}^{n-1}\right)}^{p}\Big)^{\frac{1}{p}}
\nonumber\\
&\le&
\Big(\sum_{j=m-1}^{m+1}\left\Vert \varphi_{j}(D_{t})u\right\Vert _{L_{p}\left(\mathbb{R}\right)}^{p}\left(j+1\right)^{p-1}\sum_{k=0}^{j}\left\Vert \tilde{\varphi}_{k}\left(D_{x}\right)v\right\Vert _{L_{p}\left(\mathbb{R}^{n-1}\right)}^{p}\Big)^{\frac{1}{p}}
\nonumber\\
&&+
\Big(\sum_{j=0}^{\infty}\sum_{k=m-1}^{m+1}3^{p-1}\left\Vert \varphi_{k}\left(D_{t}\right)u\right\Vert _{L_{p}\left(\mathbb{R}\right)}^{p}\left\Vert \tilde{\varphi}_{j}\left(D_{x}\right)v\right\Vert _{L_{p}\left(\mathbb{R}^{n-1}\right)}^{p}\Big)^{\frac{1}{p}}
\nonumber\\
&\le&
\left(m+2\right)^{1-\frac{1}{p}}3^{\frac{1}{p}}C\left\Vert u\right\Vert _{L_{p}\left(\mathbb{R}\right)}\Big(\sum_{k=0}^{\infty}\left\Vert \tilde{\varphi}_{k}\left(D_{x}\right)v\right\Vert _{L_{p}\left(\mathbb{R}^{n-1}\right)}^{p}\Big)^{\frac{1}{p}}
\nonumber\\
&&+3C\left\Vert u\right\Vert _{L_{p}\left(\mathbb{R}\right)}\Big(\sum_{j=0}^{\infty}\left\Vert \tilde{\varphi}_{j}\left(D_{x}\right)v\right\Vert _{L_{p}\left(\mathbb{R}^{n-1}\right)}^{p}\Big)^{\frac{1}{p}}
\nonumber\\
&\le&
\tilde{C}\left(m+1\right)\left\Vert u\right\Vert _{L_{p}\left(\mathbb{R}\right)}\left\Vert v\right\Vert _{B_{p}^{0}\left(\mathbb{R}^{n-1}\right)}.
\nonumber
\end{eqnarray}
On the other hand, with $\Vert\cdot\Vert$ denoting the norm in $L_p(\mathbb R^n)$,
\begin{eqnarray}
\lefteqn{\left\Vert u\right\Vert_{L_{p}\left(\mathbb{R}\right)}^{p}
\left\Vert v\right\Vert _{B_{p}^{0}\left(\mathbb{R}^{n-1}\right)}^{p}=\sum_{k=0}^{\infty}\left\Vert u\otimes\tilde{\varphi}_{k}\left(D_{x}\right)v\right\Vert _{L_{p}\left(\mathbb{R}^{n}\right)}^{p}}
\nonumber\\
&=&
\sum_{j=0}^{\infty}\left\Vert \sum_{k=m-1}^{m+1}\varphi_{k}(D_{t})u\,\tilde{\varphi}_{j}(D_{x})v\right\Vert^p
\nonumber\\
&\le&
\sum_{j=0}^{m+1}\left\Vert \sum_{k=m-1}^{m+1}\varphi_{k}(D_{t})u\,\tilde{\varphi}_{j}(D_{x})v\right\Vert^p
+\sum_{j=m+2}^{\infty}\left\Vert \sum_{k=m-1}^{m+1}\varphi_{k}(D_{t})u\,\tilde{\varphi}_{j}(D_{x})v\right\Vert^p
\nonumber\\
&\overset{(1)}{\le}&
\sum_{j=0}^{m+1}\sum_{k=m-1}^{m+1}3^{p-1}\left\Vert \varphi_{k}(D_{t})u\, \tilde{\varphi}_{j}(D_{x})v\right\Vert^{p}
+\left\Vert u\otimes v\right\Vert _{B_{p}^{0}\left(\mathbb{R}^{n}\right)}^{p}
\nonumber\\
&\overset{(2)}{=}&
\sum_{j=0}^{m+1}\sum_{k=m-1}^{m+1}\!\!3^{p-1}\!\left\Vert 
\sum_{l=m-2}^{m+2}\left(\varphi_{k}(D_{t})\otimes\tilde{\varphi}_{j}(D_{x})\right)\psi_{l}(D)\left(u\otimes v\right)\right\Vert^p
+\left\Vert u\otimes v\right\Vert _{B_{p}^{0}\left(\mathbb{R}^{n}\right)}^{p}
\nonumber\\
&\le&
\sum_{j=0}^{m+1}\sum_{k=m-1}^{m+1}\!\!\!15^{p-1}\!\!\!\sum_{l=m-2}^{m+2}\left\Vert \left(\varphi_{k}\left(D_{t}\right)\otimes\tilde{\varphi}_{j}\left(D_{x}\right)\right)\psi_{l}\left(D\right)\left(u\otimes v\right)\right\Vert^p
+\left\Vert u\otimes v\right\Vert _{B_{p}^{0}\left(\mathbb{R}^{n}\right)}^{p}
\nonumber
\end{eqnarray}
\begin{eqnarray}
&\le&
3^{p}5^{p-1}C^{2}\left(m+2\right)\sum_{l=0}^{\infty}\left\Vert \psi_{l}\left(D\right)\left(u\otimes v\right)\right\Vert _{L_{p}\left(\mathbb{R}^{n}\right)}^{p}+\left\Vert u\otimes v\right\Vert _{B_{p}^{0}\left(\mathbb{R}^{n}\right)}^{p}
\nonumber\\
&\le&
\tilde{C}\left(m+1\right)\left\Vert u\otimes v\right\Vert _{B_{p}^{0}\left(\mathbb{R}^{n}\right)}^{p}.\nonumber
\end{eqnarray}

We have used in $(1)$ that $\text{supp}\left(\mathcal{F}u\right)\subset K_{m}$
and, therefore, we have
\[
\sum_{j=m+2}^{\infty}\Big\Vert \Big(\sum_{k=m-1}^{m+1}\varphi_{k}\left(D_{t}\right)u\Big)\tilde{\varphi}_{j}\left(D_{x}\right)v\Big\Vert^p%
\le\sum_{j=m+2}^{\infty}\Big\Vert \sum_{k=m-1}^{m+1}\varphi_{k}\left(D_{t}\right)u\tilde{\varphi}_{j}\left(D_{x}\right)v\Big\Vert^p
\]
\[
+\sum_{j=0}^{m+1}\Big\Vert \varphi_{j}\left(D_{t}\right)u\Big(\sum_{k=0}^{j}\tilde{\varphi}_{k}\left(D_{x}\right)v\Big)
\!\!+\!\!
\Big(\sum_{k=0}^{j-1}\varphi_{k}\left(D_{t}\right)u\Big)\tilde{\varphi}_{j}\left(D_{x}\right)v\Big\Vert^p 
=
\left\Vert u\otimes v\right\Vert _{B_{p}^{0}\left(\mathbb{R}^{n}\right)}^{p}.
\]
We have used in $(2)$ that for $j\in\left\{ 0,...,m+1\right\} $
and $k\in\left\{ m-1,...,m+1\right\} $
\[
\sum_{l=m-2}^{m+2}\psi_{l}\left(D\right)\left(\varphi_{k}\left(D_{t}\right)\otimes\tilde{\varphi}_{j}\left(D_{x}\right)\right)=\varphi_{k}\left(D_{t}\right)\otimes\tilde{\varphi}_{j}\left(D_{x}\right).
\]
\end{proof}
\begin{proof}
(of Proposition \ref{prop:LpBpMellin}) Let $u\in\mathcal{T}_{\frac{1}{2}}\left(\mathbb{R}_{+}\right)$,
$v\in\mathcal{S}\left(\mathbb{R}^{n-1}\right)$ and suppose that $\text{supp}\big(\tau\mapsto\mathcal{M}_{\frac{1}{2}}u\left(i\tau\right)\big)\subset K_{m}$, 
$m\in \mathbb N_0$.
Define $\tilde{u}\in\mathcal{S}\left(\mathbb{R}\right)$ by
$\tilde{u}\left(t\right)=u\left(e^{-t}\right)$. Hence $\mathcal{F}_{t\to\xi}\tilde{u}\left(\xi\right)=\mathcal{F}_{t\to\xi}\left(u\left(e^{-t}\right)\right)\left(\xi\right)=\mathcal{M}_{\frac{1}{2}}u\left(i\xi\right)$.
Therefore, there is a constant $m\in\mathbb{N}_{0}$ such that $\text{supp}\left(\mathcal{F}\tilde{u}\right)\subset K_{m}$. Hence, Lemma \ref{lem:lpb0b0} implies that
\begin{eqnarray*}\lefteqn{
\frac{1}{C(m+1)}\left\Vert u\right\Vert _{L_{p}\left(\mathbb{R}_{+},\frac{dt}{t}\right)}
\left\Vert v\right\Vert _{B_{p}^{0}\left(\mathbb{R}^{n-1}\right)}
=\frac{1}{C(m+1)}\left\Vert \tilde{u}\right\Vert _{L_{p}\left(\mathbb{R}\right)}\left\Vert v\right\Vert _{B_{p}^{0}\left(\mathbb{R}^{n-1}\right)}}\\
&\le&\left\Vert \tilde{u}\otimes v\right\Vert _{B_{p}^{0}\left(\mathbb{R}^{n}\right)}
\le C\left(m+1\right)\left\Vert \tilde{u}\right\Vert _{L_{p}\left(\mathbb{R}\right)}\left\Vert v\right\Vert _{B_{p}^{0}\left(\mathbb{R}^{n-1}\right)}\\
&\le& C\left(m+1\right)\left\Vert u\right\Vert _{L_{p}\left(\mathbb{R}_{+},\frac{dt}{t}\right)}\left\Vert v\right\Vert _{\mathcal{B}_{p}^{0}\left(\mathbb{R}^{n-1}\right)}.
\end{eqnarray*}
As $\left\Vert \tilde{u}\otimes v\right\Vert _{B_{p}^{0}\left(\mathbb{R}^{n}\right)}=\left\Vert u\otimes v\right\Vert _{\mathcal{B}_{p}^{0,\frac{n}{2}}\left(\mathbb{R}^{n}_+\right)}$,
we conclude that 
\begin{eqnarray*}
\lefteqn{\frac{1}{C(m+1)}\left\Vert u\right\Vert _{L_{p}\left(\mathbb{R}_{+},\frac{dt}{t}\right)}\left\Vert v\right\Vert _{B_{p}^{0}\left(\mathbb{R}^{n-1}\right)}}\\
&\le&\left\Vert u\otimes v\right\Vert _{\mathcal{B}_{p}^{0,\frac{n}{2}}(\mathbb{R}^{n}_+)}
\le C\left(m+1\right)\left\Vert u\right\Vert _{L_{p}\left(\mathbb{R}_{+},\frac{dt}{t}\right)}\left\Vert v\right\Vert _{{B}_{p}^{0}\left(\mathbb{R}^{n-1}\right)}.
\end{eqnarray*}
The general result follows using a partition of unity.
\end{proof}

\subsection{Proof of the invertibility of the conormal symbol}

We notice that $\mathcal{T}_{\frac{1}{2}}\left(\mathbb{R}_{+},C^{\infty}\left(X,E_{j},F_{j}\right)\right)$
is a dense space of $\mathcal{HB}_{pE_{j},F_{j}}\left(X^{\wedge}\right)$.
\begin{defn}
Let $W$ be a Fréchet space, $\epsilon>0$, $\tau_{0}\in\mathbb{R}$.
We define $T_{\epsilon}:\mathcal{T}_{\frac{1}{2}}\left(\mathbb{R}_{+},W\right)\to\mathcal{T}_{\frac{1}{2}}\left(\mathbb{R}_{+},W\right)$
and $R_{\epsilon,\tau_{0}}:\mathcal{T}_{\frac{1}{2}}\left(\mathbb{R}_{+},W\right)\to\mathcal{T}_{\frac{1}{2}}\left(\mathbb{R}_{+},W\right)$
by $T_{\epsilon}u(t)=u\left(\frac{t}{\epsilon}\right)$ and $R_{\epsilon,\tau_{0}}u(t)=\epsilon^{\frac{1}{p}}t^{-i\tau_{0}}u\left(t^{\epsilon}\right)$.
\end{defn}
The above operators are invertible: $T_{\epsilon}^{-1}=T_{\frac{1}{\epsilon}}$
and $R_{\epsilon,\tau_{0}}^{-1}=R_{\frac{1}{\epsilon},-\frac{\tau_{0}}{\epsilon}}$.
The next proposition is analogous to Lemma \ref{lem:Propriedades R_s}.
\begin{prop}
\label{prop:Tepsilonconv} For an UMD Banach space W with the property ($\alpha$), the operators $T_{\epsilon},R_{\epsilon,\tau_{0}}:\mathcal{T}_{\frac{1}{2}}\left(\mathbb{R}_{+},W\right)\to\mathcal{T}_{\frac{1}{2}}\left(\mathbb{R}_{+},W\right)$
have the following properties:
\begin{enumerate}\renewcommand{\labelenumi}{\arabic{enumi}{\rm)}}
\item  $T_{\epsilon}$ extends to an isometry
\[
T_{\epsilon}:\mathcal{H}_{p}^{1,\frac{1}{2}}\left(\mathbb{R}_{+},W\right)\to\mathcal{H}_{p}^{1,\frac{1}{2}}\left(\mathbb{R}_{+},W\right).
\]
If $W=C^{\infty}\left(X,E_{j},F_{j}\right)$, $j=0,1$, then the operator
$T_{\epsilon}$ extends to an isometry $T_{\epsilon}:\mathcal{HB}_{pE_{j},F_{j}}\left(X^{\wedge}\right)\to\mathcal{HB}_{pE_{j},F_{j}}\left(X^{\wedge}\right)$. 

\item  For all $\varepsilon>0$, $R_{\epsilon,\tau_{0}}$ extends to a bijective continuous map
\[
R_{\epsilon,\tau_{0}}:\mathcal{H}_{p}^{1,\frac{1}{2}}\left(\mathbb{R}_{+},W\right)\to\mathcal{H}_{p}^{1,\frac{1}{2}}\left(\mathbb{R}_{+},W\right).
\]
There exists a $C\ge0$ with $\left\Vert R_{\epsilon,\tau_{0}}\right\Vert _{\mathcal{B}(\mathcal{H}_{p}^{1,\frac{1}{2}}\left(\mathbb{R}_{+},W\right))}\le C\left(1+\left|\tau_{0}\right|\right)$, $\epsilon<1$.

\item 
{\rm(i)} Let $h\in M\tilde{B}_{E_{0},F_{0},E_{1},F_{1}}^{p}\left(X,\mathbb{R}_{+};\Gamma_{0}\right)\cap C\left(\overline{\mathbb{R}_{+}},\mathcal{\tilde{B}}_{E_{0},F_{0},E_{1},F_{1}}^{p}\left(X,\Gamma_{0}\right)\right)$
and $h_{0}(z):=h(0,z)$. For any $u\in\mathcal{T}_{\frac{1}{2}}\left(\mathbb{R}_{+},C^{\infty}\left(X,E_{0},F_{0}\right)\right)$
we then have 
\[
\lim_{\epsilon\to0}\left\Vert op_{M}^{\frac{1}{2}}\left(h\right)T_{\epsilon}u-T_{\epsilon}op_{M}^{\frac{1}{2}}\left(h_{0}\right)u\right\Vert _{\mathcal{HB}_{pE_{1},F_{1}}\left(X^{\wedge}\right)}=0.
\]
{\rm(ii)} Let $h\in\tilde{B}_{E_{0},F_{0},E_{1},F_{1}}^{p}\left(X,\Gamma_{0}\right)$
and $u\in\mathcal{T}_{\frac{1}{2}}\left(\mathbb{R}_{+},C^{\infty}\left(X,E_{0},F_{0}\right)\right)$. Then
\[
\lim_{\epsilon\to0}\left\Vert op_{M}^{\frac{1}{2}}\left(h\right)R_{\epsilon,\tau_{0}}u-R_{\epsilon,\tau_{0}}h\left(i\tau_{0}\right)u\right\Vert _{\mathcal{HB}_{pE_{1},F_{1}}\left(X^{\wedge}\right)}=0.
\]
\end{enumerate}
\end{prop}

\begin{proof}
1) Since $\left\Vert T_{\epsilon}u\right\Vert _{L_{p}\left(\mathbb{R}_{+},W,\frac{dt}{t}\right)}=\left\Vert u\right\Vert _{L_{p}\left(\mathbb{R}_{+},W,\frac{dt}{t}\right)}$
and $\left\Vert \left(t\partial_{t}\right)\left(T_{\epsilon}u\right)\right\Vert _{L_{p}\left(\mathbb{R}_{+},W,\frac{dt}{t}\right)}=\left\Vert t\partial_{t}u\right\Vert _{L_{p}\left(\mathbb{R}_{+},W,\frac{dt}{t}\right)}$, we conclude that  $\left\Vert T_{\epsilon}u\right\Vert _{\mathcal{H}_{p}^{1,\frac{1}{2}}\left(\mathbb{R}_{+},W\right)}=\left\Vert u\right\Vert _{\mathcal{H}_{p}^{1,\frac{1}{2}}\left(\mathbb{R}_{+},W\right)}$.

In order to show that $T_{\epsilon}:\mathcal{HB}_{pE_{j},F_{j}}\left(X^{\wedge}\right)\to\mathcal{HB}_{pE_{j},F_{j}}\left(X^{\wedge}\right)$
is an isometry, it remains to prove that $T_{\epsilon}:\mathcal{B}_{p}^{0,\frac{n}{2}}\left(\partial X^{\wedge},F_{j}\right)\to\mathcal{B}_{p}^{0,\frac{n}{2}}\left(\partial X^{\wedge},F_{j}\right)$
is an isometry. This follows with a partition of unity and the fact
that $T_{\epsilon}:\mathcal{B}_{p}^{0,\frac{n}{2}}\left(\mathbb{R}_{+}^{n}\right)\to\mathcal{B}_{p}^{0,\frac{n}{2}}\left(\mathbb{R}_{+}^{n}\right)$
given by $T_{\epsilon}u\left(t,x\right)=u\left(\frac{t}{\epsilon},x\right)$
is an isometry. In fact, if $v\left(s,x\right)=u\left(e^{-s},x\right)$,
then $\left(T_{\epsilon}u\right)\left(e^{-s},x\right)=v\left(s+\ln\left(\epsilon\right),x\right)$.
Hence
\[
\left\Vert T_{\epsilon}u\right\Vert _{\mathcal{B}_{p}^{0,\frac{n}{2}}\left(\mathbb{R}_{+}^{n}\right)}=\left\Vert \left(s,x\right)\mapsto v\left(s+\ln\left(\epsilon\right),x\right)\right\Vert _{B_{p}^{0}\left(\mathbb{R}^{n}\right)}=\Vert  
v\Vert _{B_{p}^{0}\left(\mathbb{R}^{n}\right)}=\left\Vert u\right\Vert _{\mathcal{B}_{p}^{0,\frac{n}{2}}\left(\mathbb{R}_{+}^{n}\right)}.
\]

2) It is easy to see that $\left\Vert R_{\epsilon,\tau_{0}}u\right\Vert _{L_{p}\left(\mathbb{R}_{+},W,\frac{dt}{t}\right)}=\left\Vert u\right\Vert _{L_{p}\left(\mathbb{R}_{+},W,\frac{dt}{t}\right)}.$
As $t\partial_{t}\left(R_{\epsilon,\tau_{0}}u\right)=\left(-i\tau_{0}\right)R_{\epsilon,\tau_{0}}u+\epsilon R_{\epsilon,\tau_{0}}\left(t\partial_{t}u\right)$,
we conclude that
\[
\left\Vert R_{\epsilon,\tau_{0}}u\right\Vert _{\mathcal{H}_{p}^{1,\frac{1}{2}}\left(\mathbb{R}_{+},L_{p}\left(X,E_{j}\right)\oplus B_{p}^{0}\left(X,F_{j}\right)\right)}\le\left(1+\left|\tau_{0}\right|\right)\left\Vert u\right\Vert _{\mathcal{H}_{p}^{1,\frac{1}{2}}\left(\mathbb{R}_{+},L_{p}\left(X,E_{j}\right)\oplus B_{p}^{0}\left(X,F_{j}\right)\right)}.
\]

3.i) We first show {\bf $\mathbf L_p$-convergence}: For 
$u\in\mathcal{T}_{\frac{1}{2}}\left(\mathbb{R}_{+},C^{\infty}\left(X,E_{0},F_{0}\right)\right),$
\begin{equation}
\lim_{\epsilon\to0}\left\Vert T_{\epsilon}^{-1}op_{M}^{\frac{1}{2}}\left(h\right)T_{\epsilon}u-op_{M}^{\frac{1}{2}}\left(h_{0}\right)u\right\Vert _{L_{p}\left(\mathbb{R}_{+},L_{p}\left(X,E_{1}\right)\oplus B_{p}^{0}\left(\partial X,F_{1}\right);\frac{dt}{t}\right)}=0.
\label{eq:convH1BdM-1-2}
\end{equation}

The proof here is exactly the same as the proof of \cite[Lemma 3.9]{SchroheSeilerSpectConical}. It relies on the fact that
$
T_{\epsilon}^{-1}op_{M}^{\frac{1}{2}}\left(h\right)T_{\epsilon}=op_{M}^{\frac{1}{2}}\left(h_{\epsilon}\right),
$
where $h_{\epsilon}\left(t,z\right)=h\left(\epsilon t,z\right)$, and on 
Lebesgue's dominated convergence theorem.

Next we establish the $\mathbf{L_p}${\bf -convergence of the derivative}: 
\begin{equation}
\lim_{\epsilon\to0}\left\Vert T_{\epsilon}^{-1}op_{M}^{\frac{1}{2}}\left(h\right)T_{\epsilon}u-op_{M}^{\frac{1}{2}}\left(h_{0}\right)u\right\Vert _{\mathcal{H}_{p}^{1,\frac{1}{2}}\left(\mathbb{R}_{+},L_{p}\left(X,E_{1}\right)\oplus B_{p}^{0}\left(\partial X,F_{1}\right)\right)}=0.\label{eq:convH1BdM-1-1-1}
\end{equation}
This follows almost immediately from the fact that 
\[
\left(-t\partial_{t}\right)op_{M}^{\frac{1}{2}}\left(h_{\epsilon}\right)u=op_{M}^{\frac{1}{2}}\left(\left(\left(-t\partial_{t}\right)h\right)_{\epsilon}\right)u+op_{M}^{\frac{1}{2}}\left(h_{\epsilon}\right)\left(\left(-t\partial_{t}\right)u\right).
\]
Using \eqref{eq:convH1BdM-1-1-1}, the fact that $T_{\epsilon}$ are isometries and Proposition \ref{prop:Estimativa Besov do cone}, we conclude that,
as $\epsilon \to 0$,
\begin{eqnarray*}
\lefteqn{\left\Vert op_{M}^{\frac{1}{2}}\left(h\right)T_{\epsilon}u-T_{\epsilon}op_{M}^{\frac{1}{2}}\left(h_{0}\right)u\right\Vert _{\mathcal{H}_{p}^{0,\frac{n+1}{2}}\left(X^{\wedge},E_{1}\right)\oplus\mathcal{B}_{p}^{0,\frac{n}{2}}\left(\partial X^{\wedge},F_{1}\right)}
}\\
&\le& C\left\Vert T_{\epsilon}^{-1}op_{M}^{\frac{1}{2}}\left(h\right)T_{\epsilon}u-op_{M}^{\frac{1}{2}}\left(h_{0}\right)u\right\Vert _{\mathcal{H}_{p}^{1,\frac{1}{2}}\left(\mathbb{R}_{+},L_{p}\left(X,E_{1}\right)\oplus B_{p}^{0}\left(\partial X,E_{1}\right)\right)}\to0.
\end{eqnarray*}

3.ii) It is straightforward to check that
$R_{\epsilon,\tau_{0}}^{-1}op_{M}^{\frac{1}{2}}\left(h\right)R_{\epsilon,\tau_{0}}=op_{M}^{\frac{1}{2}}\left(h_{\epsilon}\right)$,
where $h_{\epsilon}\left(z\right)=h\left(\epsilon z+i\tau_{0}\right)$.
Repeating the previous arguments, we conclude that
\[
\lim_{\epsilon\to0}\left\Vert R_{\epsilon,\tau_{0}}^{-1}op_{M}^{\frac{1}{2}}\left(h\right)R_{\epsilon,\tau_{0}}u-h\left(i\tau_{0}\right)u\right\Vert _{L_{p}\left(\mathbb{R}_{+},L_{p}\left(X,E_{1}\right)\oplus B_{p}^{0}\left(\partial X,F_{1}\right);\frac{dt}{t}\right)}=0.
\]
Moreover,  $\left(-t\partial_{t}\right)op_{M}^{\frac{1}{2}}\left(h\left(\epsilon z+i\tau_{0}\right)\right)u=op_{M}^{\frac{1}{2}}\left(h\left(\epsilon z+i\tau_{0}\right)\right)\left(-t\partial_{t}u\right)$.
Hence
\[
\lim_{\epsilon\to0}\left\Vert R_{\epsilon,\tau_{0}}^{-1}op_{M}^{\frac{1}{2}}\left(h\right)R_{\epsilon,\tau_{0}}u-h\left(i\tau_{0}\right)u\right\Vert _{\mathcal{H}^{1,\frac{1}{2}}\left(\mathbb{R}_{+},L_{p}\left(X,E_{1}\right)\oplus B_{p}^{0}\left(\partial X,F_{1}\right)\right)}=0.
\]
Finally, using Proposition \ref{prop:Estimativa Besov do cone} and item 2, we conclude that, as $\epsilon\to 0$, 
\begin{eqnarray*}
\lefteqn{\left\Vert op_{M}^{\frac{1}{2}}\left(h\right)R_{\epsilon,\tau_{0}}u
-R_{\epsilon,\tau_{0}}h\left(i\tau_{0}\right)u\right\Vert _{\mathcal{HB}_{pE_{1},F_{1}}\left(X^{\wedge}\right)}}\\
&\le& 
C\left\Vert R_{\epsilon,\tau_{0}}\left(R_{\epsilon,\tau_{0}}^{-1}op_{M}^{\frac{1}{2}}\left(h\right)R_{\epsilon,\tau_{0}}u-h\left(i\tau_{0}\right)u\right)\right\Vert _{\mathcal{H}_{p}^{1,\frac{1}{2}}\left(\mathbb{R}_{+},L_{p}\left(X,E_{1}\right)\oplus B_{p}^{0}\left(\partial X,E\right)\right)}\\
&\le&\tilde{C}\left(1+\left|\tau_{0}\right|\right)\left\Vert R_{\epsilon,\tau_{0}}^{-1}op_{M}^{\frac{1}{2}}\left(h\right)R_{\epsilon,\tau_{0}}u-h\left(i\tau_{0}\right)u\right\Vert _{\mathcal{H}_{p}^{1,\frac{1}{2}}(\mathbb{R}_{+},L_{p}\left(X,E_{1})\oplus B_{p}^{0}\left(\partial X,E\right)\right)}\to0.
\end{eqnarray*}
\end{proof}
The next lemma is analogous to Lemma \ref{lem:Rs em supp de F na bola}.
\begin{lem}
\label{lem:propertiesRepslonTepsilon}The operators $T_{\epsilon}$
and $R_{\epsilon,\tau_{0}}$ satisfy the following properties:
\begin{enumerate}\renewcommand{\labelenumi}{\arabic{enumi}{\rm)}}
\item If $u\in\mathcal{T}_{\frac{1}{2}}\left(\mathbb{R}_{+}\right)$
with $\mbox{supp}(\mathcal{M}_{\frac{1}{2}}u)\subset\left\{ \xi\in\Gamma_{0};\,\left|\xi\right|\le\frac{1}{2}\right\} $, $v\in C^{\infty}\left(X,E_{j},F_{j}\right)$ and $\epsilon<1$, then $\mbox{supp}(\mathcal{M}_{\frac{1}{2}}\left(R_{\epsilon,\tau_{0}}u\right))\subset K_{m}$,
where $K_{0}:=\left\{ \xi\in\Gamma_{0};\left|\xi\right|\le2\right\} $,
$K_{j}:=\left\{ \xi\in\Gamma_{0};2^{j-1}\le\left|\xi\right|\le2^{j+1}\right\} $,
$j\in\mathbb{N}_{0}\backslash\left\{ 0\right\} $. The number $m\in\mathbb{N}_{0}$
is equal to $0$ if $\left|\tau_{0}\right|+\frac{1}{2}<2$ and, for $\left|\tau_{0}\right|+\frac{1}{2}>2$,
$m$ is the smallest number such that $2^{m-1}<\left|\tau_{0}\right|-\frac{1}{2}<\left|\tau_{0}\right|+\frac{1}{2}<2^{m+1}$.
Hence $m\le C\left\langle \ln\left(\tau_{0}\right)\right\rangle $.

\item There is a constant $C>0$ such that for all $\epsilon<1$, $v\in C^{\infty}\left(X,E_{j},F_{j}\right)$
and $u\in\mathcal{T}_{\frac{1}{2}}\left(\mathbb{R}_{+}\right)$ with
$\mbox{supp}(\mathcal{M}_{\frac{1}{2}}u)\subset\left\{ \xi\in\mathbb{R};\,\left|\xi\right|\le\frac{1}{2}\right\} $
\begin{eqnarray*}
\lefteqn{\frac{1}{C\left\langle \ln\tau_{0}\right\rangle }\left\Vert u\otimes v\right\Vert _{\mathcal{HB}_{pE_{j},F_{j}}\left(X^{\wedge}\right)}}\\
&\le&\left\Vert R_{\epsilon,\tau_{0}}\left(u\otimes v\right)\right\Vert _{\mathcal{HB}_{pE_{j},F_{j}}\left(X^{\wedge}\right)}\le C\left\langle \ln\tau_{0}\right\rangle \left\Vert u\otimes v\right\Vert _{\mathcal{HB}_{pE_{j},F_{j}}\left(X^{\wedge}\right)}.
\end{eqnarray*}

\item For all $u\in\mathcal{T}_\frac12\left(\mathbb{R}_{+},C^{\infty}\left(X,E_{j},F_{j}\right)\right)$,
we have $\lim_{\epsilon\to0}T_{\epsilon}\left(u\right)=0$ weakly
in $\mathcal{HB}_{pE_{j},F_{j}}\left(X^{\wedge}\right)$.
\end{enumerate}
\end{lem}

\begin{proof}
1) An easy computation shows  that $\mathcal{M}_{\frac{1}{2}}\left(R_{\epsilon,\tau_{0}}u\right)(z)=\epsilon^{\frac{1}{p}-1}\mathcal{M}_{\frac{1}{2}}u\left(\frac{z}{\epsilon}-\frac{i\tau_{0}}{\epsilon}\right)$.
When $\epsilon<1$, this means that, if $x\in\mathbb{R}$
is such that $\mathcal{M}_{\frac{1}{2}}\left(R_{\epsilon,\tau_{0}}u\right)\left(ix\right)\ne0$,
then $\tau_{0}-\frac{1}{2}<x<\tau_{0}+\frac{1}{2}$,
which implies that $\text{supp}(\mathcal{M}_{\frac{1}{2}}\left(R_{\epsilon,\tau_{0}}u\right))$
is contained in some ball of radius $\frac{1}{2}$.

2) As $\mbox{supp}(\mathcal{M}_{\frac{1}{2}}u)\subset K_{0}$, Proposition \ref{prop:LpBpMellin} implies that
\begin{eqnarray}
\lefteqn{\left\Vert u\otimes v\right\Vert _{\mathcal{HB}_{pE_{j},F_{j}}\left(X^{\wedge}\right)}\le C_{1}\left\Vert u\otimes v\right\Vert _{L_{p}\left(\mathbb{R}_{+},L_{p}\left(X,E_{j}\right)\oplus B_{p}^{0}\left(\partial X,F_{j}\right);\frac{dt}{t}\right)}}
\nonumber\\
&=&
C_{1}\left\Vert \left(R_{\epsilon,\tau_{0}}u\right)\otimes v\right\Vert _{L_{p}\left(\mathbb{R}_{+},L_{p}\left(X,E_{j}\right)\oplus B_{p}^{0}\left(\partial X,F_{j}\right);\frac{dt}{t}\right)}\nonumber\\
&\le& C_{2}\left\langle \ln\tau_{0}\right\rangle \left\Vert \left(R_{\epsilon,\tau_{0}}u\right)\otimes v\right\Vert _{\mathcal{HB}_{pE_{j},F_{j}}\left(X^{\wedge}\right)}
\nonumber
\end{eqnarray}
and
\begin{eqnarray}
\lefteqn{\left\Vert \left(R_{\epsilon,\tau_{0}}u\right)\otimes v\right\Vert _{\mathcal{HB}_{pE_{j},F_{j}}\left(X^{\wedge}\right)}\le C_{3}\left\langle \ln\tau_{0}\right\rangle \left\Vert \left(R_{\epsilon,\tau_{0}}u\right)\otimes v\right\Vert _{L_{p}\left(\mathbb{R}_{+},L_{p}\left(X,E_{j}\right)\oplus B_{p}^{0}\left(\partial X,F_{j}\right);\frac{dt}{t}\right)}}
\nonumber\\
&\!\!\!\!=&\!\!\!\!
C_{3}\left\langle \ln\tau_{0}\right\rangle \left\Vert u\otimes v\right\Vert _{L_{p}\left(\mathbb{R}_{+},L_{p}\left(X,E_{j}\right)\oplus B_{p}^{0}\left(\partial X,F_{j}\right);\frac{dt}{t}\right)}\le C_{4}\left\langle \ln\tau_{0}\right\rangle \left\Vert u\otimes v\right\Vert _{\mathcal{HB}_{pE_{j},F_{j}}\left(X^{\wedge}\right)}.
\nonumber
\end{eqnarray}

3) We  identify the dual of $\mathcal{HB}_{pE_{j},F_{j}}\left(X^{\wedge}\right)$
with $\mathcal{HB}_{qE_{j},F_{j}}\left(X^{\wedge}\right)$,
where $\frac{1}{p}+\frac{1}{q}=1$, using the scalar product $L_{2}\left(\mathbb{R}_{+},L_{2}\left(X,E_{j}\right)\oplus L_{2}\left(\partial X,F_{j}\right),\frac{dt}{t}\right)$.
As $T_{\epsilon}$ is an isometry in $\mathcal{HB}_{pE_{j},F_{j}}\left(X^{\wedge}\right)$,
it is enough to prove that
\[
\lim_{\epsilon\to0}\int_{\mathbb{R}_{+}}\left\langle u(t/\epsilon),v(t)\right\rangle _{L_{2}\left(X,E_{j}\right)\oplus L_{2}\left(\partial X,F_{j}\right)}\frac{dt}{t}=0
\]
for all $u,v\in C_{c}^{\infty}\left(\mathbb{R}_{+},C^{\infty}\left(X,E_{j},F_{j}\right)\right)$.  But this is true. In fact, let $a,b,R>0$ be such that $\text{supp}\left(u\right)\subset\left[0,R\right]$
and $\text{supp}\left(v\right)\subset\left[a,b\right]$, then, for
$\epsilon<\frac{a}{R}$, we have $\text{supp}\left(T_{\epsilon}u\right)\cap\text{supp}\left(v\right)=\emptyset$.
Hence we obtain the result.
\end{proof}
\begin{lem}
\label{lem:Lemafundamentalconce}Let $h\in\tilde{\mathcal{B}}_{E_{0},F_{0},E_{1},F_{1}}^{p}\left(X,\Gamma_{0}\right)$
and suppose that there is a constant $c>0$ such that, for each $u\in\mathcal{T}_{\frac{1}{2}}\left(\mathbb{R}_{+},C^{\infty}\left(X,E_{0},F_{0}\right)\right)$,
we have
\[
\left\Vert u\right\Vert _{\mathcal{HB}_{pE_{0},F_{0}}\left(X^{\wedge}\right)}\le c\left\Vert \mbox{op}_{M}^{\frac{1}{2}}\left(h\right)\left(u\right)\right\Vert _{\mathcal{HB}_{pE_{1},F_{1}}\left(X^{\wedge}\right)}.
\]
Then, for every $v\in C^{\infty}\left(X,E_{0},F_{0}\right)$ and $\tau\in\mathbb{R}$,
we have 
\[
\left\Vert v\right\Vert _{H_{p}^{0}\left(X,E_{0}\right)\oplus B_{p}^{0}\left(\partial X,F_{0}\right)}\le C\left\langle \ln\left(\tau\right)\right\rangle ^{2}\left\Vert h\left(i\tau\right)v\right\Vert _{H_{p}^{0}\left(X,E_{1}\right)\oplus B_{p}^{0}\left(\partial X,F_{1}\right)},
\]
for some constant $C$ independent of $v$.
\end{lem}
\begin{proof}
Let $0\ne u\in\mathcal{T}_{\frac{1}{2}}\left(\mathbb{R}_{+}\right)$, 
be a function with $\text{supp}(\mathcal{M}_{\frac{1}{2}}\left(u\right))\subset\left\{ z\in\Gamma_{0};\left|z\right|<\frac{1}{2}\right\} $
and $v\in C^{\infty}\left(X\right)$. Then item 2 of Lemma \ref{lem:propertiesRepslonTepsilon}
implies that
\begin{eqnarray*}\lefteqn{\left\Vert u\otimes v\right\Vert _{\mathcal{HB}_{pE_{0,}F_{0}}\left(X^{\wedge}\right)}}\\
&\le& C_{1}\left\langle \ln\tau_{0}\right\rangle \big\Vert op_{M}^{\frac{1}{2}}\left(h\right)R_{\epsilon,\tau_{0}}\left(u\otimes v\right)-R_{\epsilon,\tau_{0}}h\left(i\tau_{0}\right)\left(u\otimes v\right)\big\Vert _{\mathcal{HB}_{pE_{1},F_{1}}\left(X^{\wedge}\right)}
\nonumber\\
&&+
C_{2}\left\langle \ln\tau_{0}\right\rangle \left\Vert R_{\epsilon,\tau_{0}}h\left(i\tau_{0}\right)\left(u\otimes v\right)\right\Vert _{\mathcal{HB}_{pE_{1},F_{1}}\left(X^{\wedge}\right)}.
\nonumber
\end{eqnarray*}
As $\lim_{\epsilon\to0}\big\Vert op_{M}^{\frac{1}{2}}\left(h\right)R_{\epsilon,\tau_{0}}\left(u\otimes v\right)-h\left(i\tau_{0}\right)\left(u\otimes v\right)\big\Vert _{\mathcal{HB}_{pE_{1,}F_{1}}\left(X^{\wedge}\right)}=0$,
we see again from Lemma \ref{lem:propertiesRepslonTepsilon} that
\begin{eqnarray*}
\left\Vert u\otimes v\right\Vert _{\mathcal{HB}_{pE_{0,}F_{0}}\left(X^{\wedge}\right)}
&\le& C_{1}\left\langle \ln\tau_{0}\right\rangle \big\Vert \left(R_{\epsilon,\tau_{0}}u\right)\otimes h\left(i\tau_{0}\right)v\big\Vert _{\mathcal{HB}_{pE_{1},F_{1}}\left(X^{\wedge}\right)}\\
&\le& C_{2}\left\langle \ln\tau_{0}\right\rangle ^{2}\left\Vert u\otimes h\left(i\tau_{0}\right)v\right\Vert _{\mathcal{HB}_{pE_{1},F_{1}}\left(X^{\wedge}\right)},
\end{eqnarray*}
where $\left(u\otimes h\left(i\tau_{0}\right)v\right)(t,x):=u(t)\left(h\left(i\tau_{0}\right)v\right)(x)$.
Now, it is easy to conclude that
\begin{eqnarray*}
\lefteqn{\left\Vert v\right\Vert _{H_{p}^{0}\left(X,E_{0}\right)\oplus B_{p}^{0}\left(\partial X,F_{0}\right)}\le\frac{C}{\left\Vert u\right\Vert _{L_{p}\left(\mathbb{R}_{+},\frac{dt}{t}\right)}}\left\Vert u\otimes v\right\Vert _{\mathcal{HB}_{pE_{0,}F_{0}}\left(X^{\wedge}\right)}}
\nonumber\\
&\le&
\tilde{C}\frac{1}{\left\Vert u\right\Vert _{L_{p}\left(\mathbb{R}_{+},\frac{dt}{t}\right)}}\left\langle \ln\tau_{0}\right\rangle ^{2}\left\Vert u\otimes h\left(i\tau_{0}\right)v\right\Vert _{\mathcal{HB}_{pE_{1,}F_{1}}\left(X^{\wedge}\right)}
\\
&\le&
\tilde{C}\left\langle \ln\tau_{0}\right\rangle ^{2}\left\Vert h\left(i\tau_{0}\right)v\right\Vert _{H_{p}^{0}\left(X,E_{1}\right)\oplus B_{p}^{0}\left(\partial X,F_{1}\right)}.
\nonumber
\end{eqnarray*}

\end{proof}
We finish with the following proposition that proves the invertibility
of the conormal symbol.
\begin{prop}
Let $A\in\tilde{C}^{p}\left(\mathbb{D};k\right)$ be a Fredholm operator
in the space 
$$\mathcal{B}\left(\mathcal{HB}_{pE_{0,}F_{0}}\left(X^{\wedge}\right),\mathcal{HB}_{pE_{1,}F_{1}}\left(X^{\wedge}\right)\right).$$
Then the conormal symbol is invertible on $\Gamma_{0}$, and its inverse
is an element of  $\tilde{\mathcal{B}}_{E_{1},F_{1},E_{0},F_{0}}^{p}\left(X,\Gamma_{0}\right)$.
\end{prop}

\begin{proof}
We are going to consider operators given as
\begin{equation}
A=\omega op_{M}^{\frac{1}{2}}\left(h\right)\omega_{0}+\left(1-\omega\right)P\left(1-\omega_{1}\right)+G,\label{eq:FormulaA}
\end{equation}
where $P\in\tilde{\mathcal{B}}_{2E_{0},2E_{1},2F_{0},2F_{1}}^{p}\left(2\mathbb{D}\right)$, 
$G\in\tilde{C}_{\mathcal{O}\,E_{0},F_{0},E_{1},F_{1}}^{p}\left(\mathbb{D},k\right)$, 
and $h\left(t,z\right)=a\left(t,z\right)+\tilde{a}\left(z\right)$
with functions
$a\in C^{\infty}\left(\overline{\mathbb{R}_{+}},\tilde{M}_{\mathcal{O}\,E_{0},F_{0},E_{1},F_{1}}^{p}\left(X\right)\right)$ and $\tilde{a}\in M_{P\,E_{0},F_{0},E_{1},F_{1}}^{-\infty}\left(X\right)$
for some asymptotic type $P$ with 
$\pi_{\mathbb{C}}P\cap\Gamma_{0}=\emptyset$. In particular,
$h_{0}\left(z\right):=h\left(0,z\right)=\sigma_{M}^{0}\left(A\right)\left(z\right)$.

Let us first prove that 
for all $u\in\mathcal{HB}_{pE_{0},F_{0}}\left(X^{\wedge}\right)$
\begin{equation}
\left\Vert u\right\Vert _{\mathcal{HB}_{pE_{0},F_{0}}\left(X^{\wedge}\right)}
\le 
c\big\Vert \mbox{op}_{M}^{\frac{1}{2}}\left(h_{0}\right)\left(u\right)\big\Vert _{\mathcal{HB}_{pE_{1},F_{1}}\left(X^{\wedge}\right)}.\label{eq:Inequality}
\end{equation}
It suffices to show this for $u\in C_{c}^{\infty}\left(\mathbb{R}_{+},C^{\infty}\left(X,E_{0},F_{0}\right)\right)$.
We find operators $B_{1}\in\mathcal{B}\left(\mathcal{HB}_{pE_{1},F_{1}}\left(\mathbb{D}\right),\mathcal{HB}_{pE_{0},F_{0}}\left(\mathbb{D}\right)\right)$
and $K_{1}\in\mathcal{B}\left(\mathcal{HB}_{pE_{0},F_{0}}\left(\mathbb{D}\right),\mathcal{HB}_{pE_{0},F_{0}}\left(\mathbb{D}\right)\right)$,
where $K_{1}$ is compact, such that $B_{1}A-1=K_{1}$. Let us choose
$\sigma$ and $\sigma_{1}$ in $C_{c}^{\infty}\left(\left[0,1\right[\right)$
such that $\sigma\sigma_{1}=\sigma$, $\sigma_{1}\omega_{1}=\sigma_{1}$
and $\sigma_{1}\omega=\sigma_{1}$. Then
\[
K_{1}\sigma=B_{1}A\sigma-\sigma=B_{1}\sigma_{1}A\sigma+B_{1}\left(1-\sigma_{1}\right)A\sigma-\sigma.
\]
As the supports of $\sigma$ and $1-\sigma_{1}$ are disjoint, the
operator $\left(1-\sigma_{1}\right)A\sigma$ is a Green operator and
therefore compact. Hence 
\[
\sigma_{1}B_{1}\sigma_{1}A\sigma-\sigma=\sigma_{1}K_{1}\sigma-\sigma_{1}B_{1}\left(1-\sigma_{1}\right)A\sigma=\sigma_{1}K_{2}\sigma,
\]
where $K_{2}$ is a compact. Using Equation (\ref{eq:FormulaA}) for
$A\sigma$, we conclude that $\sigma=Bop_{M}^{\frac{1}{2}}\left(h\right)\sigma-K$,
where $B=\sigma_{1}B_{1}\sigma_{1}$ and $K=\sigma_{1}\left(K_{2}-B_{1}\sigma_{1}G\right)\sigma$
is compact.

Now let $u\in C_{c}^{\infty}\left(\mathbb{R}_{+},C^{\infty}\left(X,E_{0},F_{0}\right)\right)$.
We know that $T_{\epsilon}\left(u\right)=\sigma T_{\epsilon}\left(u\right)$,
when $\epsilon$ is small. As $\sigma=Bop_{M}^{\frac{1}{2}}\left(h\right)\sigma-K$,
we have that, for $\epsilon$ sufficiently small,
\begin{eqnarray*}
\lefteqn{\left\Vert u\right\Vert _{\mathcal{HB}_{pE_{0},F_{0}}\left(X^{\wedge}\right)}=\left\Vert \sigma T_{\epsilon}\left(u\right)\right\Vert _{\mathcal{HB}_{pE_{0},F_{0}}\left(X^{\wedge}\right)}}
\nonumber\\
&\le&
\left\Vert B\right\Vert _{\mathcal{B}\left(\mathcal{HB}_{pE_{1},F_{1}}\left(X^{\wedge}\right),\mathcal{HB}_{pE_{0},F_{0}}\left(X^{\wedge}\right)\right)}
\big\Vert op_{M}^{\frac{1}{2}}\left(h\right)T_{\epsilon}\left(u\right)-T_{\epsilon}op_{M}^{\frac{1}{2}}\left(h_{0}\right)u\big\Vert _{\mathcal{HB}_{pE_{1},F_{1}}\left(X^{\wedge}\right)}
\nonumber\\
&&+
\left\Vert B\right\Vert _{\mathcal{B}\left(\mathcal{HB}_{pE_{1},F_{1}}\left(X^{\wedge}\right),\mathcal{HB}_{pE_{0},F_{0}}\left(X^{\wedge}\right)\right)}\big\Vert T_{\epsilon}op_{M}^{\frac{1}{2}}\left(h_{0}\right)u\big\Vert _{\mathcal{HB}_{pE_{1},F_{1}}\left(X^{\wedge}\right)}\\
&&+\left\Vert KT_{\epsilon}\left(u\right)\right\Vert _{\mathcal{HB}_{pE_{0},F_{0}}\left(X^{\wedge}\right)}.
\nonumber
\end{eqnarray*}

As $T_{\epsilon}u$  weakly tends to zero  and $K$ is compact,
 $\lim_{\epsilon\to0}\left\Vert KT_{\epsilon}\left(u\right)\right\Vert _{\mathcal{HB}_{pE_{0},F_{0}}\left(X^{\wedge}\right)}=0$.
Using that $T_{\epsilon}$ is an isometry and  item $3.ii)$ of
Proposition \ref{prop:Tepsilonconv}, we conclude that Inequality
(\ref{eq:Inequality}) holds. This result together with Lemma \ref{lem:Lemafundamentalconce}
implies that
\begin{equation}
\left\Vert v\right\Vert _{H_{p}^{0}\left(X,E_{0}\right)\oplus B_{p}^{0}\left(\partial X,F_{0}\right)}\le C\left\langle \ln\left(\tau\right)\right\rangle ^{2}\left\Vert h\left(i\tau\right)v\right\Vert _{H_{p}^{0}\left(X,E_{1}\right)\oplus B_{p}^{0}\left(\partial X,F_{1}\right)},\label{eq:vlehv}
\end{equation}
for some constant $C$ independent of $v$.

As $A\in\tilde{C}_{E_{0},F_{0},E_{1},F_{1}}^{p}\left(\mathbb{D};k\right)$
is Fredholm, so is $A^{*}\in\tilde{C}_{E_{1},F_{1},E_{0},F_{0}}^{q}\left(\mathbb{D};k\right)$. The above argument implies that $\sigma_{M}^{0}\left(A^{*}\right)\left(z\right)=h\left(i\tau\right)^{*}$
also satisfies an estimate as (\ref{eq:vlehv}),
for $q$ instead of $p$, where $\frac{1}{p}+\frac{1}{q}=1$. Hence,
for all $\tau\in\mathbb{R}$, $h\left(i\tau\right)$ is injective,
has closed range and the same is true for its adjoint. Lemma
\ref{lem:injetora-adj-implica-iso} implies that
$h\left(i\tau\right)$ is bijective and
\[
\big\Vert h\left(i\tau\right)^{-1}\big\Vert _{\mathcal{B}\left(H_{p}^{0}\left(X,E_{1}\right)\oplus B_{p}^{0}\left(\partial X,F_{1}\right),H_{p}^{0}\left(X,E_{0}\right)\oplus B_{p}^{0}\left(\partial X,F_{0}\right)\right)}\le\tilde{C}\left\langle \ln\left(\tau\right)\right\rangle ^{2}.
\]

Theorem \ref{thm:Teoremabmcomparametros} implies that $h_{0}^{-1}\in\tilde{\mathcal{B}}_{E_{1},F_{1},E_{0},F_{0}}^{p}\left(X,\Gamma_{0}\right)$.
\end{proof}

\subsection{Spectral invariance of boundary value problems with conical
singularities\label{subsec:SpectralInvarianceofBVConeAlgebra}}

Once we know the equivalence of Fredholm property and ellipticity,
we can establish the spectral invariance.

\begin{thm}\label{spectralinvariance}
Let $A\in\tilde{C}_{E_{0},F_{0},E_{1},F_{1}}^{p}\left(\mathbb{D},k\right)$. Suppose that, for
each $\lambda\in\Lambda$, the operator
\[
A:\mathcal{H}_{p}^{0,\frac{n+1}{2}}\left(\mathbb{D},E_{0}\right)\oplus\mathcal{B}_{p}^{0,\frac{n}{2}}\left(\mathbb{B},F_{0}\right)\to\mathcal{H}_{p}^{0,\frac{n+1}{2}}\left(\mathbb{D},E_{1}\right)\oplus\mathcal{B}_{p}^{0,\frac{n}{2}}\left(\mathbb{B},F_{1}\right).
\]
is invertible. Then $A^{-1}\in\tilde{C}_{E_{1},F_{1},E_{0},F_{0}}^{p}\left(\mathbb{D},k\right)$.
\end{thm}

\begin{proof}
The operator $A$ is invertible, hence it is Fredholm and there
are operators $B\in\tilde{C}_{E_{1},F_{1},E_{0},F_{0}}^{p}\left(\mathbb{D},k\right)$,
$K_{1}\in\tilde{C}_{G\,E_{1},F_{1},E_{1},F_{1}}^{p}\left(\mathbb{D},k\right)$
and $K_{2}\in\tilde{C}_{G\,E_{0},F_{0},E_{0},F_{0}}^{p}\left(\mathbb{D},k\right)$
such that $AB=I+K_{1}$ and $BA=I+K_{2}$. These identities imply
that
\[
A^{-1}=B-K_{2}B+K_{2}A^{-1}K_{1}.
\]
As $B\in\tilde{C}_{E_{1},F_{1},E_{0},F_{0}}^{p}\left(\mathbb{D},k\right)$,
$K_{2}B\in\tilde{C}_{G\,E_{1},F_{1},E_{0},F_{0}}^{p}\left(\mathbb{D},k\right)$
and $K_{2}A^{-1}K_{1}$ belongs to $\tilde{C}_{G\,E_{1},F_{1},E_{0},F_{0}}^{p}\left(\mathbb{D},k\right)$,
we obtain the result. 
\end{proof}

\begin{thm}
\label{thm:Teorema Principal}Let $A\in C_{E_{0},F_{0},E_{1},F_{1}}^{m,d}\left(\mathbb{D},\gamma,\gamma-m,k\right)$, where $m\in \mathbb Z$, $d=\max\left\{ m,0\right\} $. Suppose that
there is an $s\in \mathbb Z$, $s\ge d$ such that
\[
A:\mathcal{H}_{p}^{s,\gamma}\left(\mathbb{D},E_{0}\right)\oplus\mathcal{B}_{p}^{s-\frac{1}{p},\gamma-\frac{1}{2}}\left(\mathbb{B},F_{0}\right)\to\mathcal{H}_{p}^{s-m,\gamma-m}\left(\mathbb{D},E_{1}\right)\oplus\mathcal{B}_{p}^{s-m-\frac{1}{p},\gamma-m-\frac{1}{2}}\left(\mathbb{B},F_{1}\right)
\]
is invertible. Then, $A^{-1}\in C_{E_{1},F_{1},E_{0},F_{0}}^{-m,d'}\left(\mathbb{D},\gamma-m,\gamma,k\right)$
, where $d':=\max\left\{ -m,0\right\} $. 
In particular, for all $s>d-1+\frac{1}{q}$ and $1<q<\infty$ the operator 
$A$ is invertible
in $\mathcal{B}\big(\mathcal{H}_{q}^{s,\gamma}\left(\mathbb{D},E_{0}\right)\oplus\mathcal{B}_{q}^{s-\frac{1}{p},\gamma-\frac{1}{2}}\left(\mathbb{B},F_{0}\right),\mathcal{H}_{q}^{s-m,\gamma-m}\left(\mathbb{D},E_{1}\right)\oplus\mathcal{B}_{q}^{s-m-\frac{1}{q},\gamma-m-\frac{1}{2}}\left(\mathbb{B},F_{1}\right)\big)$.
\end{thm}

\begin{proof}
As in the proof of Corollary \ref{equivalence} we consider the operator 
$\tilde A\in \tilde{C}_{E_{0},F_{0},E_{1},F_{1}}^{p}\left(\mathbb{D},k\right)$ 
defined by \eqref{eq:reduction}. As $A$ is invertible,  so is $\tilde A$. 
We infer from Theorem  \ref{spectralinvariance} that $(\tilde A)^{-1}$ belongs to $\tilde{C}_{E_{1},F_{1},E_{0},F_{0}}^{p}\left(\mathbb{D},k\right)$ and hence 
$A^{-1}\in C_{E_{1},F_{1},E_{0},F_{0}}^{-m,d'}\left(\mathbb{D},\gamma-m,\gamma,k\right)$.
\end{proof}

\end{document}